\documentclass[11pt]{amsart}
\usepackage{latexsym, amsmath, amssymb, amsthm, mathrsfs}
\usepackage[alwaysadjust]{paralist}
\usepackage{caption}
\usepackage{subcaption}
\usepackage[T1]{fontenc}
\usepackage{lmodern}
\usepackage[colorlinks=true,linkcolor=blue,urlcolor=blue, citecolor=blue]%
  {hyperref}
\usepackage[alphabetic]{amsrefs}
\usepackage[all]{xy}
\usepackage[T1]{fontenc}
\usepackage{mathptmx}
\usepackage{microtype}
\usepackage{framed}
\usepackage{tikz}
\usepackage{graphicx}
\usepackage{wasysym}
\usepackage[alwaysadjust]{paralist}

\makeatletter
\providecommand\@enum@widestlabel{7}
\makeatother

\usepackage[centering, includeheadfoot, hmargin=1in, vmargin=1in,
  headheight=30.4pt]{geometry}
\newtheorem{lemma}{Lemma}[subsection]
\newtheorem{lem}[lemma]{Lemma}
\newtheorem{theorem}[lemma]{Theorem}
\newtheorem{thm}[lemma]{Theorem}
\newtheorem{corollary}[lemma]{Corollary}
\newtheorem{cor}[lemma]{Corollary}
\newtheorem{proposition}[lemma]{Proposition}
\newtheorem{prop}[lemma]{Proposition}
\newtheorem{conjecture}[lemma]{Conjecture}
\newtheorem{conj}[lemma]{Conjecture}
\newtheorem{exercise}[lemma]{Exercise}
\newtheorem{exer}[lemma]{Exercise}
\theoremstyle{plain}

\newtheorem*{goal}{Goal}
\theoremstyle{definition}
\newtheorem{definition}[lemma]{Definition}
\newtheorem{defn}[lemma]{Definition}
\newtheorem{remark}[lemma]{Remark}
\newtheorem{rmk}[lemma]{Remark}
\newtheorem{example}[lemma]{Example}
\newtheorem{eg}[lemma]{Example}
\newtheorem{question}[lemma]{Question}
\newtheorem{constr}[lemma]{Construction}

\renewcommand{\theequation}%
{\arabic{section}.\arabic{lemma}.\arabic{equation}}

\newcommand{\CC}{\ensuremath{\mathbb{C}}} 
\newcommand{\NN}{\ensuremath{\mathbb{N}}} 
\newcommand{\PP}{\ensuremath{\mathbb{P}}} 
\newcommand{\QQ}{\ensuremath{\mathbb{Q}}} 
\newcommand{\RR}{\ensuremath{\mathbb{R}}} 
\newcommand{\ZZ}{\ensuremath{\mathbb{Z}}} 
\newcommand{\sF}{\ensuremath{\mathscr{F}}} 
\newcommand{\sI}{\ensuremath{\kern -1pt \mathscr{I}\kern -2pt}} 
\newcommand{\sJ}{\ensuremath{\kern -2pt \mathscr{J}\kern -2pt}} 
\newcommand{\sO}{\ensuremath{\mathscr{O}}} 
\newcommand{\sOX}{\ensuremath{\mathscr{O}^{}_{\! X}}}

\newcommand{\shu}{\ensuremath{\mathcal{U}}}
\newcommand{\shs}{\ensuremath{\mathcal{S}}}

\newcommand{\shl}{\ensuremath{\mathcal{L}}}
\newcommand{\shx}{\ensuremath{\mathcal{X}}}
\newcommand\bbQ{\mathbb Q}

\renewcommand{\geq}{\geqslant}
\renewcommand{\leq}{\leqslant}

\DeclareMathOperator{\codim}{codim}
\DeclareMathOperator{\mult}{mult}
\DeclareMathOperator{\Nef}{Nef}

\DeclareMathOperator{\Sym}{Sym}
\DeclareMathOperator{\Supp}{Supp}
\DeclareMathOperator{\ord}{ord}
\DeclareMathOperator{\Null}{Null}
\DeclareMathOperator{\Neg}{Neg}
\DeclareMathOperator{\vol}{vol}
\DeclareMathOperator{\length}{length}

\DeclareMathOperator{\Bbig}{Big}

\DeclareMathOperator{\pr}{pr}
\DeclareMathOperator{\Area}{Area}
\DeclareMathOperator{\dist}{dist}
\DeclareMathOperator{\intt}{interior}

\DeclareMathOperator{\res}{res}
\DeclareMathOperator{\Bs}{Bs}
\DeclareMathOperator{\ev}{ev}
\DeclareMathOperator{\CDiv}{CDiv}
\DeclareMathOperator{\im}{im}
\DeclareMathOperator{\Zeroes}{Zeroes}
\DeclareMathOperator{\Exc}{Exc}
\DeclareMathOperator{\Eff}{Eff}
\DeclareMathOperator{\Amp}{Amp}
\DeclareMathOperator{\rk}{rk}
\DeclareMathOperator{\gon}{gon}

\newcommand{\equ}{\ensuremath{\,=\,}}
\newcommand{\dgeq}{\ensuremath{\,\geq\,}}
\newcommand{\dleq}{\ensuremath{\, \leq\, }}
\newcommand{\deq}{\ensuremath{\stackrel{\textrm{def}}{=}}}

\newcommand{\dsubseteq}{\ensuremath{\,\subseteq\,}}
\newcommand{\st}[1]{\ensuremath{\left\{ #1 \right\}   }}
\newcommand{\zj}[1]{\ensuremath{\left( #1 \right)}}
\newcommand{\dsupseteq}{\ensuremath{\,\supseteq\,}}
\newcommand{\lra}{\ensuremath{\longrightarrow}}
\newcommand{\og}{\ensuremath{\overline{\Gamma}}}         
\newcommand{\e}{\ensuremath{\epsilon}}

\newcommand{\hh}[3]{\ensuremath{h^{#1}\left(#2,#3\right)}}

\newcommand{\HH}[3]{\ensuremath{H^{#1}\left(#2,#3\right)}}

\newcommand{\ha}[3]{\ensuremath{\widehat{h}^{#1}\left(#2,#3\right)}}

\newcommand{\iss}[1]{\ensuremath{\Delta^{-1}_{#1}}}
\newcommand{\eone}{\ensuremath{\textup{\textbf{e}}_1}}
\newcommand{\rat}{\ensuremath{\dashrightarrow}}
\newcommand{\hxl}{\ensuremath{\HH{0}{X}{\sO_X(L)}}}
\newcommand{\bb}{\ensuremath{\mathfrak{b}}}
\newcommand{\sbl}[1]{\ensuremath{\mathbf{B}(#1)}}
\newcommand{\NO}[2]{\ensuremath{\Delta_{#1}(#2)}}
\newcommand{\ybul}{\ensuremath{Y_{\bullet}}}
\newcommand{\dyl}{\NO{\ybul}{L}}

\newcommand{\dsupset}{\ensuremath{\,\supset\,}}
\newcommand{\vl}[2]{ \ensuremath{\text{vol}_{#1}(#2)}}

\newcommand{\Bplus}{\ensuremath{\textbf{\textup{B}}_{+} }}
\newcommand{\Bminus}{\ensuremath{\textbf{\textup{B}}_{-} }}
\newcommand{\Bstable}{\ensuremath{\textbf{\textup{B}} }} 
\newcommand{\origin}{\ensuremath{\textup{\textbf{0}}}}
\newcommand{\ei}{\ensuremath{\textup{\textbf{e}}_i}}
\newcommand{\en}{\ensuremath{\textup{\textbf{e}}_n}}
\newcommand{\nob}[2]{\ensuremath{\Delta_{#1}(#2)}}
 
\newcommand{\inob}[2]{\ensuremath{\widetilde{\Delta}_{#1}(#2)}}
\newcommand{\liminob}[2]{\ensuremath{\widetilde{\Delta}_{#1}^{\lim}(#2)}}

\newcommand{\sU}{\ensuremath{\mathscr{U}}}
\newcommand{\etwo}{\ensuremath{\textup{\textbf{e}}_2}}
\newcommand{\DCx}[1]{\ensuremath{\nob{(C,x)}{#1}}} 
\newcommand{\oEff}{\ensuremath{\overline{\Eff}}}
\newcommand{\set}[1]{\ensuremath{ \left\{\, #1\, \right\} }}
\newcommand\with{\ \vrule\ }
\newcommand\I{\mathcal{I}}
\newcommand\NR{N^1_\RR}

\newcommand\vecspan[1]{\left\langle#1\right\rangle}
\newcommand\orth{^\perp}
\newcommand\nonneg{^{\geqslant 0}}

\renewcommand\({\left(}
\renewcommand\){\right)}
\newcommand{\shi}{\ensuremath{{\mathcal I}}}
\newcommand{\euler}[2]{\ensuremath{\chi\left(#1,#2\right)}} 
\newcommand{\II}{\ensuremath{\mathcal I}}
\newcommand\E{\mathbb E}
\newcommand\eps{\varepsilon}
\newcommand\calo{{\mathcal O}}

\newcommand\newop[2]{\def#1{\mathop{\rm #2}\nolimits}}
\newop\BigCone{Big}
\newop\Face{Face}

\newenvironment{items}
   {\list{\labelitemi}{
      \parsep=0cm \itemsep=0cm \topsep=0cm \partopsep=0.5\baselineskip
      \def\makelabel##1{\hss\llap{\rm##1}}}}
   {\endlist}

 \newcommand{\m}{\ensuremath{\mathfrak{m}}}

\begin{document}

\title{Geometric aspects of Newton--Okounkov bodies}

\author[A.~K\" uronya]{Alex K\" uronya}
\author[V.~Lozovanu]{Victor Lozovanu}

\address{Alex K\"uronya, Johann-Wolfgang-Goethe Universit\"at Frankfurt, Institut f\"ur Mathematik, Robert-Mayer-Stra\ss e 6-10., D-60325
Frankfurt am Main, Germany}
\address{Budapest University of Technology and Economics, Department of Algebra, Egry J\'ozsef u. 1., H-1111 Budapest, Hungary}
\email{{\tt kuronya@math.uni-frankfurt.de}}

\address{Victor Lozovanu, Leibniz-Universit\"at Hannover, Institut f\"ur Algebraische Geometrie}
\email{\tt victor.lozovanu@gmail.com}

\maketitle

\tableofcontents

\section*{Introduction}

This is a survey article on Newton--Okounkov bodies in projective geometry focusing on the relationship between 
positivity of divisors and Newton--Okounkov bodies. The inspiration for this writing came from the first author's lectures at the workshop 'Asymptotic invariants of linear series' in Cracow, at the same time it has its roots in earlier talks and lecture series of the  authors including the RTG workshop at the University of Illinois at Chicago,  
the CRM workshop 'Positivity and valuations' in Barcelona, at the Universit\'a di Milano Bicocca,  and the first author's habilitation thesis. While it is mostly an expository effort, it sporadically contains  small amounts 
of new material. Such an overview is necessarily biased, and by no means is the manuscript meant as a complete account of the current state of affairs in the area. 

The intended audience is modelled on the heterogenous crowd participating in the Cracow workshop: it ranges from interested graduate students to senior researchers in and around algebraic 
geometry. The style  will necessarily reflect this variation, in particular, the level of detail is by no means uniform.
This writing is not intended as a systematic account of the theory, much rather an overview of the main ideas that tries to identify new lines of research as well. 

 Nevertheless, as prerequisites we expect that the reader is familiar 
with the basics of algebraic geometry, and the technical baggage of  schemes and cohomology roughly on the level of Chapters II and III of \cite{HS}, or at least can make up for it in mathematical maturity. The precise requirements vary from section to section, the overall mathematical difficulty is weakly increasing. Exercises denoted by an asterisk are believed to be more demanding and/or require some extra knowledge.

\medskip
\noindent {\bf Acknowledgements.} 
 The authors are grateful to Thomas Bauer, Lawrence Ein, S\'andor Kov\'acs, Rob Lazarsfeld, Catriona Maclean, John C. Ottem, and Quim Ro\'e for helpful discussions. We would like to use this opportunity to thank the organizers of the RTG Workshop on Newton--Okounkov bodies at the University of Illinois at Chicago (Izzet Coskun, Lawrence Ein, and Kevin Tucker), and the organizers of the miniPAGES Workshop 'Asymptotic invariants attached to linear series' (Jaros\l aw Buczy\'nski, Piotr Pokora, S\l awomir Rams, and Tomasz Szemberg). The lecture notes of Brian Harbourne and Quim Ro\'e written for the same workshop were a  great inspiration. 
 
  The illustrations were done using the Ti$k$Z package.

\section{Positivity and Newton--Okounkov bodies}

Let $X$ be an $n$-dimensional  smooth projective variety over the complex numbers, $L$ a Cartier divisor on $X$. 
Studying  positivity properties --- both global  and local near a point $x\in X$ --- of $L$  has been a central topic of 
algebraic geometry for over 150 years. The main tool we use for this purpose are Newton--Okounkov bodies, a collection of convex bodies  inside $\RR^n$ associated  to the pair $(X,L)$.

This section is devoted to a quick introduction to notions of positivity for line bundles and their interaction with Newton--Okounkov bodies. Although the two could in theory be presented independently of each other, it is our firm belief 
that a description intertwining the two is the most natural way from both perspectives. The material in this section that is not about Newton--Okounkov bodies can be find mostly in Chapters 1 and 2 of \cite{PAGI}.

Very roughly speaking 'positivity' of a line bundle means that it has 'many global sections', whatever this should mean at this point; in any case such properties  have strong implications of geomet\-ric/cohomolo\-gical/numerical nature.

\subsection{Prerequisites}

We will work exclusively over the complex number field. Many of the statements, including the basics of Newton--Okounkov bodies remain true over an algebraically closed field of arbitrary characteristic, nevertheless, we will not follow this path. 

Varieties will be generally assumed to be smooth for simplicity, unless otherwise mentioned. Again, this restriction is not necessary from a purely mathematical point of view, but saves us some trouble. Throughout these notes divisors are always meant to be  Cartier divisors. We will often makes use of Cartier divisors with rational or real coefficients, 
meaning elements of 
\[
\CDiv(X)_\QQ\deq \CDiv(X)\otimes_\ZZ\QQ\  ,\ \text{and}\   \CDiv(X)_\RR \deq \CDiv(X)\otimes_\ZZ\RR\ ,
\]
respectively. Note that intersecting curves with  $\QQ$- or $\RR$-divisors makes sense just the way one would think. 
As customary, we will use divisor and line bundle language interchangeably. 

\begin{definition}
 With notation as above, two Cartier divisors (possibly with $\QQ$- or $\RR$-coefficients $L_1$ and $L_2$ are called \emph{numerically equivalent} 
 (notation: $L_1\equiv L_2$),  if $(L_1\cdot C)=(L_2\cdot C)$  for all curves $C\subseteq X$. The group of numerical equivalence classes of integral Cartier divisors is 
 called the \emph{N\'eron--Severi group}, and it is denoted by  $N^1(X)$. We will write $N^1(X)_\RR\deq N^1(X)\otimes_\ZZ \RR$ and call it the (real) \emph{N\'eron--Severi space}. 
\end{definition}

\begin{remark}\label{rmk:NS}
 Note that there is another natural candidate for the real N\'eron--Severi space, namely the real vector space  $\CDiv(X)_\RR/\equiv$. Luckily, there is a natural isomorphism 
 \[
  \CDiv(X)_\RR/\equiv \ \stackrel{\sim}{\lra}\ N^1(X)\otimes_\ZZ \RR
 \]
therefore no confusion about this notion can arise. 
\end{remark}

\begin{exer}
Prove the statement of Remark~\ref{rmk:NS}. 
\end{exer}

\begin{remark}
We will say that a property of divisors is 'numerical', if it respects the numerical equivalence of divisors. 
\end{remark}

\begin{rmk}
It is a general tendency one has been observing since the study of volumes of line bundles that properties of divisors  that are defined using 'large enough multiples' or in terms of limits over all multiples tend to be numerical. 
\end{rmk}

\subsection{A crash course on  positivity}

We give a quick outline of basic positivity concepts. This serves several purposes, including fixing terminology and notation, and to introduce Newton--Okounkov bodies in their natural habitat from the very beginning. 
For more detailed information we refer the interested reader to 
\cite{PAGI}*{Chapter 1 \& 2]} and \cite{HS}*{Section II.7} and the references therein.

As a first approximation, the positivity of a line bundle means that the associated Kodaira map 
\[
 \phi_L\colon X \rat \PP 
\]
is a morphism.  Here we explain what this means. 

\begin{example}
 Let $X\subseteq\PP^n$ be a projective variety, $f_0,\dots,f_m\in\CC[x_0,\dots,x_n]$   homogeneous polynomials of the same degree with the property that for every point
 $P\in X$ there exists an index $0\leq i\leq m$ such that $f_i(P)\neq 0$. 
 
 Then
 \begin{eqnarray*}
  \phi\colon X & \lra & \PP^m \\ P & \mapsto & (f_0(P)\colon \ldots \colon f_m(P)) 
 \end{eqnarray*}
is a morphism of varieties; this is the baby case of the Kodaira map. 
\end{example}

\begin{remark}
 Homogeneous polynomials of degree $d$ on $\PP$ are elements of $\HH{0}{\PP}{\sO_{\PP}(d)}$, hence we have $f_i\in \HH{0}{X}{\sO_{\PP}(d)|_X}$ in the above example. 
\end{remark}

What happens in general is that given $X$ and $L$ as above, let $s_0,\dots,s_m\in\HH{0}{X}{\sO_X(L)}$, then 
\begin{eqnarray*}
 \phi_{s_0,\dots,s_m} \colon X & \rat & \PP^m \\ P & \mapsto & (s_0(P)\colon \ldots\colon s_m(P))
\end{eqnarray*}
gives rise to a rational map, the so-called \emph{Kodaira map} associated to the global sections $s_0,\dots, s_m$. If $s_0,\dots,s_m$ form a vector space basis 
of $\HH{0}{X}{\sO_X(L)}$, then the associated $\phi\colon X\rat\PP$ is uniquely determined by $L$ (up to a projective automorphism of $\PP$), and we will simply 
write $\phi_L$ for it.

The rational map $\phi_L$ is defined at a point $P\in X$ if not all elements of $\hxl$ vanish at $P$. 

\begin{definition}[Base points and global generation]
 With notation as above, we say that a point $P\in X$ is a \emph{base point of $L$} (or more precisely: it is a base point of the complete linear series $|L|$, 
 if $s(P)=0$ for every $s\in\hxl$. We write $\Bs L$ for the set of base points of $L$, we call it the \emph{base locus of $L$}. 
 
 The Cartier divisor $L$ is called \emph{globally generated} or \emph{base-point free} if $|L|$ has no base points, equivalently,
 if the associated Kodaira map $\phi_L$  is a morphism. We say that $L$ is \emph{semi-ample} if some multiple $mL$ of $L$ is globally generated.  
\end{definition}

\begin{example}
 While being globally generated for a line bundle presupposes that it has global sections, this is not the case for semi-ample ones. As an example, the divisor consisting of a point on an elliptic curve is semi-ample, but the associated line bundle has a base point. 
\end{example}

\begin{exer}
Let $X$ be a non-hyperelliptic curve, $P,Q,R\in X$ distinct points, and $L\deq P+Q-R$. Verify that $H^0(\sO_X(L))=0$, even though $L$ is ample. 
\end{exer}

\begin{eg}
For an interesting example in dimension two, consider  Mumford's fake projective plane \cite{Mumford_fake}. 
This a surface $X$ with ample canonical class such that  $(K_X^2)=9$, $p_g=0$, and $H^0(X,\sOX(K_X))=0$. 
\end{eg}

\begin{exercise}
 If $L$ is semi-ample, then in fact there exists a positive integer $m_0=m_0(L)$ such that $mL$ is globally generated for every $m\geq m_0$. 
\end{exercise}

\begin{remark}
 The base locus $\Bs L$ in fact carries a natural scheme structure; this is seen as follows: let $V\subseteq \hxl$ be a linear subspace, then evaluating 
 sections gives rise to a morphism of vector bundles 
 \[
  \ev_V \colon V\otimes_\CC \sO_X \lra \sO_X(L) \ ,
 \]
 which in turn yields  the morphism  of vector bundles 
 \[
  \ev^*_V \colon V\otimes_\CC \sO_X(-L) \lra \sO_X \ . 
 \]
The image $\bb(V)$ of $\ev^*_V$ is called the \emph{base ideal} of the linear series $V$, one can prove that indeed $\Bs V$ is the closed subscheme associated to 
$V$.   In any case, the scheme structure of base loci will be irrelevant for us for the most part. 
\end{remark}

\begin{definition}
 Let $X$ be a projective variety, $L$ a Cartier divisor on $X$. The \emph{stable base locus} of $L$ is defined as  
 \[
 \sbl{L} \deq \bigcap_{m=1}^{\infty} \Bs (mL)\ .
 \]
\end{definition}

\begin{exer}
 Show  that $\sbl{L}=\Bs (mL)$ for sufficiently large and divisible $m\in\NN$, conclude that $\sbl{L} = \sbl{mL}$ for all $m\geq 1$. 
\end{exer}

\begin{definition}
A divisor $L$ is called \emph{very ample}, if $\phi_L$ is a closed embedding, $L$ is called \emph{ample} if there exists a multiple $mL$ which is very ample.  
\end{definition}

Ampleness turns out to be one of the central concepts of projective geometry. 

\begin{exer}
Use the above definition of ampleness to verify that $mL$ is very ample for all $m\gg 0$ whenever $L$ is ample. 
\end{exer}

\begin{example}
Consider the case $X=\PP^n$, and $L=\sO_X(d)$. Then $L$ is globally generated precisely if $d\geq 0$, and very ample whenever $d>0$. For $d>0$ the closed embedding 
$\phi_L$ is the $d$\textsuperscript{th} Veronese embedding of $\PP^n$. 
\end{example}

If $X$ is a smooth curve, then the Riemann--Roch theorem yields a good answer. Recall the statement: with notation as above, 
\[
 \chi(X,L) \equ \deg_X L + 1-g(X)\ .
\]
With a bit of work one obtains  reasonably precise sufficient conditions. 

\begin{prop}[Global generation and very ampleness on curves]
Let $X$ be a smooth projective curve, $L$ a Cartier divisor on $X$. 
\begin{enumerate}
 \item If $\deg_X L\geq 2g(X)$, then $L$ is base-point free;
 \item if $\deg_X L\geq 2g(X)+1$, then $L$ is very ample. 
\end{enumerate}
\end{prop}

\begin{cor}
A divisor $L$ on a curve is ample if and only if semi-ample if and only if $\deg_X L>0$. 
\end{cor}

\begin{proof}
We will indicate how to verify the first  statement; for a full proof see \cite{HS}*{Section V.1}. 
Let $P\in X$ be an arbitrary point, note that we can identify $\HH{0}{X}{L-P}$ with the subspace of $\hxl$ consisting of global sections vanishing at $P$. 
Consequently, it will suffice to prove that $\hh{0}{X}{L} > \hh{0}{X}{L-P}$. 

Since $\deg L,\deg (L-P) > 2g(X)-2$ by assumption, we have 
\[
 \deg (K_X-L)\ ,\ \deg (K_X-(L-P)) < 0 \ .
\]
Therefore $\HH{0}{X}{\sO_X(K_X-L)}=\HH{0}{X}{\sO_X(K_X-(L-P))} = 0$, and 
\[
 \HH{1}{X}{L} \equ \HH{1}{X}{L-P} \equ 0
\]
by Serre duality. 

As a consequence, 
\[
 \hh{0}{X}{\sO_X(L)} \equ \chi(X,L) \equ \deg L + 1- g(X)\ ,
\]
and 
\[
 \hh{0}{X}{\sO_X(L-P)} \equ \chi(X,L-P) \equ \deg (L-P) + 1- g(X)\ ,
\]
which implies $\hh{0}{X}{\sO_X(L)} > \hh{0}{X}{\sO_X(L-P)}$.
\end{proof}

\begin{remark}
The higher-dimensional generalizations of the above innocent-looking statement gave rise to a large body of interesting research in projective geometry. 
Realizing that the condition 
\[
 \deg_X L \dgeq 2g(X)
\]
(or in other words $\deg_X (L-K_X) > 0$), can be reinterpreted as $L= K_X+ mA$ for an ample divisor $A$ on $X$ where $m\geq \dim X +1$, Fujita  proposed the following 
conjectures: let $X$ be an $n$-dimensional smooth projective variety, $A$ an ample line bundle on $X$. Then 
\begin{enumerate}
 \item $K_X+mA$ is globally generated whenever $m\geq \dim X + 1$,
 \item $K_X+mA$ is very ample provided $m\geq \dim X+2$. 
\end{enumerate}
Fujita's conjectures have been completely proven by Reider in the surface case \cite{Reider}; in higher dimensions the global generation statement has been verified by 
Ein--Lazarsfeld \cite{EL} in dimension three, Kawamata and Helmke in dimension four, and very recently by \cite{YZh} in dimension five. There is a very interesting  connection to Bridgeland stability conditions put forward by Arcara--Bertram \cite{AB}, and explored further by \cite{BBMT}. 

The case of $X=\PP^n$ and $A$ the hyperplane divisor shows that the conjectured bounds are optimal in general. 
\end{remark}

\begin{exer}$\star$
 Prove Fujita's conjectures in the case when $A$ itself is globally generated. 
\end{exer}

\begin{eg}[Elliptic curves]
Let $X$ be an elliptic curve (a smooth projective curve of genus one), then $\deg L\geq 2$ implies that $L$ is globally generated, while $L$ is very ample 
if $\deg L\geq 3$. Note that these bounds are sharp: if $\deg L=2$, then $\phi_L\colon X\to \PP^1$ is not very ample; if $\deg L=1$, then $L$ is not globally 
generated as 
\[
 \hh{0}{X}{L} \equ \hh{0}{X}{L} - \hh{0}{X}{-L} \equ \hh{0}{X}{L} - \hh{1}{X}{L} \equ \chi(X,L) \equ 1
 \]
shows that (up to multiplication by a constant)  $L$ has a single non-zero global section. 
\end{eg}

The way ampleness and semi-ampleness are defined makes it plausible that it is a good idea to look at all multiples of a given divisor $L$. 

\begin{goal} 
 Given $X$ and $L$, study all global sections of all multiples of $L$ at the same time. 
\end{goal}
 
\begin{remark}  
To achieve the goal above we will need an organizing principle; this will come from order of vanishing of global sections along a complete flag $\ybul$ 
of subvarieties of $X$.  This will lead to a collection of convex bodies $\dyl\subseteq\RR^n$ that will provide a universal numerical invariant for sufficiently 
positive line bundles. 
\end{remark}

\begin{constr}[Newton--Okounkov bodies]\label{constr:NO-bodies}
 Let $X$ be a projective variety of dimension $n$, $L$ a Cartier divisor on $X$, $0\neq s\in \hxl$. Fix in addition a complete flag of (irreducible) 
 subvarieties 
 \[
  \ybul\ :\ X \equ Y_0 \dsupset Y_1 \dsupset \ldots\ \dsupset Y_n 
 \]
subject to the conditions that 
\begin{enumerate}
 \item $\codim_{X} Y_i = i$ for all $0\leq i\leq n$; 
 \item $Y_i$ is smooth at the point $Y_n$ for all $0\leq i\leq n$.
\end{enumerate}
Such a flag will be called \emph{admissible}.

Set 
\[
 \nu_1(s) \deq \ord_{Y_1}(s)\ ,\ \text{and}\ \ s_1 \deq \dfrac{s}{f_1^{\nu_1}}|_{Y_1}\ ,
\]
where $f_1$ is a local equation of $Y_1$ in $Y_0$ in a neighbourhood of $Y_n$. Note that since $Y_0=X$ is smooth at $Y_n$, the latter will have at least an open
neighbourhood $U_0$ in which $Y_1$ will be a Cartier divisor. The order of vanishing along $Y_1$ is measured at its generic point, hence as far as determining the value 
of $\nu_1(s)$ is concerned, we can safely restrict our attention to $U_0$.

Next, set 
\[
 \nu_2(s) \deq \ord_{Y_2}(s_1)\ ,\ \text{and}\ s_2\deq \dfrac{s_1}{f_2^{\nu_2(s)}}\big|_{Y_2}\ ,
\]
and so on. This way we associate to $s$ an $n$-tuple 
\[
 s \mapsto \nu_{\ybul}(s) \deq (\nu_1(s),\dots,\nu_n(s)) \in \NN^n\ .
\]
With this notation we define the \emph{Newton--Okounkov body of $L$ with respect to the admissible flag $\ybul$} as 
\[
 \dyl \deq \text{closed convex hull of}\ \bigcup_{m=1}^{\infty} \dfrac{1}{m}\st{\nu_{\ybul}\mid 0\neq s\neq \HH{0}{X}{mL}}\ \subseteq \RR^n\ . 
\]
\end{constr}

\begin{rmk}[Flags centered at a given point]
We say that a flag $\ybul$ is centered at a point $x\in X$ if $Y_n=\st{x}$. We will see in Subsections 2.4 and 2.5, the totality of flags centered at a certain point $x\in X$ determine local positivity of line bundles at $x$ to a large degree. 	
\end{rmk}

\begin{eg}[Complete intersection flags]
Let $X$ be a projective variety of dimension $n$, $L$ a very ample line bundle (ample and base-point free in dimension at least two would suffice). By Bertini's theorem  general elements  $H_1,\dots,H_n\in |L|$ give rise to an admissible flag $\ybul$ via the following construction:
\[
Y_i \deq \begin{cases} H_1\cap\ldots\cap H_i & \text{if $0\leq i\leq n-1$,} \\ \text{a point in $H_1\cap\dots\cap H_n$} 
& \text{ if $i=n$.} \end{cases}
\]
Note that we cannot just set $Y_n$ to be equal to $H_1\cap\ldots\cap H_n$, since this latter consists in more than one point in general.   	
\end{eg}

\begin{exer}
Provide the details in the previous example. 
\end{exer}

\begin{remark}[History]
Newton--Okounkov bodies in the current context and generality were first introduced by Kaveh--Khovanskii \cite{KKh} and Lazarsfeld--Musta\c t\u a \cite{LM} at about 
the same time. Both works were building on earlier ideas of Okounkov \cites{Ok1,Ok2} who introduced these bodies in a representation-theoretic context. 
\end{remark}

\begin{remark}[More general point of view]
Construction~\ref{constr:NO-bodies} consists of two parts: first one build a rank $n$ valuation of the function field $\CC(X)$ using the flag $\ybul$ and then 
uses the arising normalized valuation vectors to construct the Newton--Okounkov body associated to the Cartier divisor $L$. 

In fact the above process goes through verbatim if we start directly with a rank $n$ valuation $v$ of $\CC(X)$, one can simply define
\[
 \Delta_v(L) \deq \text{closed convex hull of } \bigcup_{m=1}^{\infty} \dfrac{1}{m}\st{v(s)\mid 0\neq s\in \HH{0}{X}{mL}}\ \subseteq \RR^n\ .
\]
The added flexibility lets one consider flags on birational models of $X$, or even more generally, to associate convex bodies to graded subalgebras of $\CC(X)$.  This is the point of view taken in \cites{KKh,Roe,CFKLRS} for instance. 

On the other hand, at the expense of changing the model we work on, we can be content with working with valuations arising from flags. According to \cite{CFKLRS}, Proposition 2.8 and Theorem 2.9, 
for every rank $n$ valuation $v$ of the function field $\CC(X)$ of a smooth variety $X$ of dimension $n$ there exists a proper birational morphism 
$\pi\colon \widetilde{X}\to X$ and a flag $\ybul$ on $\widetilde{X}$ such that $\nu_{\ybul}$ and $v$ coincide. 
\end{remark}

The territory in which Newton--Okounkov bodies are most useful is the cone of divisors with 'sufficiently many sections'. 

\begin{definition}[Big divisors]
Let $X$ be an $n$-dimensional  projective variety, $L$ a Cartier divisor on $X$. We say that $L$ is \emph{big} if 
\[
\vl{X}{L} \deq  \limsup_{m\to\infty} \dfrac{ \hh{0}{X}{mL}}{m^n/n!} \ > \ 0\ .
\]
The quantity $\vl{X}{L}$ is called the \emph{volume of $L$}. 
\end{definition}

\begin{thm}[Characterizations  of big divisors]
Let $X$ be a projective variety, $L$ a Cartier divisor on $X$. Then the following are equivalent. 
\begin{enumerate}
 \item $L$ is big.
 \item There exists $m\gg 0$ with the property that the rational map $\phi_{mL} \colon X \rat \PP$
is birational onto its image.
\item  There exists a natural number  $m_0=m_0(L)$ such that $\phi_{mL}\colon X\rat\PP$ is birational onto its image for all $m\geq m_0$. 
\item (Kodaira's lemma) There exists a natural number $m$ such that $mD\simeq A+E$ where $A$ is an ample, $E$ is an effective integral Cartier divisor. 
\end{enumerate}
\end{thm}
\begin{proof}
All results are proven in Section 2.2 of \cite{PAGI}. 
\end{proof}

\begin{remark}
 Note that it follows from the bigness of $L$ the the sequence $n! \hh{0}{X}{mL}/m^n$ in fact converges to $\vl{X}{L}$. 
\end{remark}

\begin{thm}[Properties of the volume]\label{thm:volume}
Let $X$ be a projective variety of dimension $n$, $L$ a Cartier divisor on $X$. Then the following hold.
 \begin{enumerate}
  \item $\vl{X}{aL}=a^n\cdot \vl{X}{L}$ for all natural numbers $a\in\NN$.
  \item If $L_1\equiv L_2$, then $\vl{X}{L_1}=\vl{X}{L_2}$.
  \item The volume induces a continuous log-concave function $\vol_X\colon N^1(X)_\RR \to \RR_{\geq 0}$. 
 \end{enumerate}
\end{thm}
\begin{proof}
The original proofs are found in \cite{PAGI}*{Section 2.3}. For big divisor classes most of the listed properties follow immediately from the formal properties of Newton--Okounkov bodies. 
\end{proof}

\begin{remark}[Volume and bigness for rational and real divisors]
	As numerically trivial divisors have equal volumes, it makes sense define the volume of an element of $N^1(X)$ as the common value of $\vl{X}{D}$ of any representative. For an element $\alpha\in N^1(X)_\QQ$, let $m\in\NN$ be such that $m\alpha$ is 
	integral. Set 
	\[
	\vl{X}{\alpha} \deq \frac{1}{m^n}\cdot \vl{X}{m\alpha}\ .
	\]
	A quick check will convince that the homogeneity of the volume implies that this way we obtain a well-defined function on $N^1(X)_\RR$. 
	
	A more elaborate argument will lead to the observation  that $\vl{X}{\alpha}$ is a locally uniformly continuous function on $N^1(X)_\QQ$, hence can be extended uniquely to $N^1(X)_\RR$ in a continous fashion.  
\end{remark}

	Once this is in place, we can define bigness for elements $\alpha\in N^1(X)_\RR$. 
	
\begin{defn}
With notation as above, an $\RR$-divisor  class $\alpha\in N^1(X)_\RR$ is called \emph{big}, if $\vl{X}{\alpha}>0$
\end{defn}	
 
\begin{exer}
Prove  the following extension of Kodaira's Lemma to big divisor classes: a divisor class $\delta\in N^1(X)_\RR$ is big precisely if it can be written as $\delta = \alpha + \epsilon$, where $\alpha$ is ample, and $\epsilon$ is effective (note that we have not yet defined what ample or effective mean for an $\RR$-divisor class).  
\end{exer} 

\begin{remark}[Volume of ample divisors]
The asymptotic version of the Riemann--Roch theorem says that 
\[
\chi(X,\sOX(mL)) \equ \frac{(L^n)}{n!}\cdot m^n + O(m^{n-1})
\] 
for a divisor $L$ on $X$.  If $L$ is ample, then  Serre vanishing implies 
\[
H^1(X,\sOX(mL)) \equ 0\ \ \ \text{for all $m\gg 0$,}
\]
which yields  $\vl{X}{L}=(L^n)$.  
\end{remark}

\begin{exer}
 Compute the volume function for the blow-ups of $\PP^2$ in one and two points. 
\end{exer}

\begin{remark}[Higher asymptotic cohomology]
Analogously to the volume of a divisor, one can measure the asymptotic rate of growth of higher cohomology groups of divisors by considering 
\[
 \widehat{h}^i(X,L) \deq \limsup_{m\to \infty} \dfrac{ \hh{i}{X}{\sO_X(mL)}}{m^n/n!}\ .
\]
The asymptotic cohomology functions defined this way enjoy many of the formal properties of the volume; in particular, they extend to continuous functions $N^1(X)_\RR\to\RR$. Their theory is developed 
in \cite{ACF}, some applications are given in \cites{dFKL,Dem1,Dem2}.

As opposed to the situation for volumes of divisors, at the time of writing it is an open question whether the upper limits $\widehat{h}^i(X,L)$ are in fact limits or not. Equally missing is a convex geometric 
interpretation of $\widehat{h}^i(X,L)$, which would likely yield  a positive answer to the previous question. 
\end{remark}

\begin{thm}[Formal properties of Newton--Okounkov bodies]\label{thm:NO formal}
 Let $X$ be a projective variety of dimension $n$, $L$ a line bundle on $X$. 
 \begin{enumerate}
  \item The subset $\dyl\subseteq \RR^n$ is a compact, convex, and has non-empty interior if $L$ is big.
  \item  $n!\cdot\vl{\RR^n}{\dyl} \equ \vl{X}{L}$, in particular, the volume of the Newton--Okounkov body $\dyl$ is independent of the 
  choice of the flag $\ybul$. 
  \item $\NO{\ybul}{aL} = a\cdot \dyl \subseteq \RR^n$ for all $a\in\NN$. 
  \item $\NO{\ybul}{\ }$ is invariant with respect to numerical equivalence.
 \end{enumerate}
\end{thm}

\begin{proof}
 We refer the reader to \cite{LM}. 
\end{proof}

\begin{remark}
 Part $(3)$ implies that $\NO{\ybul}{L}$ makes sense when $L$ is a $\QQ$-divisor, while $(4)$ yields that the notion descends to numerical equivalence classes. These claims put together let us define 
 $\NO{\ybul}{\alpha}$ for an arbitrary $\QQ$-divisor class $\alpha$. 
 
 As it turns out, one can go even further.  Lazarsfeld--Musta\c t\u a \cite{LM} use a continuity argument to extend $\NO{\ybul}{\alpha}$ to arbitrary big classes in $N^1(X)_\RR$. It is important to point out 
 that their proof does not imply continuity of Newton--Okounkov bodies as one converges to the boundary of $\Bbig(X)$,  not in the least since this latter is false (see \cite{CHPW1}, and Remark~\ref{rem:nonbig}
 below about Newton--Okounkov bodies of non-big divisors).
\end{remark}

\begin{remark}[Newton--Okounkov bodies for non-big divisors]\label{rem:nonbig}
 It is immediate from their definition that Newton--Okounkov bodies make sense for arbitrary divisors, nevertheless, most of the good properties are valid only for big ones. 
 
 In \cites{CHPW1,CHPW2} the authors study Newton--Okounkov bodies for not necessary big divisor classes. An important point to make is that for non-big classes one has two a priori different options: one can make use of the original definition of 
 Newton--Okounkov bodies (this is called a 'valuative Newton--Okounkov body' by the authors), or one can try to force 
 continuity and use limits of Newton--Okounkov bodies of big classes (so-called 'limiting Newton--Okounkov bodies'). 
 
 One of the lessons of \cite{CHPW1} is that even in dimension two, these two notions might and will differ, in particular, 
 as mentioned above, forming Newton--Okounkov bodies is not a continuous operation as one converges towards non-big classes. 
\end{remark}

\begin{rmk}
 One of the first interesting applications of Newton--Okounkov bodies was that they give a simple and formal explanation
 for the continuity of the volume function for big divisors, its log-concavity \cite{LM}*{Corollary 4.12}, 
 and the existence of Fujita approximations \cite{LM}*{Theorem 3.5}. 
\end{rmk}

An interesting turn of events is that the collection of all Newton--Okounkov bodies forms a 'universal numerical invariant' in the following sense. 

\begin{thm}[Jow \cite{Jow}]\label{thm:Jow}
 Let $X$ be a projective variety, $L$ and $L'$ big (integral) Cartier divisors on $X$. If $\NO{\ybul}{L}=\NO{\ybul}{L'}$ for every admissible flag $\ybul$, then $L\equiv L'$. 
\end{thm}

\begin{rmk}
Along the lines of \cite{Jow}, Ro\'e \cite{Roe} introduced the concept of local numerical equivalence on surfaces, and studied its connection with Newton--Okounkov bodies.
\end{rmk}

As a consequence, morally speaking every property or invariant of divisors that is invariant with respect to numerical equivalence should be observable from the function
\[
 \Delta(L) \colon \ybul \mapsto \NO{\ybul}{L}\ .
\]
This motivates the question that will guide our investigations in these notes. 

\begin{question}
How can we read off positivity properties of line bundles from the collection of its Newton--Okounkov bodies? In particular, how do we decide if the given line bundle is ample? 
\end{question}

\begin{eg}[Case of curves, \cite{LM}*{Example 1.13}]
Let $X$ be a smooth projective curve of genus $g$. An admissible flag $\ybul$ then consists of $X$ itself and the choice of a point $P\in X$. For a $\QQ$-effective divisor $L$ on $X$ 
and a natural number $m$ we have 
\begin{eqnarray*}
 \im \nu_{\ybul}  & \colon & \HH{0}{X}{\sO_X(mL)} \lra \NN \\ & = & \text{ the vanishing sequence of $mL$ at $P$} \\
 & = & \text{ the complement of the Weierstra\ss\ gaps sequence of $mL$ at $P$}\ .
\end{eqnarray*}
From the Riemann--Roch theorem we obtain that for every $0\leq i\leq \deg(mL) - 2g-1$ there exists a global section $s\in\HH{0}{X}{\sO_X(mL)}$ such that $\ord_P(s)=i$. 
This leads to 
\[
 \dyl \equ [0,\deg L]\ .
\]
\end{eg}

\begin{defn}[Valuative point]
Let $X$ be a projective variety, $\ybul$ an admissible flag, $L$ a Cartier divisor on $X$. We say that a point $v\in \dyl\cap\QQ^n\subseteq\RR^n$ is \emph{valuative}, if there is a 
section $s\in\HH{0}{X}{\sO_X(mL)}$ for some $m\geq 1$ such that $v=\tfrac{1}{m}\cdot\nu_{\ybul(s)}$. 
\end{defn}

\begin{exer}
 Show that all rational points in the interior of $\dyl$ are valuative. 
\end{exer}

\begin{exer}
 Let $X$ be a smooth curve of genus $g$, $L$ a divisor with $\deg L>0$,  $\ybul \colon X \supset \st{P}$  an admissible flag on $X$. Give a necessary and sufficient criterion for 
 $\deg L \in \dyl$ to be valuative. 
\end{exer}

\begin{eg}[Projective space]
Let $X=\PP^n$ with homogeneous coordinates $T_0,\dots,T_n$. Set 
\[
 Y_i \deq V(T_0,\dots, T_{i-1}) \dsubseteq \PP^n\ ,
\]
and consider the resulting admissible flag $\ybul$ consisting of coordinate subspaces. 

Then $\nu_{\ybul}$ is the lexicographic ordering given on monomials by 
\[
 \nu_{\ybul}(T_0^{a_0}\dots T_n^{a_n}) \deq (a_0,\dots, a_{n-1})\ ,
\]
hence 
\[
 \NO{\ybul}{\sO_{\PP}(1)} \equ \st{x\in \RR^n\mid x\geq 0\, ,\, \sum_{i=1}^{n}x_i\leq 1}
\]
is a simplex. 
\end{eg}

\begin{remark}[Toric varieties]
As pointed out in \cite{LM}*{Section 6.1}, the connection to toric geometry observable in the example of projective spaces carrier over to arbitrary smooth toric varieties. More precisely, if $X$ is a toric variety, 
$\ybul$ an admissible flag consisting of $T$-invariant subvarieties, $L$ a $T$-invariant integral Cartier divisor, then $\NO{\ybul}{L}$ is a shift of the moment polytope associated to $(X,L)$. 

In the light of this example one could think of the theory of Newton--Okounkov bodies as trying to do toric geometry without a torus action. The analogy can in fact be made more concrete. The torus action on a toric variety gives rise to a compatible set of gradings on all space $H^0(X,\sOX(L))$. In the absence of such a group action, the theory of toric varieties replaces 
this grading by an infinite collection of compatible filtrations (one filtration for every rank $n$ valuation of the function field) on the spaces of global sections. 
\end{remark}

The following  a slightly  different definition of Newton--Okounkov bodies. It has  already appeared in print  in \cite{KLM1}, and although it is an 
 immediate consequence of \cite{LM}, a complete proof was first given in \cite{Bou1}*{Proposition 4.1}.
 
 \begin{proposition}[Equivalent definition of Newton-Okounkov bodies]\label{prop:definition}
 	Let $\xi\in\textup{N}^1(X)_{\RR}$ be a big $\RR$-class and $Y_{\bullet}$ be an admissible flag on $X$. Then
 	\[
 	\Delta_{Y_{\bullet}}(\xi) \ = \ \textup{closed convex hull of\ }\{ \nu_{Y_{\bullet}}(D) \ | \ D\in \textup{Div}_{\geq 0}(X)_{\RR}\, ,\, D\equiv \xi\},
 	\]
 	where the valuation $\nu_{Y_{\bullet}}(D)$, for an effective $\RR$-divisor $D$, is constructed inductively as in the case of integral divisors.
 \end{proposition}
 
 \begin{remark}
 	Just as in the case of the original definition of Newton--Okounkov bodies, it becomes a posteriori clear that valuation vectors $\nu_{\ybul}(D)$ form a dense subset 
 	of 
 	\[
 	\textup{closed convex hull of }\{ \nu_{Y_{\bullet}}(D) \ | \ D\in \textup{Div}_{\geq 0}(X)_{\RR}, D\equiv \xi\}\ ,
 	\]
 	hence it would  suffice to take closure in Proposition~\ref{prop:definition}. 
 \end{remark}
 
Next, we give a convex-continuity statement for Newton--Okounkov bodies. 

\begin{lemma}\label{lem:nested}
	Let $D$ be a big $\RR$-divisor,  $Y_\bullet$  an admissible flag on $X$.  Then the following hold.
	\begin{enumerate}
		\item For any real number $\epsilon>0$ and any ample $\RR$-divisor $A$ on $X$, we have $\Delta_{Y_\bullet}(D) \subseteq \Delta_{Y_\bullet}(D+\epsilon A)$.
		\item If $\alpha$ is an arbitrary nef  $\RR$-divisor class, then $\nob{\ybul}{D}\subseteq \nob{\ybul}{D+\alpha}$.
		\item If $\alpha_m$ is any sequence of nef $\RR$-divisor classes with the property that $\alpha_m-\alpha_{m+1}$ is nef and 
		$\|\alpha_m\|\to 0$ as $m\to\infty$  with respect to some norm on $N^1(X)_\RR$, then 
		\[
		\Delta_{\ybul}(D) \equ \bigcap_m \nob{\ybul}{D+\alpha_m}\ .
		\]
	\end{enumerate}
\end{lemma}

\begin{proof}
	For the first claim, since $A$ is an ample $\RR$-divisor, one can find an effective $\RR$-divisor $M\sim_{\RR}A$ with  $Y_n\notin\Supp(M)$. Then for any arbitrary 
	effective divisor $F\sim_{\RR}D$ one has  $F+\epsilon M \equiv_{\RR} D+\epsilon A$ and $\nu_{\ybul}(F+M)=\nu_{\ybul}(F)$. Therefore
	\[
	\st{\nu_{Y_{\bullet}}(\xi) \mid \xi\in \textup{Div}_{\geq 0}(X)_{\RR}, D\equiv \xi} \subseteq %
	\st{ \nu_{Y_{\bullet}}(\xi) \mid \xi \in \textup{Div}_{\geq 0}(X)_{\RR}, D+\epsilon A\equiv \xi}
	\]
	and we are done by Proposition~\ref{prop:definition}.
	
	For $(ii)$, note first that whenever $\alpha$ is ample, we can write $\alpha=\sum_{i=1}^{r}\epsilon_iA_i$ for suitable real numbers $\epsilon_i>0$ and 
	ample integral classes $A_i$,  therefore an interated application of $(i)$ gives the claim. The general case then follows by continuity and 
	and  approximating a nef $\RR$-divisor class by a sequence of ample ones. 
	
	The equality in $(iii)$  is a  consequence of $(ii)$ and the continuity  
	of Newton--Okounkov bodies.
\end{proof}

Back to ampleness:  part of the usefulness of the notion stems from the existence of characterizations in terms of geometric/cohomological/numerical terms. Most of the results below (and much more) are proven in \cite{PAGI}*{Sections 1.2 -- 1.4}.

\begin{thm}[Cartan--Grothendieck--Serre]
Let $X$ be a complete variety, $L$ a Cartier divisor on $X$. Then the following are equivalent. 
\begin{enumerate}
 \item $L$ is ample.
 \item For every coherent sheaf $\sF$ on $X$ there exists a positive integer $m_0=m_0(L,\sF)$ such that $\sF\otimes\sO_X(mL)$ is globally generated for all $m\geq m_0$. 
 \item For every coherent sheaf $\sF$ on $X$ there exists a positive integer $m_1=m_1(L,\sF)$ such that $$\HH{i}{X}{\sF\otimes\sO_X(mL)}=0$$ for all $i\geq 1$ and $m\geq m_1$. 
\end{enumerate}
\end{thm}

\begin{thm}[Nakai--Moishezon--Kleiman]\label{thm:Nakai Moishezon}
Let $X$ be a projective variety, $L$ a Cartier divisor on $X$. Then $L$ is ample if and only if $(L^{\dim V}\cdot V)>0$ for every positive dimensional  irreducible subvariety $V\subseteq X$ (including $X$ itself).  
\end{thm}

\begin{remark}
	If $X$ is a curve, then $L$ is ample precisely if $\deg L>0$, which happens exactly if $L$ is $\QQ$-effective. On a surface $X$ a divisor $L$ is ample if and only if $(L\cdot C)>0$ for every curve $C\subseteq X$ and 
	$(L^2)>0$. Note that the latter condition is not superfluous. 
\end{remark}

\begin{remark}[Extension of ampleness]
	The results above let one extend the notion of ampleness to $\QQ$- or $\RR$-divisors in a natural way; in addition we can make sense of ampleness for  numerical equivalence classes: a class $\alpha\in N^1(X)$ is ample, if it is represented by an ample divisor, similarly for $\QQ$- and $\RR$-divisor classes.  
\end{remark}

\begin{remark}
	It follows that ampleness is a numerical property, that is, if $L_1\equiv L_2$, then $L_1$ is ample if and only if $L_2$ is ample. An interesting consequence is that being an adjoint 
	divisor is also a numerical property. Recall that a divisor is called adjoint if it has the form $K_X+A$ with $A$ ample. If $L\equiv K_X+A$, then $L-K_X\equiv A$, that is, $L-K_X$ is numerically equivalent to an ample divisor, hence it is itself ample. 
\end{remark}

As a consequence, we can see that ampleness becomes a notion for classes $\alpha\in N^1(X)_\RR$.

\begin{exer}
Check  that ampleness is a convex property: for $\alpha_1,\alpha_2\in N^1(X)_\RR$ ample, $t_1,t_2\geq 0$ with $t_1+t_2=1$ we have that  $t_1\alpha_1+t_2\alpha_2$ is ample again.
\end{exer}

\begin{defn}[Ample cone]
Let $X$ be a projective variety.  The ample cone $\Amp(X)\subseteq N^1(X)_\RR$ of $X$ is the convex cone consisting of all ample $\RR$-divisor classes. 
\end{defn}

\begin{exer}
Let $X$ be a projective variety. Show that the ample cone equals the convex cone spanned by all ample classes $\alpha\in N^1(X)$. Check that the ample cone is strictly convex (it does not contain a line), and it has non-empty interior. 
\end{exer}

It is well-known that in toric geometry ampleness can be read off the polytope associated to a line bundle. An important fact about the relationship between positivity and Newton--Okounkov bodies is that an analogous result holds. 

\begin{thm}[Ampleness via Newton--Okounkov bodies]\label{thm:NO ampleness}
 Let $X$ be a projective variety, $L$ a big divisor (potentially with rational or real coefficients). Then the following are equivalent.
 \begin{enumerate}
  \item $L$ is ample.
  \item For every admissible flag $\ybul$ there exists a positive real number $\lambda>0$ such that $\Delta_\lambda\subseteq\dyl$. 
  \item For every point $x\in X$ there exists an admissible flag $\ybul$ with $Y_n=\st{x}$ and $Y_1$ ample,  and a positive real number $\lambda>0$ such that $\Delta_\lambda\subseteq \dyl$.
 \end{enumerate}
\end{thm}

\begin{proof}
We prove the result in Subsection~2.5. 
\end{proof}

\begin{eg}
 Let $\pi\colon X\to\PP^2$ be the blowing-up of $\PP^2$ at a point $P\in\PP^2$ with exceptional divisor $E$; we write $H$ for the pullback of the hyperplane class. The N\'eron--Severi space $N^1(X)_\RR$ 
 has $\st{H,E}$ as a basis, and a divisor $aH+bE$  is ample if and only if $a>0$ and $0<-b<a$. 
 
 It is an unfortunate fact that $H$ is not ample, in fact, more generally, the pullback of an ample divisor under a proper birational morphism (that is not an isomorphism) is never ample. On the other hand, 
 one can see that $(H\cdot C)\geq 0$ for all curves $C$ on $X$. 
\end{eg}

\begin{remark}[Ampleness in families]
Ampleness is open in families in a very general sense. More precisely, let $f\colon \shx\to T$ be 	a proper morphism of varieties (or schemes), $\shl$ a line bundle on $\shx$. Write $\shx_t$ for $f^{-1}(t)$ ($t\in T$) and $\shl_t\deq \shl|_{\shx_t}$. If $\shl_{t_0}$ for some $t_0\in T$, then there exist an open neighbourhood $t_0\in U\subseteq T$ such that $\shl_t$ is ample for all $t\in U$. 
\end{remark}

\begin{defn}[Nefness]
 Let $X$ be a complete variety. An $\RR$-divisor $L$ on $X$ is \emph{nef} if $(L\cdot C)\geq 0$ for every curve $C\subseteq X$. Nef $\RR$-divisor classes are defined analogously. 
\end{defn}

\begin{thm}[Kleiman]\label{thm:Kleiman}
Let $X$ be a complete variety of dimension $n$, $L$ a nef $\RR$-divisor on $X$. Then 
\[
(L^m\cdot V) \dgeq 0 
\]
for every $m$-dimensional irreducible subvariety $V\subseteq X$. 
\end{thm}

\begin{exer}
Verify that ample, globally generated, or semi-ample line bundles are nef. 
\end{exer}

\begin{remark}
The nef cone $\Nef(X)\subseteq N^1(X)_\RR$ of a complete variety $X$ is defined analogously to the ample cone: it is simply the cone consisting of all nef $\RR$-divisor classes (or, equivalently, it is the cone spanned in $N^1(X)_\RR$  by all integral nef classes). 
\end{remark}

\begin{exer}
For a nef $\RR$-divisor $L$ and an ample $\RR$-divisor $A$  on a projective variety $X$, verify that
$L+\epsilon A$ is ample for all $\epsilon>0$. Conversely, if $L$ and $A$ are $\RR$-divisors on $X$, $A$ is ample, 
and $L+\epsilon A$ is ample for all $\epsilon>0$, then $L$ is nef. 
 	
Use these observations to show that the nef cone is the (topological) closure of the ample cone inside $N^1(X)_\RR$, and 
$\Amp(X)$ is the interior of $\Nef(X)$. 
\end{exer}

\begin{rmk}[Birational nature of big and nef]
Note that the pullback of an ample divisor under a proper birational morphism  is nef, at the same time the pullback of a nef divisor under an arbitrary proper morphism remains nef. 	

By the Leray spectral sequence pulling back by a proper birational morphism preserves volumes, hence the pullback of an ample divisor is big and nef; this latter property is however remains unchanged after proper birational pullbacks, hence this is the 'right' notion of positivity in birational geometry. 

To support this point note that many of the main building blocks of the Mori program are proven for big and nef divisors  and not only for ample ones.  
\end{rmk}

\begin{rmk}[Nefness in families]
As opposed to ampleness, the question whether nefness is open in families, is a lot more subtle (cf. \cite{PAGI}*{Proposition 1.4.14}). Lesieute in \cite{Les} gave an example where nefness is not open for a family of $\RR$-divisors. Nevertheless, the more geometric case of families of big integral divisors is still open to the best of our knowledge. 
\end{rmk} 

\begin{exer} $\star$
Let $\shl$ be a family of big line bundles over a flat family of smooth projective surfaces. Show that nefness is open in
$\shl$ (cf. \cite{Moriwaki}).	
\end{exer}

\begin{remark}[Positivity of higher-codimension cycles]
Fulger and Lehmann \cite{FL_cones} have recently developed an extremely interesting theory of positivity and Zariski decomposition for (co)cycles of higher codimension. One of the main points of interest is that nefness falls apart into four different notions.  
\end{remark}

\begin{exer}
Prove that a Cartier divisor is nef if and only if it is universally pseudo-effective. More specifically, let $L$ be a line bundle on a projective variety $X$. Show that $L$ is nef if and only if for every morphism of projective varieties $f\colon Y\to X$ with $Y$ smooth, $f^*L$ is pseudo-effective. 
\end{exer}

Just as in the case of ampleness, nef divisors also have a description in terms of Newton--Okounkov bodies. The proof is referred to Section~2.4. 

\begin{thm}[\cite{KL15a}]
	For a big $\RR$-divisor $L$ on a smooth projective variety $X$,  the following are equivalent.
	\begin{enumerate}
		\item $D$ is nef.
		\item For every point $x\in X$  there exists an admissible flag $\ybul$ on $X$ centered at $x$ such that $\origin\in \Delta_{\ybul}(D)\subseteq\RR^n$.
		\item One  has $\origin\in \Delta_{\ybul}(D)$ for every admissible flag $\ybul$ on $X$.  
	\end{enumerate}
\end{thm}

Before moving on, here is a quick summary of notions of positivity for divisors.

\begin{itemize}
	\item Ample, nef, big, and pseudo-effective are numerical properties, and they make sense for elements of $N^1(X)_\RR$. In comparison, the properties very ample, globally generated, semi-ample, effective, and $\QQ$-effective are not numerical\footnote{This means that there exist divisors $L\equiv L'$ where one is for instance globally generated while the other one is not.}.  
	\item Ample \& big are open  (inside the real vector space $N^1(X)_\RR$), while nef \& pseudo-effective are closed.
	\item Ample implies big, nef implies pseudo-effective, but in general not the other way around. 
	\item Globally generated/semi-ample imply nef, and effective/$\QQ$-effective imply pseudo-effective. 
\end{itemize}

\begin{rmk}
It is a very important and equally difficult question in higher-dimensional geometry to decide when a nef divisor is globally generated/semi-ample or a pseudo-effective divisor is $\QQ$-effective/effective. 

The general yoga is that problems of 'numerical' nature are more approachable than 'geometric' ones. 
\end{rmk}

\begin{exer}
Upon looking at the list of positivity properties, one is clearly missing, the variant of bigness that corresponds to very ample divisors. Come up with a definition of 'birationally very ample' or 'very big'. 
\end{exer}

\begin{rmk}
	The following simplistic argument can provide an initial justification for the shape of the results we expect. Let $X$ be a smooth projective curve, $\ybul$ an admissible flag on $X$, which  simply means the specification of $Y_1=\{P\}$, and $L$ a line bundle on $X$. Then 
	\[
	\text{$L$ is nef } \ \ \ \Leftrightarrow\ \ \ \deg_XL\dgeq 0 \ \ \ \Leftrightarrow\ \ \  0\in \nob{\ybul}{L} \equ [0,\deg_X L]\ ,
	\] 
	and 
	\[
	\text{$L$ is ample } \ \ \ \Leftrightarrow\ \ \ \deg_X L \,>\, 0 \ \ \ \Leftrightarrow\ \ \  [0,\epsilon)\in \nob{\ybul}{L} \equ [0,\deg_X L]\  \text{for some $\epsilon>0$.}
	\] 
\end{rmk}

Finally a quick note about the question when Newton--Okounkov bodies of line bundles can be rational polytopes. They always are in dimension one, and are very close to being so in dimension
two. In fact, as shown in Theorem~\ref{thm:NO-polygons}, Newton--Okounkov bodies on surfaces are always polygons. 

Here the dependence on the flag starts becoming an issue. As proven in Proposition~\ref{prop:surface} below, if $L$ is a big line bundle on a surface, then one will always be 
able to find a flag with respect to which $\dyl$ is a rational polygone, but there is an abundant cash of examples that show that this is not always the case. Nonetheless, there is no
known global obstruction to $\dyl$ being a rational polygone in dimension two, the volume of a line bundle is for instance always a rational number.

The situation changes drastically in dimension three and above. It is known that $\vl{X}{L}$ is no longer necessarily rational, and by Theorem~\ref{thm:volume} $(2)$, this implies that 
$\dyl$ cannot be a rational polytope (this is observed in \cite{LM}). Going down this road, it is possible to construct examples of Mori dream spaces $X$  with ample line bundles $L$ (see \cite{KLM1} 
for instance) such that $\dyl$ is not polyhedral for a suitable choice of a flag. 

In the other direction, one would expect that divisors with finitely generated section rings would possess rational polytopes are Newton--Okounkov bodies. While this is false in general (as 
mentioned above), one can prove that under a judicious choice of the flag, the associated Newton--Okounkov body will indeed be a rational polytope. 

\begin{theorem}[\cite{AKL}]\label{thm:rtl poly}
Let $X$ be a smooth projective variety, $L$ a line bundle on $X$ with $R(X,L)$ finitely generated. Then there exists an admissible  flag $\ybul$ on $X$ such that $\dyl$ is a rational polytope.
 \end{theorem}

\begin{rmk}
The algebraic object whose finite generation guarantees that $\dyl$ is a rational polytope is the valuation semigroup $\Gamma_{\ybul}(L)$ (or, equivalently,  the associated semigroup algebra).
Note that this latter happens to be finitely generated a lot less often. 
\end{rmk}

\subsection{Local positivity}

In order to study positivity questions in terms of Newton--Okounkov bodies, we need to be able to talk about  positivity
at or in a neighbourhood of a point. 

\begin{defn}[Local ampleness]
Let $X$ be a projective variety, $x\in X$ an arbitrary point. We say that \emph{$L$ is ample at $x$}, if $x$ possesses an open neighbourhood $U\subseteq X$ such that 
\[
\phi_{|L^{\otimes m}|}\big|_U \colon U \lra \PP 	
\]
is an embedding for some (or, equivalanetly, for all) large enough $m\in\NN$. 
\end{defn}

Recall that global generation/semi-ampleness can be expressed in terms of subvarieties of $X$: for instance $L$ is semi-ample if and only if $\sbl{L}\equ \emptyset$. We wish to obtain analogous characterizations of ample/nef/etc. eventually. 

\begin{defn}[Augmented base locus, \cite{ELMNP1}]
Let $X$ be a projective variety, $L$ a line bundle, $A$ an ample line bundle on $X$, $x\in X$ arbitrary. We set
\[
\Bplus(L) \deq \sbl{L-tA} \ \ \ \text{for some $0<t\ll 1$. }	
\]
\end{defn}

One of the reasons for using augmented (and soon restricted or diminished base loci) of divisors instead of stable base loci is that the stable base locus of a line bundle does not respect numerical equivalence. 
We briefly summarize some important basic properties of augmented base loci. 

\begin{lemma}[Properties of the augmented base locus]\label{lem:props of augm base loci}
With notation as above, we have the following.
\begin{enumerate}
	\item $\Bplus(L)$ does not depent on the choice of $A$ nor the actual value of $t$ as long as it is taken to be small enough (in terms of $L$ and $A$).
	\item If $L\equiv L'$, then $\Bplus(L)=\Bplus(L')$.
	\item $L$ is ample precisely if $\Bplus(L)=\emptyset$.
	\item $L$ is big precisely if $\Bplus(L)\neq X$.
	\item $\Bstable(L)\subseteq \Bplus(L)\subseteq X$ Zariski closed. 
\end{enumerate}	
\end{lemma}
\begin{proof}
All proofs can be found in \cite{ELMNP1}*{Section 1}. 
\end{proof}

\begin{rmk}[Augmented base locus for $\RR$-divisors]
Since $\Bstable(mL)=\Bstable(L)$, one immediately obtains $\Bplus(mL)=\Bplus(L)$, and hence we can extend the definition of 
augmented base loci for $\QQ$-divisors by setting 
\[
\Bplus(D)\deq \Bplus(mD)
\]
 for some natural number $m$ for which $mD$ is 
integral. For an $\RR$-divisor class $\delta$ we define 
\[
 \Bplus(\delta) \deq \Bplus(\delta-t\alpha)
\]
for $\alpha$ an arbitrary ample class, $t>0$ sufficiently positive and such that $\delta-t\alpha$ is rational. 
\end{rmk}

The relationship between local ampleness and the augmented base locus is quickly explained by a result of Boucksom, Cacciola and Lopez. 

\begin{thm}[Boucksom--Cacciola--Lopez \cite{BCL}, Theorem A]\label{thm:BCL}
Let $X$ be a normal (do we really need this?) projective variety, $x\in X$, $L$ an line bundle on $X$. Then 
\[
\text{$L$ is ample at $x$}\ \ \Leftrightarrow\ \ x\not\in\Bplus(L)\ .
\]
\end{thm}

An analogous result for $\RR$-divisors is given in \cite{Lopez}. In a slightly different direction, since ampleness can be characterized by the existence of certain resolutions (cf. \cite{PAGI}*{Example 1.2.21}), it is not unexpected that augmented base loci can be described in the same terms. 
 
\begin{proposition}[\cite{Kur10}, Proposition 3.6]\label{prop:Kur10}
Let $L$ be an integral Cartier divisor on a projective variety. Then $\Bplus(L)$ is the smallest subset of $X$ such that 
for all coherent sheaves $\sF$ there exists a possibly infinite sequence of sheaves of the form
\[
\cdots \lra \bigoplus_{i=1}^{r_i}\sOX(-m_iL) \lra \cdots\lra \bigoplus_{i=1}^{r_0}\sOX(-m_0L) \lra \sF	\ ,
\]
with all $m_i\geq 1$, which is exact off $\Bplus(L)$.
\end{proposition}

Augmented base loci have a counterpart describing nefness/pseudo-effectivity.

\begin{defn}[Restricted base locus, \cite{ELMNP1}]
Let $X$ be a projective variety, $x\in X$, $L$ an integral Cartier divisor, $A$ an ample line bundle on $X$. We define
\[
\Bminus(L) \deq \bigcup_{m=1}^{\infty} \Bstable(L+\tfrac{1}{m}A)\ .
\]
\end{defn}

\begin{rmk}[Terminology]
In various sources (see \cite{Les} for instance) restricted base loci are called 'diminished'. We will stick to the original terminology of \cite{ELMNP1}.  
\end{rmk}

\begin{lemma}[Properties of the restricted base locus]\label{lem:props of restr base loci}
With notation as above, restricted base loci have the following properties.
\begin{enumerate}
	\item $\Bminus(L)$ is independent of the choice of $A$.
	\item $\Bminus(L)=\emptyset$ if and only if $L$ is nef.
	\item $\Bminus(L)\neq X$ if and only if $L$ is pseudo-effective.
	\item If $L\equiv L'$, then $\Bminus(L)=\Bminus(L')$.
	\item $\Bminus(L)\subseteq \Bstable(L)$
\end{enumerate}
\end{lemma}
\begin{proof}
Verifications of these statements can be found in \cite{ELMNP1}*{Section 2}. 
\end{proof}

\begin{rmk}[Restricted base loci are not always Zariski closed]
Lesieutre \cite{Les} has produced various examples of divisors with real coefficients whose restricted base loci are not Zariski closed.	
\end{rmk}

\begin{exer} $\star$ Let $X$ be a normal projective variety, $L$ a big line bundle on $X$ with $R(X,L)$ finitely generated. Show that $\Bminus(L)=\Bstable(L)$. 
\end{exer}

\begin{exer}\label{prop:openclosed}
	Let $X$ be a  projective variety, $x\in X$ an arbitrary point. Then 
	\begin{enumerate}
		\item $B_+(x) \deq \st{\alpha\in N^1(X)_\RR\mid x\in \Bplus(\alpha)} \dsubseteq N^1(X)_\RR$ is closed, 
		\item $B_-(x) \deq \st{\alpha\in N^1(X)_\RR\mid x\in \Bminus(\alpha)} \dsubseteq N^1(X)_\RR$ is open, 
	\end{enumerate}
	both with respect to the metric topology of $N^1(X)_\RR$.  
\end{exer}

\begin{exer} $\star$ Let  $f\colon Y\to X$ be a proper birational morphism between smooth projective varieties, $L$ a big line bundle on $X$. Determine $\Bplus(f^*L)$ in terms of $\Bplus(L)$, and likewise for restricted base loci.   
\end{exer}

\subsection{Newton--Okounkov bodies, nefness,  and restricted base loci}

Our immediate plan now is to describe augmented/restricted base loci in terms of Newton--Okounkov bodies. Here we will deal 
with restricted base loci and the closely related notions of nefness and pseudo-effectivity. Most of the discussion is taken
almost verbatim from \cite{KL15a}. The fundamental result  connecting restricted base loci is the following. 

\begin{thm}[\cite{KL15a}]\label{thm:restr bl}
	Let $X$ be a smooth projective variety, $D$ a big $\RR$-Cartier divisor on $X$, $x\in X$. Then the following are equivalent.
	\begin{enumerate}
		\item $x\notin \Bminus(D)$.
		\item There exists an admissible flag $\ybul$ centered at $x$ such that $\origin\in\nob{\ybul}{D}$.
		\item One has $\origin\in\nob{\ybul}{D}$  for all admissible flags $\ybul$ centered at $x$.
	\end{enumerate} 
\end{thm}

The connection between base loci and nefness comes from combining Theorem~\ref{thm:restr bl} with Lemma~\ref{lem:props of restr base loci}. 

\begin{corollary}\label{cor:nef}
	With notation as above the following are equivalent for a big $\RR$-divisor $D$.
	\begin{enumerate}
		\item $D$ is nef.
		\item For every point $x\in X$  there exists an admissible flag $\ybul$ on $X$ centered at $x$ such that $\origin\in \Delta_{\ybul}(D)\subseteq\RR^n$.
		\item One  has $\origin\in \Delta_{\ybul}(D)$ for every admissible flag $\ybul$ on $X$.  
	\end{enumerate}
\end{corollary}
\begin{proof}
	Immediate from Theorem~\ref{thm:restr bl} and  Lemma~\ref{lem:props of restr base loci}.
\end{proof}

\begin{remark}\label{rmk:non-smooth nef}
	In fact, the proof of Theorem~\ref{thm:restr bl} yields somewhat more: assuming (as we have to) that $x\in X$ is smooth, one can in fact see that the implication $(1)\Rightarrow (3)$ holds on an arbitrary projective variety both in Theorem~\ref{thm:restr bl} and Corollary~\ref{cor:nef}. 
\end{remark}

\begin{exer}
Comb through the proof of Theorem~\ref{thm:restr bl} in order to verify the claim of Remark~\ref{rmk:non-smooth nef}. 
\end{exer}

\begin{rmk}[Pseudo-effective $\RR$-divisors]
	In \cite{CHPW3}*{Theorem 4.2} the authors observe  Theorem~\ref{thm:restr bl} and Corollary~\ref{cor:nef} to limiting
	Newton--Okounkov bodies of pseudo-effective $\RR$-divisors. 
\end{rmk}
 
\begin{exer}
Prove that Theorem~\ref{thm:restr bl} holds for limiting Newton--Okounkov bodies of divisors using the definition
\[
\Delta^{\lim}_{\ybul}(D) \deq \bigcap_{m=1}^{\infty} \Delta_{\ybul}(D+\alpha_m)\ ,
\]
where $(\alpha_m)_{m\in\NN}\searrow 0$ is a decreasing sequence of ample $\RR$-divisor classes. Check first that the definition is independent of the choice of the sequence $\alpha_m$. 
\end{exer}

The essence of the proof of Theorem~\ref{thm:restr bl} is to connect the asymptotic multiplicity of $D$ at $x$ to a certain  function defined on the Newton-Okounkov body of $D$. Before turning to the actual proof, we will quickly recall the notion of the asymptotic multiplicity  (or the asymptotic order of vanishing) of 
a $\QQ$-divisor $F$ at a point $x\in X$. 

\begin{defn}[Asymptotic multiplicity]
Let $F$ be  an effective Cartier divisor on $X$, defined locally by the equation $f\in \sO_{X,x}$. Then \textit{multiplicity} of $F$ at $x$ is defined to be 
$\mult_x(F)=\textup{max}\{n\in\NN | f\in \mathfrak{m}^n_{X,x}\}$, where $\mathfrak{m}_{X,x}$ denotes the maximal ideal of the local ring $\sO_{X,x}$. If $|V|$ is a linear series, 
then the multiplicity of $|V|$ is defined to be 
\[
\mult_x(|V|) \deq \min_{F\in |V|}\{\mult_x(F)\}\ . 
\]
The \emph{asymptotic multiplicity} of a $\QQ$-divisor $D$ 
at $x$  is then defined to be
\[
\mult_x(||D||) \deq  \lim_{p\rightarrow \infty}\frac{\mult_x(|pD|)}{p}\ .
\]
\end{defn}
 
\begin{rmk}
By  semicontinuity $\min_{F\in |V|}\{\mult_x(F)\}$ equals  the multiplicity  of a general element in $|V|$ at $x$.	
The multiplicity at $x$ coincides with the order of vanishing at $x$, given in Definition~2.9 from \cite{ELMNP1}. 
In what follows we will talk about the multiplicity 
of a divisor, but the order of vanishing of a section of a line bundle. 
\end{rmk}

An important technical  ingredient of  the proof of Theorem~\ref{thm:restr bl} is a result of \cite{ELMNP1}, which we now recall. In some sense this is the counterpart of the geometric descriptions we had earlier for augmented base loci in Theorem~\ref{thm:BCL} and Proposition~\ref{prop:Kur10}.

\begin{proposition}(\cite{ELMNP1}*{Proposition 2.8})\label{prop:ELMNP}
	Let $D$ be a big $\QQ$-divisor on a smooth projective variety $X$, $x\in X$ an arbitrary (closed) point. Then the following are equivalent. 
	\begin{enumerate}
		\item There exists  $C>0$ having the property that  $\mult_x(|pD|)<C$, whenever $|pD|$ is nonempty for some positive integer $p$..
		\item $\mult_x(\|D\|) = 0$.
		\item $x\notin \Bminus(D)$. 
	\end{enumerate}
\end{proposition}

The connection between asymptotic multiplicity and Newton--Okounkov bodies comes from the claim below. 

\begin{lemma}\label{lem:1}
	Let $M$ be an integral Cartier divisor on a projective variety  $X$ (not necessarily smooth),  $s\in H^0(X,\sO_X(M))$  a non-zero global section. Then 
	\begin{equation}\label{eq:1}
	\ord_x(s) \ \leq \ \sum_{i=1}^{i=n} \nu_i(s),
	\end{equation}
	for any admissible flag $Y_{\bullet}$ centered $x$, where $\nu_{\ybul}=(\nu_1,\ldots ,\nu_n)$ is the valuation map arising from $Y_{\bullet}$.
\end{lemma}

\begin{proof}
	Since $Y_{\bullet}$ is an admissible flag and the question is local, we can assume without loss of generality that each element in the flag is smooth, 
	thus $Y_{i}\subseteq Y_{i-1}$ is Cartier for each $1\leq i\leq n$. 
	
	As the local ring $\sO_{X,x}$ is regular, order of vanishing is multiplicative. Therefore 
	\[
	\ord_x(s) \equ \nu_1(s)+\ord_x (s-\nu_1(s)Y_1) \dleq \nu_1(s)+\ord_x ( (s-\nu_1(s)Y_1)|_{Y_1})
	\]
	by the very definition of $\nu_{\ybul}(s)$, and the rest follows by induction. 
\end{proof}

\begin{remark}\label{rem:1}
	Note  that the inequality in (\ref{eq:1}) is not in general  an equality for the reason that  the zero locus of $s$ might  not intersect an element of the flag 
	transversally. For the simplest example of this phenomenon set  $X=\PP^2$, and take $s=xz-y^2\in H^0(\PP^2,\sO_{\PP^2}(2))$, $Y_1=\{x=0\}$ and $Y_2=[0:0:1]$. 
	Then clearly $\nu_1(s)=0$, and $\nu_2(s)=\ord_{Y_2}(-y^2)=2$, but since  $Y_2$ is a smooth point of $(s)_0=\{xz-y^2=0\}$,  $\ord_{Y_2}(s)=1$ and hence 
	$\ord_{Y_1}(s)<\nu_1(s)+\nu_2(s)$.
\end{remark}

For a compact convex body $\Delta\subseteq \RR^n$, we define the \textit{sum function} $\sigma:\Delta\rightarrow \RR_{+}$ by  $\sigma(x_1,\ldots ,x_n)=x_1+\ldots +x_n$. Being  continuous on a compact topological space, it takes on its extremal values.  If $\Delta_{Y_{\bullet}}(D)\subseteq \RR^n$ be a Newton--Okounkov body, then we denote the sum function by $\sigma_D$, even though it does depend on the choice of the flag $Y_{\bullet}$.

\begin{proposition}\label{prop:1}
	Let  $D$ be  a big $\QQ$-divisor on a projective variety $X$ (not necessarily smooth) and let $x\in X$ a  point. Then 
	\begin{equation}\label{eq:2}
	\mult_x(||D||) \ \leq \ \min \sigma_D .
	\end{equation}
	for any admissible flag $Y_{\bullet}$ centered at $x$.
\end{proposition}

\begin{proof}
	Since both sides of (\ref{eq:2})  are homogeneous of degree one in $D$, we can assume without loss of generality that $D$ is integral. 
	Fix a natural number $p\geq 1$ such that $|pD|\neq \varnothing$, and  let $s\in H^{0}(X,\sO_X(pD))$ be a non-zero global section. Then
	\[
	\frac{1}{p}\mult_x(|pD|)  \dleq  \frac{1}{p}\ord_x(s)  \dleq  \frac{1}{p}\big(\sum_{i=1}^{i=n}\nu_i(s)\big)
	\]
	by Lemma~\ref{lem:1}. 
	
	Multiplication of sections and the definition of the multiplicity of a linear series then yields
	$\mult_x(|qpD|)\leq q\mult_x(|pD|)$ for any $q\geq 1$, which, after taking limits leads to 
	\[
	\mult_x(||D||)  \dleq  \frac{1}{p}\mult_x(|pD|) \dleq  \frac{1}{p}\big(\sum_{i=1}^{i=n}\nu_i(s)\big).
	\]
	Varying the section $s$ and taking into account that $\Delta_{\ybul}(D)$  is 
	the closure of the set  of normalized valuation vectors of  sections, we deduce the required statement.
\end{proof}

\begin{example}\label{rem:2}
	The inequality in (\ref{eq:2}) is usually strict. For a concrete example  take $X=\textup{Bl}_P(\PP^2)$, $D=\pi^*(H)+E$ and the flag $Y_{\bullet}=(C,x)$, 
	where $C\in|3\pi^*(H)-2E|$ is the proper transform of a rational curve with a single cusp at $P$,  and $\{x\}=C\cap E$, i.e. 
	the point where $E$ and $C$ are tangent to each other. Then 
	\[
	\mult_x(||D||) \equ  \lim_{p\to\infty}\Big(\frac{\mult_x(|pD|)}{p}\Big) \equ \lim_{p\to\infty}\Big(\frac{\mult_x(|pE|)}{p}\Big)  \equ 1\ .
	\]
	On the other hand, a direct computation using \cite{LM}*{Theorem 6.4} shows that 
	\[
	\Delta_{Y_{\bullet}}(D) \equ  \{ (t,y)\in\RR^2 \ | \ 0\leq t\leq \frac{1}{3}, \textup{ and } 2+4t\leq y\leq 5-5t\}\ . 
	\]
	As a result, $\min\sigma_D =2 > 1$. 
	
	For more on this phenomenon in the surface case, see \cite{KL15a}*{Proposition 2.10}.
\end{example}

\begin{remark}
	We note here a connection with functions on Okounkov bodies coming from divisorial valuations. With the notation of \cites{BKMS12, KMS_B12}, 
	Lemma~\ref{lem:1} says that $\phi_{\ord_x} \leq \sigma_D$, and a quick computation shows that we obtain equality in the case of projectice spaces, 
	hyperplane bundles, and linear flags. Meanwhile, Example~\ref{rem:2} illustrates  that $\min \phi_{\ord_x} \neq \mult_x \|D\|$
	in general. 
\end{remark}

\begin{exer}
Let $X=\PP^2$, $D=kH$ with $H$ being the hyperplane class, $\ybul$ a linear flag (i.e. a flag consisting of linear subspaces of $\PP^2$).  Show that every vanishing behaviour along $\ybul$ can be realized by a collection of lines. More precisely, for every  vector $v$ in the image of 
\[
\nu_{\ybul}\colon  H^0(\sO_{\PP^2}(mD)) \lra \NN^2
\]
there exists a global section $s\in  H^0(\sO_{\PP^2}(mD))$ whose zero locus is a collection of lines, and $\nu_{\ybul}(s)=v$.
What happens if  $\ybul$ is an arbitrary (admissible) flag? What happens in $\PP^n$ for a linear flag?    
\end{exer}

\begin{proof}[Proof of Theorem~\ref{thm:restr bl}] 
	$(1)\Rightarrow(3)$ We are assuming  $x\notin \Bminus(D)$;  let us  fix a sequence of ample $\RR$-divisor classes $\alpha_m$ with the properties that 
	$\alpha_m-\alpha_{m+1}$ is ample for all $m\geq 1$,  $\|\alpha_m\|\to 0$ as $m\to \infty$, and  such that $D+\alpha_m$ is a $\QQ$-divisor. 
	Let $\ybul$ be an arbitrary admissible  flag centered at   $x$.
	
	Then  $x\notin \textbf{B}(D+\alpha_m)$ for every $m\geq 1$, furthermore,  Lemma~\ref{lem:nested} yields
	\begin{equation}\label{eq:3}
	\Delta_{Y_{\bullet}}(D) \equ \bigcap_{m=1}^{\infty} \Delta_{Y_{\bullet}}(D+\alpha_m) \ .
	\end{equation}
	Because  $x\notin \Bstable(D+\alpha_m)$ holds for any $m\geq 1$,  there must exist a sequence of natural numbers $n_m\geq 1$ and a  sequence of 
	global sections  $s_m\in H^0(X,\sO_X(n_m(D+\alpha_m)))$ such that $s_m(x)\neq 0$. This implies that $\nu_{Y_{\bullet}}(s_m)=\origin$ for each $m\geq 1$. 
	In particular, $\origin\in \Delta_{Y_{\bullet}}(D+\alpha_m)$ for each $m\geq 1$. 
	By  (\ref{eq:3})  we deduce that  $\Delta_{Y_{\bullet}}(D)$  contains the origin as well.
	
	The implication $(3)\Rightarrow (2)$ being trivial, we will now take care of $(2)\Rightarrow (1)$.  To this end assume that  $Y_{\bullet}$ 
	is an admissible flag  centered at $x$ having the property that  $\origin\in\Delta_{Y_{\bullet}}(D)$, $\alpha_m$ a sequence of ample $\RR$-divisor classes 
	such that $\alpha_m-\alpha_{m+1}$ is nef,  $\|\alpha_m\|\to 0$ as $m\to\infty$, and  $D+\alpha_m$ is  a $\QQ$-divisor for all $m\geq 1$.
	
	By Lemma~\ref{lem:nested}, 
	\[
	\origin \in \nob{\ybul}{D} \dsubseteq  \nob{\ybul}{D+\alpha_m} 
	\]
	for all $m\geq 0$, whence  $\min\sigma_{D+\alpha_m}=0$ for all sum functions    $\sigma_{D+\alpha_m}:\Delta_{Y_{\bullet}}(D+\alpha_m)\rightarrow\RR_{+}$.
	By Proposition~\ref{prop:1}  this forces $\mult_x(||D+\alpha_m||)=0$ for all $m\geq 1$, hence \cite{ELMNP1}*{Proposition 2.8} leads to $x\notin\Bminus(D+\alpha_m)$
	for all $m\geq 1$.
	
	As 
	\[
	\Bminus(D) \equ \bigcup_{m} \Bminus(D+\alpha_m) \equ \bigcup_m\Bstable(D+\alpha_m)
	\]
	according to \cite{ELMNP1}*{Proposition~1.19}, we are done. 
\end{proof}

It turns out that one can use Newton--Okounkov bodies to study varieties from an infinitesimal point of view. Let $X$ again be a projective variety, $x\in X$ a smooth point on $X$. Since centers of admissible flags are supposed to be smooth, we are not losing anything with this assumption. 

We will denote by $\pi\colon X'\to X$ the blow-up of $X$ at $x$ with  exceptional divisor $E$. As $x$ is smooth, $X'$ is again a projective variety, and $E$ is an irreducible Cartier divisor on $X'$, which is smooth as a subvariety of 
$X'$. 

\begin{defn}[Infinitesimal Newton--Okounkov bodies]
	With notation as above, we call a flag  $\ybul$ an   \emph{infinitesimal flag over the point} $x$, if $Y_1=E$ and each $Y_i$ is a linear subspace in $E\simeq\PP^{n-1}$ of dimension $n-i$. 
	We will often write $Y_n=\{z\}$. An \emph{infinitesimal flag over $X$} is an infinitesimal flag over $x\in X$ for some smooth point $x$. 
	
	Let $D$ be a line bundle, or more generally, an $\RR$-divisor on $X$. 
	The symbol $\inob{\ybul}{D}$ stands for an \emph{infinitesimal Newton--Okounkov body of $D$}, that is, 
	\[
	\inob{\ybul}{D} \deq \nob{\ybul}{\pi^*D} \subseteq \RR^n_+\ ,
	\]
	where  $\ybul$ is an infinitesimal flag over $x$.
\end{defn}

\begin{remark}(Difference in terminology)
	One should be careful here, as our terminology deviates from that of  \cite{LM}*{Section 5.2}. What Lazarsfeld and Musta\c t\u a call an infinitesimal Newton--Okounkov body, is in our 
	terms  (following \cite{KL14})  the \emph{generic} infinitesimal Newton--Okounkov body. 
\end{remark}

\begin{theorem}[KL15b]\label{thm:bminus inf}
	Let $X$ be a smooth projective variety, $D$ a big $\RR$-divisor and $x\in X$ an arbitrary point on $X$. Then the following are equivalent. 
	\begin{enumerate}
		\item $x\not\in \Bminus(D)$.
		\item There exists an infinitesimal flag $\ybul$ over  $x$ such that $\origin\in\inob{\ybul}{D}$.  
		\item For every infinitesimal flag $\ybul$ over  $x$, one has $\origin\in \inob{\ybul}{D}$.
	\end{enumerate}
\end{theorem}

\begin{corollary}[KL15b]\label{cor:nefness}
	Let $X$ be a smooth projective variety,  $D$ a big $\RR$-divisor on $X$. Then the following are equivalent.
	\begin{enumerate}
		\item $D$ is nef.
		\item For every point $x\in X$ there exists an infinitesimal flag $\ybul$ over $x$ such that $\origin\in\inob{\ybul}{D}$.
		\item The origin $\origin\in\inob{\ybul}{D}$ for every infinitesimal flag over $X$. 
	\end{enumerate} 
\end{corollary}

We will not go into detail about the proofs, since they are close  to those in the  non-infinitesimal case (for details we refer the reader to \cite{KL15b}).

\begin{remark}\label{rmk:Bminus general}
	Again,  the implication $(1)\Rightarrow (3)$ remains true under the weaker assumptions that $X$ is a projective variety and $x\in X$ a smooth point, and just as in the non-infinitesimal case,  the converse  is unclear since the proof of $(2)\Rightarrow (1)$ uses \cite{ELMNP1}*{Proposition~2.8}, which relies on   multiplier ideals and   Nadel vanishing.
	
	Nevertheless, there is something one can say in the case when $x\in X$ is a smooth point on a normal variety. Let $\mu\colon Y\to X$ be a  resolution of singularities of $X$ obtained by blowing up smooth centers contained in the singular locus of $X$; note that  $\mu$ ends up being  an isomorphism around $x$. Then there will exist an effective 
	$\mu$-exceptional divisor $F$ such that $-F$  is $\pi$-ample. 
	
	Write $x'\deq \mu^{-1}(x)$, $D'\deq \mu^*(D)$, and let $\ybul'$ be the  strict transform of the infinitesimal flag $\ybul$ on the blow-up of $X$ at $x$. 
	Note that $\inob{\ybul'}{D'}=\inob{\ybul}{D}$. 
	
	Assume now that $x\in \Bstable(D+A)$ for some small ample divisor $A$ on $X$ such that $D+A$ is a $\QQ$-divisor.  Then necessarily 
	$x'\in \Bstable(D'+A')=\mu^{-1}(\Bstable(D+A))$. Note that $H\deq A'-tF$ is ample for  $0<t\ll 1$ rational. 
	Then 
	\[
	\Bstable(D'+A') \dsubseteq \Bstable(D'+H) \cup F \ ,
	\] 
	and hence $x'\in \Bstable(D'+H) \subseteq \Bminus(D'$), as required. 
\end{remark}

\begin{remark}\label{rmk: Bminus pseff}
	For the purposes of this remark, let $D$ be a pseudo-effective $\RR$-divisor, and following \cite{CHPW2}  define 
	\[
	\liminob{\ybul}{D}	\deq \bigcap_{m=1}^{\infty} \inob{\ybul}{D+\alpha_m}\ ,
	\]
	where $\alpha_m$ is a sequence of ample $\RR$-divisor classes converging to $0$. It follows from this definition (easily seen to be independent of the choice of the sequence $\alpha_m$) and Theorem~\ref{thm:bminus inf} that the statement of 
	Theorem~\ref{thm:bminus inf} remains valid for pseudo-effective divisor classes as well. 
\end{remark} 

\begin{exer} $\star$ Find out to what extent the contents of this subsection(except the results requiring multiplier ideals) work in positive characteristic. 
\end{exer}

\subsection{Newton--Okounkov bodies, ampleness, and augmented base loci}

Here we discuss the how Newton--Okounkov bodies can be used to describe local positivity of $\RR$-divisors. Again, 
the exposition follows \cites{KL15a,KL15b} very closely.

As explained in   \cite{ELMNP1}*{Example 1.16}, one has  inclusions 
\[
\Bminus(D)\dsubseteq \textup{\textbf{B}}(D)\subseteq\Bplus(D)\ , 
\]
consequently, we expect that   whenever $x\notin\Bplus(D)$,  Newton--Okounkov bodies attached to $D$ should contain more than just the origin.  As we shall see below, it will turn out that under the condition above, in fact they  contain small 
neighbourhoods of $\origin\in \RR^n_{\geq 0}$ (or, equivalently, small simplices). 

We will write 
\[
\Delta_{\epsilon} \deq  \{ (x_1,\ldots, x_n)\in\RR^n_{+} \ | \ x_1+\ldots +x_n\leq \epsilon\}
\]
for the standard $\epsilon$-simplex.

As preparation we recall 

\begin{proposition}[\cite{KL15a}, Proposition 1.6]\label{prop:compute}
	Suppose $\xi$ is a big $\RR$-class and $Y_{\bullet}$ is an admissible flag on $X$. Then for any $t\in [0,\mu(\xi, Y_1))$, we have
	\[
	\Delta_{Y_{\bullet}}(\xi)_{\nu_1\geq t} \equ \Delta_{Y_{\bullet}}(\xi-tY_1)\ + t\eone,
	\]
	where $\mu(\xi,Y_1)=\sup \{\mu>0|\xi-\mu Y_1 \textup{ is big}\}$ and $\eone=(1,0,\ldots ,0)\in \RR^n$.
\end{proposition}

Analogously to the case of restricted base loci, our main result is a characterisation of $\Bplus(D)$ in terms of augmented base loci. 

\begin{theorem}[\cite{KL15a}]\label{thm:augm bl}
	Let $D$ be a big $\RR$-divisor on $X$,  $x\in X$ be an arbitrary (closed) point. Then the following  are equivalent.
	\begin{enumerate}
		\item $x\notin \textup{\textbf{B}}_+(D)$.
		\item There exists an admissible flag $Y_{\bullet}$ centered at $x$ with $Y_1$ ample  such that $\Delta_{\epsilon_0} 
		\subseteq \Delta_{Y_{\bullet}}(D)$ for some $\epsilon_0 >0$. 
		\item For every  admissible flag $Y_{\bullet}$ centered at $x$ there exists $\epsilon >0$ (possibly depending on $\ybul$) 
		such that $\Delta_{\epsilon}\subseteq \Delta_{Y_{\bullet}}(D)$.
	\end{enumerate}
\end{theorem}

Combining Theorem~\ref{thm:augm bl} with Lemma~\ref{lem:props of augm base loci} yields the desired characterization of ampleness. 

\begin{corollary}[\cite{KL15b}]\label{cor:ample}
	Let $X$ be a smooth projective variety, $D$ a big $\RR$-divisor  on $X$. Then the following are equivalent.  
	\begin{enumerate}
		\item $D$ is ample.
		\item For every point $x\in X$ there exists an admissible flag $Y_{\bullet}$ centered at $x$ with $Y_1$ ample  such that $\Delta_{\epsilon_0} 
		\subseteq \Delta_{Y_{\bullet}}(D)$ for some $\epsilon_0 >0$. 
		\item  For every  admissible flag $Y_{\bullet}$  there exists $\epsilon >0$ (possibly depending on $\ybul$) 
		such that $\Delta_{\epsilon}\subseteq \Delta_{Y_{\bullet}}(D)$.
	\end{enumerate}
\end{corollary}

\begin{rmk}
We can consider  Corollary~\ref{cor:ample} as a variant of Seshadri's criterion for ampleness (see \cite{PAGI}*{Theorem 1.4.13}) in the language of convex geometry. 
\end{rmk}

\begin{remark}
	It is shown in \cite{KL14}*{Theorem 2.4} and \cite{KL14}*{Theorem A} that  in dimension two one can in fact discard the condition above that $Y_1$ should be 
	ample. Note that the proofs of the cited results  rely heavily on  surface-specific tools and in general follow a line of thought different from the present one.
\end{remark}

\begin{rmk}
 In \cite{CHPW3}, the authors claim a result with a different condition on the flag using the heavy machinery of \cite{ELMNP2}. According to \emph{loc.~cit.} Theorem 3.6, it suffices to check condition $(2)$ of Theoorem~\ref{thm:augm bl} for flags where $Y_k$ becomes an irreducible component of $\Bplus(D)$ for some $k\geq 0$.  
\end{rmk}

Moving on to the proof, we prove a helpful lemma first.

\begin{lemma}\label{lem:2}
	Let $X$ be a projective variety (not necessarily smooth), $A$  an ample Cartier divisor, $Y_{\bullet}$ an admissible flag on $X$. 
	Then for all $m>>0$ there exist global sections 
	$s_0,\ldots ,s_n\in H^0(X,\sO_X(mA))$ for which 
	\[
	\nu_{Y_{\bullet}}(s_0) \equ \origin \text{ and } \nu_{Y_{\bullet}}(s_i)=\ei, \text{ for each } i=1,\ldots ,n,
	\]
	where $\{ \eone,\ldots ,\en\}\subseteq \RR^n$ denotes the standard basis.
\end{lemma}
\begin{proof}
	First, by the admissibility of the flag $\ybul$, we know that there is an open neighbourhood $\sU$ of $x$ such that $Y_i|_{\sU}$ is 	smooth for all $0\leq i\leq n$. 
	
	Since $A$ is ample,  $\sO_X(mA)$  becomes globally generated for  $m>>0$. For all such  $m$  there exists a non-zero section 
	$s_0\in H^0(X,\sO_X(mA))$ with $s_0(Y_n)\neq 0$,  in particular,  $\nu_{Y_{\bullet}}(s_0)=\origin$, as required. 
	
	It remains to show that for all $m>>0$ and  $i=1\leq i\leq n$ we can find  non-zero sections  $s_i\in H^0(X,\sO_X(mA))$ with 
	$\nu_{Y_{\bullet}}(s_i)=\ei$. To this end, fix $i$ and let $y\in Y_i\setminus Y_{i+1}$ be a smooth point. Having chosen  $m$ large enough,  
	Serre vanishing yields $H^1(X,\sI_{Y_i|X}\otimes \sO_X(mA))=0$, hence  the map $\phi_m$ in  the diagram 
	\[
	\xymatrix{
		& & H^0(X,\sO_X(mA)) \ar[d]_{\phi_m}  \\
		0\ar[r] &  H^0(Y_i,\sO_{Y_i}(m(A|_{Y_i})-Y_{i+1})) \ar[r]^<<<<{\psi_m} & H^0(Y_i,\sO_{Y_i}(mA))  }
	\]
	is surjective. 
	
	Again, by making $m$ high enough,  we can assume  $|m(A|_{Y_i})-Y_{i+1}|$ to be  very ample on $Y_i$, thus,  there will exist 
	$0\neq\tilde{s}_i\in H^0(Y_i,\sO_{Y_i}(mA)\otimes \sO_{Y_i}(-Y_{i+1}))$ not  vanishing  at $x$ or $y$. Since $\tilde{s}_i(x)\neq 0$, the section
	$\tilde{s}_i$  does not  vanish along $Y_j$ for all $j=i+1,\ldots ,n$. Also, the image   
	$\psi_m(\tilde{s}_i)\in H^0(Y_i,\sO_{Y_i}(mA))$ of $\tilde{s}_i$  vanishes at $x$, but not at the point  $y$.
	
	By the surjectivity of the map  $\phi_m$  there exists a section $s_i\in H^0(X,\sO_X(mA))$ such that $s|_{Y_i}=\psi_m(\tilde{s}_i)$ and $s(y)\neq 0$. 
	In particular, $s_i$ does not vanish along any of the $Y_j$'s for  $1\leq i\leq j$, therefore  $\nu_{Y_{\bullet}}(s)=\ei$, as promised. 
\end{proof}

\begin{proof}[Proof of Theorem~\ref{thm:augm bl}]
	$(1)\Rightarrow (3)$. First we treat the case when $D$ is $\QQ$-Cartier. Assume that  $x\notin \Bplus(D)$, which implies by definition that  
	$x\notin\Bstable(D-A)$ for some small ample $\QQ$-Cartier  divisor $A$. Choose a positive integer  $m$ large and divisible enough  such that $mA$ 
	becomes integral,  and satisfies the conclusions of  Lemma~\ref{lem:2}. Assume furthermore that $\Bstable(D-A)=\textup{Bs}(m(D-A))$ set-theoretically.  
	
	Since $x\notin \textup{Bs}(m(D-A))$,  there exists a section $s\in H^0(X,\sO_X(mD-mA))$ not vanishing at $x$,  and
	in particular $\nu_{Y_{\bullet}}(s)=\origin$. At the same time,  Lemma~\ref{lem:2} provides the existence of  global  sections 
	$s_0,\ldots, s_n\in H^0(X,\sO_X(mA))$ with the property that $\nu_{Y_{\bullet}}(s_0)\equ \origin$ and $\nu_{Y_{\bullet}}(s_i)=\ei$ for all 
	$1\leq i\leq n$. 
	
	But then the multiplicativity of the valuation map $\nu_{\ybul}$ gives 
	\[
	\nu_{Y_{\bullet}}(s\otimes s_0)=\textbf{\textup{0}}, \textup{ and } \nu_{Y_{\bullet}}(s\otimes s_i)=\ei \ \ \text{for all $1\leq i\leq n$.}
	\]
	By the construction of Newton--Okounkov bodies, then $\Delta_{\frac{1}{m}}\subseteq \Delta_{Y_{\bullet}}(D)$. 
	
	Next, let $D$ be a big $\RR$-divisor for which $x\notin\Bplus(D)$, and let  $A$ be an ample $\RR$-divisor with the property that 
	$D-A$ is a $\QQ$-divisor and $\Bplus(D)=\Bplus(D-A)$. Then we have $x\notin\Bplus(D-A)$, therefore 
	\[
	\Delta_\epsilon\dsubseteq \Delta_{\ybul}(D-A) \dsubseteq \Delta_{\ybul}(D) 
	\]
	according to the $\QQ$-Cartier case and  Lemma~\ref{lem:nested}. 
	
	Again, the implication $(3)\Rightarrow (2)$ is trivial, hence we only need to take care of $(2)\Rightarrow (1)$. 
	As $Y_1$ is ample, \cite{ELMNP1}*{Proposition 1.21} gives  the equality $\Bminus(D-\epsilon Y_1)=\Bplus(D)$.  for all  $0<\epsilon <<1$. 
	Fix an $\epsilon$ as above, subject to the additional condition that $D-\epsilon Y_1$ is a big $\QQ$-divisor. Then, according to 
	Proposition~\ref{prop:compute},  we have 
	\[
	\Delta_{Y_{\bullet}}(D)_{\nu_1\geq \epsilon} \ = \ \Delta_{Y_{\bullet}}(D-\epsilon Y_1)\ + \ \epsilon \eone \ ,
	\]
	which yields $\origin\in \Delta_{Y_{\bullet}}(D-\epsilon Y_1)$. By Theorem~\ref{thm:restr bl}, this means that  
	$x\notin\Bminus(D-\epsilon Y_1)=\Bplus(D)$, which completes the proof. 
\end{proof}

\begin{remark}\label{rmk:non-smooth ample}
	The condition that $X$ be smooth can again be dropped for the implication $(1)\Rightarrow (3)$ both in Theorem~\ref{thm:augm bl} and 
	Corollary~\ref{cor:ample} (cf. Remark~\ref{rmk:non-smooth nef}). This way, one obtains the statement that whenever $A$ is an ample $\RR$-Cartier 
	divisor on a projective variety $X$,  then every Newton--Okounkov body of $A$ contains a small simplex. 
\end{remark}
 
\begin{example}
	Let $\PP^2 = Y_0 \supset H_m\supset \st{P}$, where $H_m\in |\sO_{\PP^2}(m)|$ is a smooth hypersurface, $L$ the hyperplane divisor.  Then $\dyl$ is a triangle with vertices $(0,0)$,$(1/m,0)$, and $(0,m)$. This way we see that we can very easily make the simplex $\Delta_\lambda\subseteq \dyl$  to be as small as we wish. 
\end{example}

The question of interest is therefore how large a simplex we can fit into some Newton--Okounkov body of a given divisor $L$. 

\begin{definition}[Largest simplex constant \cites{KL14,KL15a}]
	Let $X$ be an arbitrary projective variety, $x\in X$ a smooth point, $A$ an ample $\RR$-divisor on $X$. For an admissible flag  $\ybul$ on $X$ 
	centered at $x$, we set 
	\[
	\lambda_{\ybul}(A;x) \deq \sup\st{\lambda>0\mid \Delta_\lambda\subseteq\Delta_{\ybul}(A)}\ .
	\]
	Then the \emph{largest simplex constant} $\lambda(A;x)$ is defined as 
	\[
	\lambda(A;x) \deq \sup \st{\lambda_{\ybul}(A;x)\mid \text{ $\ybul$ is an admissible flag centered at $x$}}\ .
	\]
\end{definition}

\begin{remark}
	It follows from Remark~\ref{rmk:non-smooth ample} that $\lambda(A;x)>0$. The largest simplex constant is a measure of local positivity, and  it is known 
	in dimension two that $\lambda(A;x)\leq \epsilon(A;x)$ (where the right-hand side denotes the appropriate  Seshadri constant) with strict inequality in general
	(cf. \cite{KL14}*{Proposition 4.7} and \cite{KL14}*{Remark 4.9}).
\end{remark}

\begin{exer} * Let $X$ be a smooth projective variety, $L$ an ample line bundle on $X$, $x\in X$. 
 \begin{enumerate}
 	\item Prove that $\lambda(\ \ ;x)$ is subadditive, and   $\lambda(mL;x)=m\cdot \lambda(L;x)$ for $m\in \NN$.  
 	\item Show that $\lambda(L;x) \leq \e(L;x)$. 
 	\item Find examples of ample divisors $L$ on $X$ for which $\lambda(L;x)<\e(L;x)$. 
 \end{enumerate}
\end{exer}

Analogously to the case of restricted base loci, we will now have a look at how augmented base loci can be determined in terms of infinitesimal Newton--Okounkov bodies (for the necessary definitions see Subsection~2.4). 

We start with an observation explaining the shapes of the 'right' kind of simplices that play the role of standard simplices in the infinitesimal theory.

\begin{lemma}(cf. \cite{KL15b}*{Lemma 3.4})\label{lem:2inf}
	Let $X$ be a projective variety, $x\in X$ a smooth point, and $A$ an ample Cartier divisor on $X$. Then there exists $m_0\in\NN$ such that 
	for  any infinitesimal flag $Y_{\bullet}$ over $x$ and  for every  $m\geq m_0$ there exist global sections  
	$s_0',\ldots ,s_n'\in \HH{0}{X'}{\sO_{X'}(\pi^*(mA))}$ for which 
	\[
	\nu_{Y_{\bullet}}(s_0') \equ  \origin\ ,\ \nu_{Y_{\bullet}}(s_1')\equ \eone\ ,\ \text{ and }\ \nu_{Y_{\bullet}}(s_i') \equ \eone+\ei, 
	\text{ for every $2\leq i\leq n$,}
	\]
	where $\{ \eone,\ldots ,\en\}\subseteq \RR^n$ denotes the standard basis.
\end{lemma}

\begin{proof}
	The line bundle $A$ is ample, therefore  there exists a natural number $m'_0>0$ such that $(m'_0+m)A$ is very ample for any $m\geq 0$. In particular, each linear series $|(m'_0+m)A|$ defines an embedding.  As $|m'_0A|$ separates tangent directions as well,  Bertini's theorem yields the  existence of  hyperplane sections 
	$H_1,\ldots, H_{n-1} \in |m'_0 A|$ intersecting  transversally at $x$, and $\tilde{H}_1\cap\ldots\cap\tilde{H}_i\cap E = Y_{i+1}$ for all $i=1,\ldots, n-1$, 
	where $\tilde{H}_i$ denotes  the strict transform of $H_i$ through the blow-up map $\pi$.
	
	At the same time observe that for any $m\geq m_0$ there exists a global section $t\in H^0(X,\sO_X(mA))$ not passing through $x$. By setting $s_i'\deq \pi^*(t\otimes s_{i-1})$ 
	where $s_i\in\HH{0}{X}{\sO_X(m_0A)}$ is  a section associated to $H_i$, then the sections $s_0',\ldots,s_n'$ satisfy the requirements, where we chose $m_0=2m'_0$.
\end{proof} 

\begin{definition}
	For a positive real number $\xi\geq 0$, the \emph{inverted standard simplex of size $\xi$}, denoted by $\Delta_{\xi}^{-1}$, is the convex hull of the set
	\[
	\Delta_{\xi}^{-1} \ \deq \ \st{\origin, \xi \eone,\xi(\eone+\etwo),\ldots,\xi(\eone+\en)} \dsubseteq \RR^n .
	\]
	When $\xi =0$, then $\Delta_{\xi}^{-1}=\origin$.
\end{definition}
Notice that one can find a different description of an inverted standard simplex as follows
\[
\Delta_{\xi}^{-1} \ = \ \{(x_1,\ldots ,x_n)\in\RR_{+}^n \ | \ 0\leq x_1\leq \xi, 0\leq x_2+\ldots +x_n\leq x_1\}
\]
A major difference from  the non-infinitesimal case is the fact that infinitesimal Newton--Okounkov bodies are also contained in inverted simplices 
in a very natural way.

\begin{proposition}\label{prop:inverted}
	Let $D$ be a big $\RR$-divsor $X$, then  $\Delta_{Y_{\bullet}}(\pi^*(D))\subseteq \iss{\mu(D;x)}$ 
	for any  infinitesimal flag $Y_{\bullet}$ over the point $x$. 
\end{proposition}

\begin{proof}
	By the continuity of Newton--Okounkov bodies inside the big cone it suffices  to treat the case when $D$ is a big $\QQ$-divisor. Homogeneity then lets 
	us assume that $D$ is integral. Set $\mu=\mu(D;x)$. 
	
	We will follow the line of thought of the proof of \cite{KL14}*{Proposition 3.2}. Recall that  $E\simeq \PP^{n-1}$;  we will write  $[y_1:\ldots :y_{n}]\in\PP^{n-1}$ 
	for a set of homogeneous coordinates in $E$ such that 
	\[
	Y_i \equ \Zeroes(y_1,\dots,y_{i-1}) \dsubseteq \PP^{n-1} \equ E\  \text{for all $2\leq i\leq n$.}
	\]
	With respect to a system of local coordinates  $(u_1,\ldots ,u_n)$ at  the point $x$, the blow-up $X'$ can be described (locally around $x$) as 
	\[
	X' \equ  \big\{\big((u_1,\ldots ,u_n);[y_1:\ldots :y_{n}]\big) \ | \ u_iy_j=u_jy_i \text{ for any } 1\leq i<j\leq n \big\}\ .
	\] 
	We can then write a global section $s$ of $D$  in the form 
	\[
	s \equ  P_m(u_1,\ldots ,u_n)+P_{m+1}(u_1,\ldots ,u_n)+\ldots + P_{m+k}(u_1,\ldots u_n) 
	\]
	around $x$,  where $P_{i}$ are homogeneous polynomials of degree $i$. 
	
	We will perform the computation in the open subset $U_n=\{y_n\neq 0\}$, where we can take $y_n=1$ and the defining equations of the blow-up 
	are given by $u_i=u_ny_i$ for  $1\leq i\leq n-1$. Then
	\[
	s|_{U_n} \ \equ \ u_n^m\cdot \big(P_m(y_1,\ldots ,y_{n-1},1)+u_nP_{m+1}(y_1,\ldots ,y_{n-1},1)+\ldots +u_n^kP_{m+k}(y_1,\ldots ,y_{n-1},1)\big)\ ,
	\]
	in particular,  $\nu_1(s)=m$. Notice that for the rest of $\nu_i(s)$'s  we have to restrict to the exceptional divisor $u_n=0$ and 
	thus the only term arising  in the computation is  $P_m(y_1,\ldots ,y_{n-1},1)$. 
	
	As $\deg P_m\leq m$, taking into account the algorithm for constructing the valuation vector of a section one can see that indeed 
	\[
	\nu_2(s)+\ldots +\nu_n(s) \ \leq \ \nu_1(s) \ ,
	\]
	and this finishes the proof of the proposition.
\end{proof}

 \begin{theorem}[\cite{KL15b}]\label{thm:main augmented}
 	Let $X$ be a smooth projective variety,  $D$ a big $\RR$-divisor on $X$ and $x\in X$ an arbitrary (closed) point. Then the following are equivalent.
 	\begin{enumerate}
 		\item $x\notin\Bplus(D)$.
 		\item For every infinitesimal flag $\ybul$ over $x$ there is $\xi>0$ such that $\iss{\xi}\subseteq \inob{\ybul}{D}$. 
 		\item There exists an infinitesimal $\ybul$ over $x$ and $\xi>0$  such that $\iss{\xi}\subseteq \inob{\ybul}{D}$.
 	\end{enumerate}
 \end{theorem}
 
 As an immediate consequence via the equivalence of ampleness and $\Bplus$ being empty, we obtain 
 
 \begin{corollary}[\cite{KL15b}]\label{cor:ampleness}
 	Let $X$ be a smooth projective variety and  $D$ a big $\RR$-divisor on $X$. Then the following are equivalent.
 	\begin{enumerate}
 		\item $D$ is ample.
 		\item For every point $x\in X$ and every infinitesimal flag $\ybul$ over $x$ there exists a real number $\xi>0$ for which $\iss{\xi}\subseteq \inob{\ybul}{D}$. 
 		\item For every point $x\in X$ there exists an infinitesimal flag $\ybul$ over $x$ and a real number $\xi>0$ such that  $\iss{\xi}\subseteq \inob{\ybul}{D}$. 
 	\end{enumerate}
 \end{corollary}

 \begin{remark}
 	Along the lines of Remark~\ref{rmk:Bminus general}, Theorem~\ref{thm:main augmented} extends to the case when $x\in X$ is a smooth point on a normal projective 
 	variety. Using the notation of Remark~\ref{rmk:Bminus general},  this follows from the observation found in \cite{BBP}*{Proposition 2.3} that  $\Bplus(D')=\mu^{-1}(\Bplus(D))\cup\Exc(\mu)$ whenever  $\mu\colon Y\to X$ is a resolution of singularities of $X$ that is an isomorphism over $x\in X$. 
 \end{remark}

Analogously to the non-infinitesimal case (and as observed in dimension two in \cite{KL14}), $x\notin\Bplus(D)$ implies that $\inob{\ybul}{D}$ will contain an inverted standard simplex of some 
size, hence it makes sense to ask how large these simplices  can become.
  
 \begin{definition}(Largest inverted simplex constant)
  Let $X$ be a smooth projective variety, $x\in X$, and $D$ a big $\RR$-divisor with $x\notin \Bminus(D)$. Let us first assume that $x\notin \Bplus(D)$. 
  For an infinitesimal flag $\ybul$ over $x$ write 
  \[
   \xi_{\ybul}(D;x) \deq \sup\st{\xi\geq 0\mid \iss{\xi}\subseteq \inob{\ybul}{D}}\ .
  \]
  The \emph{largest inverted simplex constant} $\xi(D;x)$ of $D$ at $x$ is then defined as 
  \[
   \xi(D;x) \deq \sup_{\ybul} \xi_{\ybul}(D;x)\ ,
  \]
  where $\ybul$ runs through all infinitesimal flags over $x$. If $x\in \Bplus(D)$ but $x\notin\Bminus(D)$, then we set  $\xi(D;x)=0$. 
 \end{definition}

\begin{remark}\label{rmk:xi continuous}
As  Newton--Okounkov bodies are homogeneous, so will be  $\xi(\ \cdot\ ;x)$ as a function on $N^1(X)_\RR$. More importantly, the function $\xi(\ \cdot\ ;x)$ is  
continuous on  the domain where $x\notin \Bplus(D)$. Indeed, one can show (cf. \cite{KL15b}) that 
$\xi_{Y_{\bullet}}(D;x)$ is in fact independent of  $Y_{\bullet}$, therefore we can use one flag for all $\RR$-divisor classes. The natural inclusion
\[
\Delta_{Y_{\bullet}}(D) + \Delta_{Y_{\bullet}}(D')\dsubseteq \Delta_{Y_{\bullet}}(D+D')
\]
shows that  $\xi(\cdot;x)$ is in fact a concave  function on 
$\Bbig(X)\setminus B_+(x)$. This  latter is an open   subset of $\textup{N}^1(X)_\RR$,  therefore $\xi(\ \cdot\ ;x)$ is continuous on its domain. For further results regarding continuity, 
the reader should consult \cite{KL15b}*{Section 5}. 
\end{remark}
 
One of the main steps of the proof of Theorem~\ref{thm:main augmented} is the observation that the largest inverted simplex constant governs the asymptotic behaviour of jet separation. 
With this in hand one can prove that the largest inverted simplex constant is equal to the moving Seshadri constant. 

\begin{thm}[\cite{KL15b}, Proposition 4.10]\label{thm: lis is mov Sesh}
 Let  $D$ be a big $\RR$-divisor on a smooth projective variety $X$ and $x\in X$ a closed point. If $\xi(D;x)>0$, then $\xi(D;x)=\epsilon(||D||;x)$. 
\end{thm}

  \newpage

\section{Newton--Okounkov bodies on surfaces}

\subsection{Zariski decomposition on surfaces} 

The main tool to compute Newton--Okounkov bodies of divisors on surfaces is Zariski decomposition, whose variation inside the big cone determines their shapes. This phenomenon is strongly related to decompositions of big or effective cones of varieties (see \cites{HK, KKL12}). In fact, along with toric varieties, this is the class of non-trivial examples that is the best understood. As such, the study of variation of Zariski decomposition on surfaces  serves as a model case for higher-dimensional investigations of positive cones associated to divisors or even higher-codimension cycles \cite{FL_Zariski}. 

In its original form put forth by Zariski \cite{Zar}, Zariski decomposition gives a unique way of decomposing an effective $\QQ$-divisor into a positive (nef) part carrying all the sections for all multiples of the divisor in question, and a negative part which is either zero or a so-called negative cycle, an effective divisor with negative definite intersection form. Zariski's proof  made essential use of the fact that the divisor we wish to decompose is effective. 
A particularly clear proof along similar lines is given by Bauer in \cite{Bauer}, where he defines the positive part as the maximal nef subdivisor. 

Subsequently, Fujita provided an alternative proof relying on linear algebra in hyperbolic vector spaces (and Riemann--Roch), thus producing two essential improvements: Fujita's version of Zariski decomposition is valid for pseudo-effective $\RR$-divisors.  

Variation of Zariski decomposition  has been first explored in \cite{BKS04}.
The upshot is that by looking at the way the support of the negative part of the Zariski decomposition changes as we move through the cone of big divisors, 
we obtain a locally finite  decomposition of $\Bbig(X)$ into locally rational polyhedral chambers. 
On the individual chambers the support of the stable base loci remains constant, and asymptotic cohomology functions (in particular the volume) are given by a single polynomial. 

The decomposition obtained this way is quite close to the one discussed in \cite{HK} or \cite{KKL12}, in fact, when the surface under consideration is a Mori dream space, the open parts of the chambers in the respective 
decompositions agree. However, it is a significant difference that here we do not rely on any kind of finite generation hypothesis, the decomposition of the cone of big divisors that we present always exists. 
The self-duality between curves and  divisors enables us to carry out this analysis with the help of linear algebra in hyperbolic vector spaces. It is this property of surfaces 
 that makes the proofs particularly transparent. 

Since we consider variation of Zariski decomposition on surfaces to be very important, we give a fairly detailed account. Most of the material presented here is borrowed freely from \cite{BKS04}.
For reference, here is the statement of Zariski in the version generalized by Fujita.

\begin{thm}[Existence and uniqueness of Zariski decompositions for $\RR$-di\-vi\-sors, \cite{KMM87}, Theorem 7.3.1]\label{thm:ZD}
   Let $D$ be a pseudo-effective $\RR$-divisor on a smooth
   projective surface. Then there exists a unique effective
   $\RR$-divisor \[ N_D=\sum_{i=1}^m a_iN_i \]  such that
   \begin{items}
      \item[(i)]
         $P_D=D-N_D$ is nef,
      \item[(ii)]
         $N_D$ is either zero or its intersection matrix
         $(N_i\cdot N_j)$ is negative definite,
      \item[(iii)] 
         $P_D\cdot N_i=0$ for $i=1,\dots, m$.
   \end{items}
   Furthermore, $N_D$ is uniquely determined by
   the numerical equivalence class of $D$, and
   if $D$ is a $\bbQ$-divisor, then so are $P_D$ and $N_D$.
   The decomposition 
   \[
   D=P_D+N_D
   \]
   is called the {\em Zariski decomposition} of $D$.
\end{thm}

First we give an example to illustrate the kind of picture we have in mind. 

\begin{eg}[Blow-up of two points in the plane]\label{eg:two_points}
Let $X$ be the blow-up of the projective plane at two points; the corresponding exceptional divisors will be denoted by $E_1$ and $E_2$. 
We denote the pullback of the hyperplane class on ${\PP}^2$ by $L$.
   These divisor classes generate the Picard group of $X$
   and their intersection numbers are: $L^2=1$, $(L.E_i)=0$ and $(E_i.E_j)=-{\delta}_{ij}$ for $1\leq i,j\leq 2$.
   
   There are  three irreducible negative curves: the two exceptional divisors,
   $E_1$,  $E_2$ and  the strict transform  of the line through the two blown-up points, $L-E_1-E_2$, and  the corresponding hyperplanes determine the 
   chamber structure on the big cone.  They divide the big cone into five regions 
   on each of which the support of the negative part of the Zariski decomposition remains constant.

   In this particular case the chambers are simply described as the set of divisors 
   that intersect negatively the same set of negative curves 
\footnote{It happens for instance on K3 surfaces that the negative part of a divisor $D$ contains irreducible curves intersecting $D$ non-negatively.}.

We will parametrize the chambers with big and nef divisors (in actual fact with faces of the nef cone containing big divisors). In our case, we pick big and nef divisors $A,Q_1,Q_2,L,P$ 
based on the following criteria:

\begin{eqnarray*}
A\ : & \text{ample} \\
Q_1\ : & (Q_1\cdot E_1)=0\ ,\ (Q_1\cdot C) > 0 \ \text{for all other curves} \\ 
Q_2\ : & (Q_2\cdot E_2)=0\ ,\ (Q_2\cdot C) > 0 \ \text{for all other curves}  \\
L\ : & (L\cdot E_1)=0\ ,\  (L\cdot E_2)=0\ ,\ (L\cdot C) > 0 \ \text{for all other curves} \\ 
P\ : & (P\cdot L-E_1-E_2)=0\ ,\ (P\cdot C) > 0 \ \text{for all other curves} \ .
\end{eqnarray*}

The divisors  $L, P,Q_1,Q_2$ are big and nef divisors in the nef boundary (hence necessarily non-ample) which are in the relative interiors of the indicated faces. 
For one of these divisors, say $P$,  the corresponding chamber consists of all $\RR$-divisor classes whose negative part consists of curves orthogonal to $P$. This is listed in the following
table. Observe that apart from the nef cone,  the chambers do not contain the nef divisors they are associated to.

Note that  not all possible combinations of negative divisors  occur. 
    This in part is accounted for by the fact that certain faces of the 
    nef cone do not contain big divisors. 

\begin{center}
 \begin{tabular}{|l|r|r|} \hline\hline
 Chamber & $\Supp \Neg(D)$ & $(D\cdot C)<0$  \\ \hline
$\Sigma_A$ & $\emptyset$ & none \\ 
$\Sigma_{Q_1}$ & $E_1$ & $E_1$ \\
$\Sigma_{Q_2}$ & $E_2$ & $E_2$ \\
$\Sigma_{L}$ & $E_1,E_2$ & $E_1,E_2$ \\
$\Sigma_{P}$ & $L-E_1-E_2$ & $L-E_1-E_2$ \\
\hline\hline
 \end{tabular}
\end{center}

The following picture describes a cross section of the effective cone of $X$ with the chamber structure indicated.

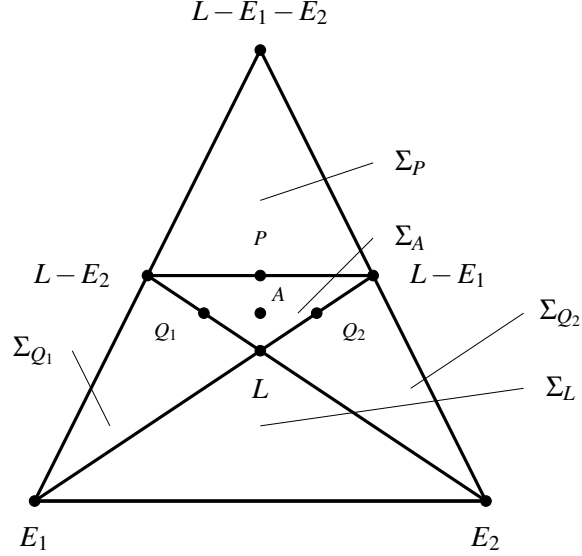
\begin{center}
\begin{tikzpicture}[scale=0.5]
\draw[very thick] (0,0) -- (12,0) -- (6,12) -- cycle;
\filldraw[black] (0,0) circle (4pt) (12,0) circle (4pt) (6,12) circle (4pt);
\draw (0,-1) node{$E_1$} (12,-1) node{$E_2$} (6,13) node{$L-E_1-E_2$};
\draw[very thick] (0,0) -- (9,6) -- (3,6) --  (12,0) -- cycle;
\filldraw[black] (9,6) circle (4pt) (3,6) circle (4pt) (6,4) circle (4pt);
\draw (6,3) node{$L$} (11,6) node{$L-E_1$} (1,6) node{$L-E_2$};
\filldraw[black] (6,5) circle (4pt) (6,6) circle (4pt) (7.5,5) circle (4pt) (4.5,5) circle (4pt);
\draw (6.5,5.5) node{\scriptsize $A$} (8.5,4.5) node{\scriptsize $Q_2$} (3.5,4.5) node{\scriptsize $Q_1$} (6,7) node{\scriptsize $P$};
\draw[very thin] (6,2) -- (13,3)  (14,3) node{$\Sigma_L$};
\draw[very thin] (7,5) -- (9,7)  (10,7) node{$\Sigma_A$};
\draw[very thin] (6,8) -- (9,9) (10,9) node{$\Sigma_P$};
\draw[very thin] (10,3) -- (13,5) (14,5) node{$\Sigma_{Q_2}$};
\draw[very thin] (2,2) -- (1,4) (0,4) node{$\Sigma_{Q_1}$};
\end{tikzpicture}
\end{center}

Let $D=aL-b_1E_i-b_2E_2$ be a big $\RR$-divisor. Then one can express 
   the volume of $D$ in terms of the coordinates $a,b_1,b_2$ as follows:
   \[
       \vl{D}= \left\{ \begin{array}{ll} D^2=a^2-b_1^2-b_2^2 & \textrm{ if $D$ is nef, i.e. $D\in\Sigma_A$} \\
       a^2-b_2^2 & \textrm{ if\ } D\cdot E_1<0 \mbox{ and } D\cdot E_2\geq 0 \textrm{ i.e. $D\in \Sigma_{Q_1}$}\\
       a^2-b_1^2 & \textrm{ if\ } D\cdot E_2<0 \mbox{ and } D\cdot E_1\geq 0 \textrm{ i.e. $D\in\Sigma_{Q_2}$}\\
       a^2 & \textrm{ if } D\cdot E_1<0 \mbox{ and } D\cdot E_2<0 \textrm{ i.e. $D\in\Sigma_L$}\\
       2a^2-2ab_1-2ab_2+2b_1b_2 & \textrm{ if\ } D\cdot(L-E_1-E_2)<0\
       \textrm{ i.e. $D\in\Sigma_P$}.
       \end{array} \right.
    \]
\end{eg}

On an arbitrary  smooth projective surface we obtain the following statement.

\begin{thm}[\cite{BKS04}]\label{thm:loc poly surfaces}
Let $X$ be smooth projective surface. Then there exists a locally finite rational polyhedral decomposition of the cone of big divisors such that on each of the resulting
chambers that support of the negative part of the Zariski decomposition is constant. 

In particular, on every such region, all asymptotic cohomology functions are given by homogeneous polynomials, and the stable base locus is constant. 
\end{thm}

It turns out that the existence of a locally finite polyhedral decomposition is relatively simple; most of the work is required for the rationality, and the explicit description of the 
chambers. For this reason, we give a separate  proof for  local finiteness. 

First, we establish some notation.  If $D$ is an $\RR$-divisor,
   we will write
   $$
      D \equ P_D+N_D
   $$
   for its Zariski decomposition, and we let
   $$
      \Null(D) \equ \set{C\with C \mbox{ irreducible curve with } D\cdot C=0}
   $$
   and
   $$
      \Neg(D) \equ \set{C\with C \mbox{ irreducible component of }    N_D}\ .
   $$
It is immediate that  $\Neg(D)\subset\Null(P_D)$.

  Given a big and nef $\RR$-divisor $P$, the Zariski chamber associated to $P$ is defined as 
   $$
      \Sigma_P \deq  \set{D\in\BigCone(X)\with\Neg(D)=\Null(P)}
      \ .
   $$

\begin{thm}
The subsets $\Sigma_P\subseteq \Bbig(X)$ are convex cones. As $P$ runs through all big and nef divisors, they form a locally finite polyhedral decomposition of $\Bbig(X)$. 
\end{thm}

\begin{proof}
Once we check that the decomposition of $\Bbig(X)$ into the convex $\Sigma_P$'s is locally finite, the locally polyhedral property comes for free 
from an observation in convex geometry (Lemma~\ref{lem:convex geometry}).

The convexity of the chambers $\Sigma_P$ follows from the uniqueness of Zariski decomposition: if 
\[
 D \equ P + \sum_{i=1}^{r}a_iE_i\ \ \text{ and }\ \ D' \equ P' + \sum_{i=1}^{r}a_i'E_i
\]
 are the respective Zariski decompositions of the big divisors $D$ and $D'$, then 
\[
 D+D' \equ (P+P') + \sum_{i=1}^{r}(a_i+a_i')E_i
\]
is the unique way to decompose $D+D'$ satisfying all the necessary requirements  of \ref{thm:ZD}.

Next, we need to show that given a big divisor $D$, there exists a big and nef divisor $P$ with
\[
 \Neg(D) \equ \Null(P)\ .
\]
Assume that $\Neg(D) \equ \st{E_1,\dots,E_r}$, and let $A$ be an arbitrary ample divisor on $X$.
 We claim that a divisor $P$ as required can be constructed
   explicitly in the form
   $A+\sum_{i=1}^k\lambda_i C_i$ with suitable non-negative
   rational numbers
   $\lambda_i$. In fact, the conditions to be fulfilled are
   $$
      \(A+\sum_{i=1}^k\lambda_i C_i\)\cdot C_j=0
        \quad\mbox{ for } j=1,\dots,k \ .
   $$
   This is a system of linear equations with negative definite
   coefficient matrix
   $(C_i\cdot C_j)$, and \cite{BKS04}*{Lemma 4.1} guarantees
   that all components $\lambda_i$
   of its solution are non-negative.
   In fact all $\lambda_i$'s  must be positive as $A$ is ample.

Last, we verify that the decomposition is locally finite.  Every big divisor has an open neighborhood
   in $\BigCone(X)$ of the form
   $$
      D+\Amp(X)
   $$
   for some big divisor $D$.
   Since 
\[
 \Neg(D+A) \dsubseteq \Neg(D)
\]
for all ample divisors $A$ by \cite{BKS04}*{Lemma 1.12},  only finitely many  chambers $\Sigma_P$ can meet this neighborhood.
\end{proof}

\begin{lem}\label{lem:convex geometry}
 Let $U\subseteq \RR^n$ be a convex open subset, and let $\shs=\st{S_i\,|\,i\in I}$ be a locally finite decomposition of $U$ into convex subsets. Then  every point $x\in U$ has  an open
neighbourhood $V\subseteq U$ such that each $S_i\cap V$ is either empty, or given as the intersection of finitely many half-spaces.  
\end{lem}
\begin{proof}
It is known that any two disjoint convex subsets of $\RR^n$ can be separated by a hyperplane. Let $x\in U$ be an arbitrary point, $x\in V\subseteq U$ a convex open neighbourhood such that only finitely
many of the $S_i$'s intersect $V$, call them $S_1,\dots,S_r$. For any pair $1\leq i<j\leq r$ let $H_{ij}$ be a hyperplane separating $S_i\cap V$ and $S_j\in V$.

Then for every $1\leq i\leq r$, $S_i\cap V$ equals the the intersection of all half-spaces $H_{ij}^+$ and $V$, as required. 
\end{proof}

\begin{cor}
With notation as above, the part of the nef cone, which sits inside the cone of big divisors is locally polyhedral.
\end{cor}

\begin{eg}
The question whether the chambers are indeed only locally finite polyhedral comes up naturally. Here we present an example with a non-polyhedral chamber.    
Take a surface $X$ with infinitely many $(-1)$-curves
   $C_1,C_2,\dots$, and blow it up at a point that is not
   contained in any of the curves $C_i$. On the blow-up
   consider the exceptional divisor $E$ and the proper transforms
   $C'_i$.
   Since the divisor $E+C'_i$ is negative definite,
   we can proceed to construct for every index $i$
   a big and nef divisor $P_i$ with
   $\Null(P_i)=\set{E,C'_i}$, and also a divisor
   $P$ such that $\Null(P)=\set{E}$.
   But then $\Face(P)$ meets contains
   all faces $\Face(P_i)$, and
   therefore it is
   not polyhedral.
\end{eg}

We now move on to explaining why the chamber decomposition is rational. This depends on two facts:
\begin{enumerate}
 \item The intersection of the nef cone with the big cone is locally finite rational polyhedral.
\item If $F$ is a face of the nef cone containing a big divisor $P$, then
\[
 \Bbig(X)\cap \overline{\Sigma_P} \equ (\Bbig(X)\cap F) + \text{ the cone generated by $\Null(P)$}\ .
\]
\end{enumerate}

Of these, the proof of the second statement involves lengthy linear algebra computations with the help of Zariski decomposition. We will not go down this road here, a detailed proof can
be found in \cite{BKS04}*{Proof of Proposition 1.8}

We will prove the first one here; it gives  a strong generalization of the Campana--Peternell theorem on the structure of the nef boundary.

 Denote by $\I(X)$ the set of all irreducible curves on $X$ with negative self-inter\-section.
   Note that if $D$ is a big divisor, then by the Hodge index theorem
   we have $\Null(D)\subset\I(X)$.
   For $C\in\I(X)$ denote
\[
 C\nonneg \deq \set{D\in\NR(X)\with D\cdot C\ge 0}\ ,
\]
and 
\[
C\orth \deq  \set{D\in\NR(X)\with D\cdot C=0}\ .
\]

\begin{prop}[\cite{BKS04}]
 The intersection of the nef cone and the big cone is locally rational
   polyhedral, that is,  for every $\RR$-divisor
   $P\in\Nef(X)\cap\BigCone(X)$ there exists a neighborhood
   $U$ and curves $C_1,\dots,C_k\in\shi(X)$ such that
   \[
      U\cap\Nef(X) \equ U\cap\( C_1\nonneg\cap\dots\cap C_k\nonneg\)
   \] 
\end{prop}

\begin{proof}
The key observation is that  given  a big and nef $\RR$-divisor $P$ on $X$, there exists a
   neighborhood $\shu$ of $P$ in $\NR(X)$ such that for all divisors
   $D\in \shu$ one has
   $$
      \Null(D)\subset\Null(P) \ .
   $$

 As the big cone is open, we may
   choose big (and effective) $\RR$-divisors $D_1,\dots,D_r$ such that
   $P$ lies in the interior of the cone
   $\sum_{i=1}^r\RR^+D_i$.

 We can have $D_i\cdot C<0$ only for finitely many curves
   $C$. Therefore,
   after possibly replacing $D_i$ with $\eta D_i$ for some small
   $\eta>0$, we can
   assume that
   $$
      (P+D_i)\cdot C\,>\, 0    \eqno (*)
   $$
   for all curves $C$ with $P\cdot C>0$.
   We conclude then from $(*)$ that
   $$
      \Null\(\sum_{i=1}^r \alpha_i (P+D_i)\)\, \subset\, \Null(P)
   $$
   for any $\alpha_i>0$. So the cone
   $$
      \shu\equ\sum_{i=1}^r \RR^+(P+D_i)
   $$
   is a neighborhood of $P$ with the desired property.

   Let now $\shu$ be a neighborhood of $P$ as above . Observe that 
   $$
      \BigCone(X)\cap\Nef(X) \equ \BigCone(X)\cap\bigcap_{C\in\I(X)}C\nonneg
   $$
   and therefore
   $$
      \shu\cap\Nef(X) \equ \shu\cap\bigcap_{C\in\I(X)}C\nonneg \ .
      \eqno(*)
   $$
   For every $C\in\I(X)$ we have either $\shu\subset C\nonneg$, in
   which case we may safely omit $C\nonneg$ from the intersection in
   $(*)$, or else $\shu\cap C\orth\ne\emptyset$.
   By our choice of $\shu$, the second option
   can only happen for finitely many curves $C$.
   
In fact, $$\shu\cap\Nef(X) \equ  \shu\cap\bigcap_{C\in\Null(P)}C\nonneg \ .$$
\end{proof}

A fun application of Theorem~\ref{thm:loc poly surfaces} is the continuity of Zariski decomposition inside the big cone. 

\begin{cor}[\cite{BKS04}]
  Let $(D_n)$ be a sequence of big divisors converging in
   $\NR(X)$ to a big divisor $D$.
   If $D_n=P_n+N_n$ is the Zariski decomposition of $D_n$, and
   if $D=P+N$ is the Zariski decomposition of $D$, then
   the sequences $(P_n)$ and $(N_n)$ converge to $P$ and $N$
   respectively.
\end{cor}
\begin{proof}
  We consider first the case where all $D_n$ lie in a fixed
   chamber $\Sigma_P$. In that case we have by definition
   $\Neg(D_n)=\Null(P)$
   for all $n$, so that
   $$
      N_n\in\vecspan{\Null(P)}
   $$
   and hence $P_n\in\Null(P)\orth$.
   As
   $$
      \NR(X)=\Null(P)\orth\oplus\vecspan{\Null(P)}
   $$
   we find that both sequences $(P_n)$ and $(N_n)$ are
   convergent.
   The limit class $\lim P_n$ is certainly nef.
   Let $E_1,\dots,E_m$ be the curves in $\Null(P)$. Then
   every $N_n$ is of the form $\sum_{i=1}^m a_i^{(n)}E_i$ with
   $a_i^{(n)}>0$. Since the $E_i$ are
   numerically independent, it follows that $\lim N_n$ is
   of the form $\sum_{i=1}^m a_iE_i$ with $a_i\ge 0$, and hence
   is either negative definite or zero. Therefore
   $D=\lim P_n+\lim N_n$ is actually the Zariski decomposition of
   $D$, and by uniqueness the claim is proved.

   Consider now the general case where the $D_n$ might lie in
   various chambers. Since the decomposition into chambers is
   locally finite, there is a neighborhood of $D$ meeting only
   finitely many of them. Thus there are finitely many
   big and nef divisors $P_1,\dots,P_r$ such that
   $$
      D_n\in\bigcup_{i=1}^r\Sigma_{P_i}
   $$
   for all $n$.
   So we may decompose the sequence $(D_n)$ into finitely many
   subsequences to which the case above applies.
\end{proof}

\begin{exer}
Prove that Zariski decomposition is continuous on $\oEff(X)$ provided there exist only finitely many negative curves on $X$. Find a counterexample to the continuity of Zariski decomposition on $\oEff(X)$ in general.
\end{exer}

The connection between Zariski decompositions and 
asymptotic cohomological functions comes from the following result.

\begin{prop}[Section 2.3.C., \cite{PAGI}]
   Let $D$ be a big integral divisor, $D=P_D+N_D$ the Zariski decomposition of $D$. Then
   \begin{items}
      \item[(i)] $\HH{0}{X}{kD}=\HH{0}{X}{kP_D}$ for all $k\geq 1$ such that $kP_D$ is integral, and
      \item[(ii)] $ \vl{X}{D}=\vl{X}{P_D}=\zj{P_D^2}.$
   \end{items}
\end{prop}

By  homogeneity and continuity of the volume we obtain that
for an arbitrary big $\RR$-divisor $D$ with Zariski decomposition 
$D=P_D+N_D$ we have $  {\rm vol}(D)=\zj{P_D^2}=\zj{D-N_D}^2$.

Let $D$ be an $\RR$-divisor on $X$. In determining the 
asymptotic cohomological functions on $X$, we distinguish 
three  cases,  according to whether $D$ 
is pseudo-effective, $-D$ is pseudo-effective or none. 

\begin{prop}\label{asymptotic for pseff}
With notation as above, if $D$ is pseudo-effective then
\[
\ha{i}{X}{D} \equ \begin{cases} (P_D^2) & \textrm{ if } i=0 \\
                             -(N_D^2) & \textrm{ if } i=1 \\
                             0 & \textrm{ if } i=2\ .
               \end{cases}
\]
If $-D$ is pseudo-effective with Zariski decomposition $-D=P_D+N_D$ then
\[
\ha{i}{X}{D} \equ \begin{cases} 0 & \textrm{ if } i=0 \\
                             -(N_{-D}^2) & \textrm{ if } i=1 \\
                             (P_{-D}^2) & \textrm{ if } i=2\ .
               \end{cases}              
\]
When neither $D$ nor  $-D$ are pseudo-effective,  one has 
\[ 
\ha{i}{X}{D} \equ  \begin{cases} 0 & \textrm { if } i=0 \\
                            -(D^2) & \textrm{ if } i=1 \\
                            0 & \textrm{ if } i=2 \ .
              \end{cases}              
\]
\end{prop}

\begin{proof}   
We treat the case of pseudo-effective divisors in detail, the case $-D$ pseudo-effective follows from Serre duality for asymptotic cohomology, while the third instance
is immediate from Serre duality, asymptotic Riemann--Roch, and the fact that non-big divisors have zero volume.

Let us henceforth assume that $D$ is pseudo-effective.  
If $D=P_D+N_D$ is the Zariski decomposition of the 
pseudo-effective   divisor $D$, then $\ha{0}{X}{D}=\zj{P_D^2}$.
Furthermore,   if $D$ is pseudo-effective then  $\ha{2}{X}{D}=0$.
In order to compute $\hat{h}^1$, consider the equality
\[
\hh{1}{X}{mD} \equ \hh{0}{X}{mD}+\hh{2}{X}{mD}-\euler{X}{mD}\ .
\]
This implies that 
\[
\ha{1}{X}{D} \equ \limsup_m{ \zj{
\frac{\hh{0}{X}{mD}}{m^2/2}+\frac{\hh{2}{X}{mD}}{m^2/2}
-\frac{\euler{X}{mD}}{m^2/2} } }\ .
\]
All three sequences on the right-hand side are convergent. The $h^0$
sequence by the fact that the volume function is in general a limit. The 
$h^2$ sequence converges by  $\ha{2}{X}{D}=0$. Finally, the convergence 
of the  sequence of Euler  characteristics follows from  the 
Asymptotic Riemann--Roch theorem. 
Therefore the $\limsup$ on the right-hand side is  a limit, and 
$\ha{1}{X}{D}=  - (N_D^2)$.
\end{proof}

Applying Theorem~\ref{thm:loc poly surfaces} we arrive at the following.

\begin{thm}
With notation as above, there exists a   locally finite decomposition of 
 $\Bbig(X)$ into rational locally polyhedral subcones such that on each of 
those the asymptotic cohomological functions are given by a single 
homogeneous quadratic polynomial.
\end{thm}

We illustrate the theory outlined above by  determining  the volume functions of del Pezzo surfaces. Note that by Proposition~\ref{asymptotic for pseff} this 
will then provide an answer for all asymptotic cohomology functions. 

  Let us establish some notation. We denote by $X=Bl_{\Sigma}({\PP}^2)$ the
   blow-up
   of the projective plane at $\Sigma \subseteq {\PP}^2$ where $\Sigma$ consists
   of at most eight points in general position. The exceptional divisors
   corresponding to the points in $\Sigma$ are denoted by $E_1,\dots ,E_r \
   (r\leq 8)$. We denote the pullback of the hyperplane class on ${\PP}^2$ by $L$.
   These divisor classes generate the Picard group of $X$
   and their intersection numbers are
   $L^2=1$, $(L.E_i)=0$ and $(E_i.E_j)=-{\delta}_{ij}$ for $1\leq i,j\leq r$.
   For each $1\leq r\leq 8$ one  can describe explicitly all extremal rays on $X$ 

\begin{prop}
   With notation as above, the set $\left\{ E^{\perp} | E\in
   {\II} \right\}$ determines the chambers for the volume function.
   More precisely, we obtain the chambers by dividing the big cone
   into finitely many parts by the hyperplanes $E^{\perp}$.
\end{prop}

\begin{proof}
   Observe that as the only negative curves on a del Pezzo surface are $(-1)$-curves, 
   the support of every negative divisor consists of pairwise orthogonal curves. 
   This can be seen as follows.  Take a negative divisor $N=\sum_{i=1}^{m}{a_iN_i}$.  
   Then, as the self-intersection matrix of $N$ is negative definite,  for any $1\leq i<j\leq m$ one has
   \[
       0 > (N_i+N_j)^2 = N_i^2+2(N_i\cdot N_j)+N_j^2 = 2(N_i\cdot N_j)-2 \ .
   \]
   As  $N_i\cdot N_j\geq 0$, this can only hold if $N_i\cdot N_j=0$.

   According to \cite{BKS04}*{Proposition 1.5}, a big divisor $D$ is in the boundary 
   of a Zariski chamber if and only if
   \[
   \Neg(D)\neq \Null(P_D)\ .
   \]
   From \cite{BKS04}*{4.3} we see that if $C\in\Null(P_D)-\Neg(D)$ for an 
   irreducible negative curve $C$ then $N_D+C$ forms a negative divisor. 
   By the previous reasoning, this implies that $N_D\cdot C=0$ hence $D\cdot C=0$, 
   that is, $D\in
   C\orth$ as required.

   Going the other way, if $D\in C\orth$ for an irreducible negative curve $C$ 
   then  either $P_D\cdot C=0$, that is, $C\in\Null(P_D)$ or $P_D\cdot C>0$.

   In the first case, $C\not\in\Neg(D)$, as otherwise we would have $N_D\cdot C<0$ 
   and consequently $D\cdot C<0$ contradicting $D\in C\orth$. Therefore $C\in \Null(P_D)-\Neg(D)$ 
   and $D$ is in the boundary of some Zariski chamber.

   In the second case, $D\cdot C=0$ and $P_D\cdot C>0$ imply $N_D\cdot C<0$. 
   From this we see  that $C\in \Neg(D)$ but this would mean $P_D\cdot C=0$ which is again a contradiction.

   The conclusion is that on a surface on which the only negative curves 
   are $(-1)$-curves, a big divisor $D$ is in the boundary of a Zariski chamber 
   if and only if there exists an $(-1)$-curve $C$ with $D\in C\orth$.
\end{proof}

One can in fact ask for more, and enumerate all the Zariski chambers on a given del Pezzo surface. The number of Zariski chambers (defined in \cite{BFN}) 
   $$
      z(X)=\#\set{\mbox{Zariski chambers on $X$}}\in\NN\cup\set\infty 
   $$
on a surface $X$ measures how complicated $X$ is from the point of view of linear series. At the same time it tells us 
\begin{itemize} 
   \item the number of different stable base loci of big divisors on $X$;
      \item the number of different possibilities for the support of the negative part of the Zariski decomposition of a big divisor;
     \item the number of maximal regions of polynomiality for the volume function.
 \end{itemize}
  
Bauer--Funke--Neumann in \cite{BFN} calculate $z(X)$ when $X$ is a del Pezzo surface. 

\begin{thm}
   Let $X_r$ be the blow-up of $\PP^2$ in $r$ general points with
   $1\le r\le 8$.
   \begin{enumerate}
   \item[(i)]
   The number $z(X_r)$ of Zariski chambers on $X_r$ is given by
   the following table:
   $$
      \begin{array}{c|*8c} \hline 
         r        & 1 & 2 & 3 & 4 & 5 & 6 & 7 & 8  \\ \hline  
         z(X_r)   & 2 & 5 & 18 & 76 & 393 & 2\,764 & 33\,645 & 1\,501\,681
      \end{array}
   $$

   \item[(ii)]
   The maximal number of curves that occur in the support of a
   Zariski chamber on $X_r$ is $r$.
   \end{enumerate}
\end{thm}

\subsection{Newton--Okounkov polygons}

Here we discuss how to determine Newton--Okounkov bodies on surfaces using variation of Zariski decomposition. The locally finite rational polyhedral nature of Zariski chamber decomposition already hints at the possibility of obtaining Newton--Okounkov bodies with locally rational polygonal structure, but Fujita's extension of Zariski decomposition to pseudo-effective $\RR$-divisors will yield that Newton--Okounkov bodies on surfaces are in fact polygons.

Let $X$ be a smooth projective surface, $D$ a big Cartier divisor on $X$. An admissible flag $\ybul$ on $X$ consists now 
of three elements $X=Y_0\supset Y_1\supset Y_2$, where admissibility implies that 
\[
Y_1\equ C \text{ irreducible curve }\ \ ,\ \ Y_2\equ \{x\}\in C \text{ a smooth point.}
\]
We will often write $\DCx{D}$ for $\nob{\ybul}{D}$ or even $\Delta_C(D)$ if the point $x\in X$ is either irrelevant or taken to be very general. 

Before launching into the description of Newton--Okounkov bodies on surfaces, we establish some notation. 

\begin{defn}
Let $\delta$ and $\gamma$ be big $\RR$-divisors classes on $X$, assume that $\delta$ is big. We set 
\[
\mu_\gamma(\delta) \deq \sup\st{t>0\mid \delta-t\gamma \text{ is big}}\ .
\]
If $D$ and $C$ are divisors, we write $\mu_C(D)$ for $\mu_{[C]}([D])$. 
\end{defn}

One typically takes $C$ to be an effective divisor or even an irreducible curve; $\mu_C(D)$ measures the distance from the boundary of the pseudo-effective cone in the direction of $-C$. Note the analogy with Seshadri constants. The invariant 
$\mu_C(D)$ is very interesting in itself, its rationality is a very tricky question. 

The starting point of our discussion is the following fundamental result. 

\begin{thm}[Lazarafeld--Musta\c t\u a, \cite{LM}, Theorem 6.4]\label{thm:LM surfaces}
Let $X$ be a smooth projective surface, $D$ a big divisor (or more generally, a big $\RR$-divisor class) $\ybul=(C,x)$ an admissible flag on $X$. Then there exist continuous functions $\alpha,\beta\colon [\nu,\mu]\to \RR_{\geq 0}$ such that 
$0\leq \nu\leq \mu \deq \mu_C(D)$ are real numbers, 
\begin{itemize}
	\item[-] $\nu = $ the coefficient of $C$ in $N_D$,
	\item[-] $\alpha(t)\equ \ord_x N_{D-tC}|_C$,
	\item[-] $\beta(t)\equ \alpha(t) + (P_{D-tC}\cdot C),$
\end{itemize}
and 
\[
\DCx{D} \equ \st{ (t,y)\in \RR^2\mid \nu\leq t\leq \mu\, ,\, \alpha(t)\leq y\leq \beta(t)}\ .
\]
\end{thm}

\begin{rmk}
As defined in the previous section, $P_{D-tC}$ and $N_{D-tC}$ stand for the positive and negative parts of the divisor $D-tC$, respectively. 
\end{rmk}

\begin{rmk}
In addition, it has been observed in \cite{LM} that the functions $\alpha(t)$ and $\beta(t)$ are piecewise affine linear with rational coefficients on every interval $[\nu,\mu']$ with $\mu'<\mu$. As a consequence the intersection of $\nob{\ybul}{D}$ with the $[\nu,\mu']\times \RR$ is a rational polygon. 

Later in Theorem\ref{thm:NO-polygons} we will see that somewhat more is true: Newton--Okounkov bodies are surfaces are polygons that are very close to being rational.  
\end{rmk}

\begin{rmk}
The proof of Theorem~\ref{thm:LM surfaces} relies on studying restricted linear series on curves, while the qualitative results about Newton--Okounkov bodies come from variation of Zariski decomposition \cite{BKS04}. 
\end{rmk}

\begin{rmk}
Note that $\Supp N_{D-tC}$ consists of finitely many irreducible curves for any given $t$, and does not contain $C$ for $t>\nu$. By \cite{BKS04}*{???}, the cardinality  of curves occurring in $N_{D-tC}$ for all $t$ is still countable, hence 
\[
\alpha(t) \deq \ord_x N_{D-tC}|_C \,\equiv\, 0
\]
$x$ does not lie on any of them. Therefore $\alpha(t)\equiv 0$ if $x\in C$ is very general. If there exist only finitely many negative curves on $X$, then a general point will have the same property. 
\end{rmk}

\begin{lem}\label{lem:missing lemma}
Let $X$ be a smooth projective surface, $D$ a big divisor, $C$ an irreducible curve on $X$. If $C\not\subseteq \Bminus(D)$, then $C\not\subseteq \Bplus(D-tC)$ for all $0<t<\mu_C(D)$. 
\end{lem}
\begin{proof}
Recall that $\Bminus(D)=\Neg(D)$, and $\Bplus(D-tC)=\Null(D-tC)$. Let us first assume that $(C^2)<0$. By Lemma~\cite{KLM1} we have $C\not\subseteq \Bminus(D-tC)$. But then $P_{D-tC}+aC+\sum_{i}a_iE_i$ is the Zariski decomposition of $D-tC$, and 
the statement follows.

Suppose now that $(C^2)=0$, and $t\geq 0$. Then 
\[
\Bplus(D-tC) \equ \Bplus(P_{D-tC})\cup \Supp N_{D-tC}\ ,
\]
and since $C$ is not a negative curve, it is not contained in $\Supp N_{D-tC}$. The divisor $P_{D-tC}$ is big and nef, therefore 
\[
\Bplus(D-tC) \equ \st{\Gamma\mid (P_{D-tC}\cdot \Gamma)=0 }\ .
\]
Again by $P_{D-tC}$ being big and nef, $(P^2_{D-tC})>0$, hence by the Hodge index theorem $(C^2)<0$ holds whenever 
$(P_{D-tC}\cdot C)=0$. But this would contradict our assumption on $(C^2)$.
\end{proof}

In the following we explain how to use variation of Zariski decomposition to 

\begin{eg}
Let $X$ be the blow-up of the projective plane at a point $P$, then the N\'eron--Severi space of $X$  has the basis consisting of $H$, the pullback of the hyperplane class, and $E$, the class of the exceptional divisor. The intersection form is given by
\[
(H^2)\equ 1\ ,\ (H\cdot E)\equ 0\ ,\ (E^2) \equ -1\ .
\]
The (pseudo-)effective cone of $X$ is spanned by $E$ and $H-E$, while the nef cone of $X$ is spanned by the classes $H$ and $H-E$. 

Let $L$ be the ample divisor  $2H-E$, $C\deq H-E$ (the proper transform of a line through the point $P$), and $x\in C$ a point not lying on $E$. Then $\alpha(t)\equiv 0$, and we will only need to worry about computing $\beta(t)$. 

Let us compute how long $L-tC$ stays in the pseudo-effective cone. 
\begin{eqnarray*}
\mu_C(L) & \deq & \sup\st{ t>0\mid L-tC \text{ is big}} \\
& = & \sup\st{ t>0\mid ((L-tC)\cdot H)  \text{ and } ((L-tC)\cdot (H-E))}
\end{eqnarray*}
as the classes $H$ and $H-E$ span the nef cone of $X$. As $L-tC=(2-t)H+(t-1)E$, we see that 
\[
\mu_C(L) \equ \sup\st{2,3} \equ 2\ . 
\]
The essential information we need is the Zariski decomposition of divisors along the ray $L-tC$ for $t>0$. By the openness of amplitude in the N\'eron--Severi space, $L-tC$ is an ample $\RR$-divisor class for small $t>0$. As long as $L-tC$ is ample, we have $P_{L-tC}=L-tC$, therefore 
\[
\beta(t) \deq (P_{L-tC}\cdot C) \equ (L\cdot C) - t(C^2)\ .
\]
Taking into account that $(C^2)=((H-E)^2) = 0$, we obtain
\[
\beta(t) \,\equiv\, (L\cdot C)  \equ   ((2H-E)\cdot(H-E)) \equ 1\ . 
\] 
By the Kleiman--Nakai--Moishezon criterion the fact that $E$ is the only negative curve on $X$ implies that $L-tC$ is nef 
if and only if $((L-tC)\cdot E) \geq 0$. As 
\[
((L-tC)\cdot E) \equ (((2H-E)-t(H-E))\cdot E) \equ (((2-t)H+(t-1)E)\cdot E) \equ 1-t\ ,
\]
we see that $L-tC$ is nef precisely when $0\leq t\leq 1$. 

For values $1\leq t\leq 2$, the negative part of $L-tC$ will be supported on $E$, this being the only non-trivial possibility for the support of a negative cycle.  We are now in luck since 
\[
L-tC \equ \underbrace{(2-t)H}_{P_{L-tC}} \ +\ \underbrace{(t-1)E}_{N_{L-tC}} 
\]
satisfies the conditions for being a Zariski decomposition, by uniqueness this is the only one. Hence, over the interval 
$1\leq t\leq 2$ we obtain 
\[
\beta(t) \equ ((2-t)H \cdot (H-E)) \equ 2-t\ .
\]
\end{eg}

\begin{exer}
Let $X$ be the blow-up of $\PP^2$ at two distinct point $P$ and $Q$. Using Example~\ref{eg:two_points}, compute examples of Newton--Okounkov bodies, and try to find ones with as many vertices as possible. 
\end{exer}

\begin{rmk}
The moral of the story is that our ability to determine Newton--Okounkov bodies on surfaces relies on our knowledge about the variation of Zariski decomposition of divisors. This in turn depends on knowing the set of negative curves on the surface $X$ and the intersection form. 

Note however that one does not need to know how Zariski decomposition varies over the whole of $\Bbig(X)$, only along the line segment $[L,L-\mu_C(L)] \subseteq \Bbig(X)$. This observation will turn out to be essential for determining shapes of Newton--Okounkov bodies in dimension two precisely. 
\end{rmk}

Up until now all examples we have seen were rational polygons. We will now see that this is not far from the truth. 

\begin{thm}[\cite{KLM1}, Theorem B]\label{thm:NO-polygons}
The Newton--Okounkov body of a big $\QQ$-divisor $L$ on a smooth projective surface is a polygon with rational slopes. 
All vertices not lying on the vertical line $t=\mu_C(L)$ are rational. 	
\end{thm}

\begin{corollary}\label{cor:rational polygon}
Let $X$ be a smooth projective surface, $L$ a big $\QQ$-divisor, $\ybul=(C,x)$ an admissible flag on $X$. Then $\dyl$ is a rational polygon if and only if $\mu_C(L)\in\QQ$.  
\end{corollary}

\begin{rmk}
The idea of the proof is to study the variation of Zariski decomposition along the line segment $[L,L-\mu_C(L)]$. At this point it becomes crucial that Zariski decomposition exists for pseudo-effective $\RR$-divisors, since the support of (the a priori only pseudo-effective $\RR$-divisor class) $N_{L-\mu_C(L)}$ acts as a natural upper bound for the negative parts of 
all the divisors $L-tC$. The other key observation is that $N_{L-tC}$ is in fact increasing with $t$.
\end{rmk}

The technical core of the proof of Theorem~\ref{thm:NO-polygons} is the following statement about negative parts of divisors. Beforehand, we introduce some notation. Set $\mu\deq\mu_C(L)$ and $L'\deq L-\mu C$; the divisor $D'$ is then pseudo-effective but not big by the definition of $\mu$. For any $t\in [\nu,\mu]$ we write $s\deq \mu-t$, and define 
\[
L_s' \deq L' + sC \equ L' + (t-\mu)C \equ L-tC\ .
\]
We will consider the line segment $\st{L_t\mid \nu\leq t\leq \mu}$ in the form 
\[
\st{L_s'\mid 0\leq s\leq \mu-\nu}\ .
\] 
Let 
\[
D_s' \equ P_s' \ +\ N_s'
\] 
be the Zariski decomposition of $D_s'$. 
 
\begin{prop}[\cite{KLM1}*{Proposition 2.1}]\label{prop:key result}
With notation as above, the function $s\mapsto N_s'$ is decreasing along the interval $s\in [0,\mu-\nu]$, that is, for each
$0\leq s'<s\leq \mu-\nu$ the divisor $N_s-N_s'$ is effective. 

If $n$ is the number of irreducible components of $N_0'$, then there is a partition $(p_i)_{0\leq i\leq k}$ of the interval $[0,\mu-\nu]$ for some $k\leq n$, and there exist divisors $A_i$ and $B_i$ with the latter having rational coefficients such that  
\[
N_s' \equ A_i +sB_i\ \ \ \text{for all $s\in [p_i,p_{i+1}]$.}
\]	
\end{prop}

\begin{proof}
 Let $C_1,\ldots, C_n$ be the irreducible components of $\Supp (N')$, where 
$N'=N'_0$. Choose real numbers $s',s$ such that $0\leq s'< s\leq \mu-\nu$. We 
can then write
\[ 
P'_{s'} \equ D'_{s'}-N'_{s'} \equ  (D'_s-(s-s')C)-N'_{s'} \equ D'_s - 
((s-s')C+N'_{s'}).
\] 
As $P'_{s'}$ is nef and the negative part of the Zariski decomposition is 
minimal, the divisor $(s-s')C+N'_{s'}-N'_s$ is effective and it remains only to show that $C$
is not in the support of $N'_s$ for any $s\in [0, \mu-\nu]$. If $C$ were in 
the support of $N'_s$ for some $s$, then for any $\lambda >0$ the Zariski 
decomposition of $D'_{s+\lambda}$ would be $D'_{s+\lambda}=P'_s+(N'_s+\lambda C)$.
In particular, $C$ would be in the support of  $N'_{\mu-\nu}$, contradicting  
the definition of $\nu$. 

Rearranging the $C_i$'s, we assume that the support of $N'_{\mu-\nu}$ consists of 
$C_{k+1}, \ldots, C_n$. We set 
\[
p_i \deq  \sup \{ s\ | \ C_i\subseteq \Supp(N'_s)\} \textup{ for all } i=1\ldots k\ .
\]
Without loss of generality $0< p_1\leq \ldots \leq p_{k-1}\leq p_k\leq \mu 
-\nu$. We will show that $N'_s$ is linear on $[p_i, p_{i+1}]$ for this choice of
$p_i$'s. By the continuity of the Zariski decomposition (see \cite{BKS04}*{Proposition 1.14}), it 
is enough to show that $N'_s$ is linear on the open interval $(p_i,p_{i+1})$. If
$s\in (p_i, p_{i+1})$ then the support of $N'_s$ is contained in 
$\{C_{i+1},\ldots ,C_{n}\}$, and $N'_s$ is determined uniquely by the equations
\[
N'_s\cdot C_j \equ (D'+sC)\cdot C_j, \textup{ for } i+1\leq j\leq n\ .
\] 
As the intersection matrix of the curves $C_{i+1},\ldots ,C_{n}$ is 
non-degenerate, there are unique divisors  $A_i$ and $B_i$ supported on 
$\cup_{l=i+1}^n C_l$ such that
\[
A_i\cdot C_j \equ D'\cdot C_j \textup{ and } B_i\cdot C_j= C\cdot C_j \textup{ for all }i+1\leq j\leq n\ .
\] 
Note that $B_i$ is a rational divisor. By \cite{BL}*{Lemma 14.9} (first proved
in \cite{Zar}) both $A_i$ and $B_i$ are effective and it follows that 
$N'_s= A_i+s B_i$ for any $s\in (p_i, p_{i+1})$.
\end{proof}

\begin{proof}[Proof of Theorem~\ref{thm:NO-polygons}]
To begin with \cite{LM}*{Theorem 6.4} implies that the function $\alpha$  is convex, $\beta$  is concave and $\alpha\leq \beta$ all the time. 
Next,  Proposition~\ref{prop:key result} yields that the functions $\alpha$ and $\beta$ are piecewise linear with only finitely many breakpoints which occur
when we cross walls of Zariski chambers, hence for rational values of $t$. 

Finally, $\alpha$ is an increasing function of $t$ by Proposition~\ref{prop:key result}, because $N_t=N_{\mu-s}'$, and $\alpha(t)= \ord_x(N_t|_X)$, therefore we are done. 
\end{proof}

\begin{exer}
Let $X$ be a smooth projective surface, $L$ a big $\RR$-divisor, $\ybul$ an admissible flag on $X$. Show that the number of vertices of the polygon $\dyl$ is at most $2\rho(X)+2$. Verify that this bound is sharp. 
\end{exer}

\begin{conjecture}(who do I attribute this to?)\label{conj:NO and Picard number}
Let $X$ be a smooth projective surface, $L$ a big $\RR$-divisor. Let $f\colon \tilde{X}\to X$ be a proper birational morphism, $\ybul$ an admissible flag on $\tilde{X}$.  Then $\nob{\ybul}{f^*L}$ has at most $2\rho(X)+2$ vertices.  
\end{conjecture}	

The essential information one needs to answer questions like Conjecture~\ref{conj:NO and Picard number} is how many Zariski chambers the line segment $[L,L-\mu_C(L)C]$ crosses. This kind of information is on the other hand hard to obtain in general, they are very closely related to the rationality of Seshadri constants.   As an illustration, here is an equivalent formulation of Nagata's conjecture. 

\begin{conjecture}[Nagata]
Let $X$ be the blow-up of $\PP^2$ at a set of $n$ general points, let $H$ be the hyperplane class, $\E\deq E_1+\ldots+E_n$ the exceptional divisor. Then the ray $H-t \E$ ($t\geq 0$) only crosses one Zariski chamber before it leaves $\oEff(X)$. 
\end{conjecture}

Surprisingly enough, it is not difficult to show that the ray $H-t\E$ in the Nagata situation only goes through at most two chambers before leaving $\oEff(X)$.

\begin{proposition}[\cite{DKMS}*{Proposition 2.5}]\label{prop:only 2 chambers}
	Let $f:X\to\PP^2$ be the blowup of $\PP^2$ at $n$ general points with
	exceptional divisor $\E\deq E_1+\ldots+E_n$.
	Let $H$ be the pull-back of the hyperplane bundle, then the ray $R=H-t\E$ meets at most two Zariski chambers on $X$.
\end{proposition}
\begin{proof}
If the multi-point Seshadri constant $\eps \deq \eps(\calo_{\P^2}(1);n)$ is maximal, that is, equal to $\frac1{\sqrt{n}}$, then the ray crosses only the nef cone.

If $\eps$ is submaximal, then there is a curve $C=dH-\sum m_iE_i$
computing this Seshadri constant, i.e. $\eps=\frac{d}{\sum m_i}$.

If this curve is homogeneous,
i.e. $m=m_1=\dots=m_s$, then we put $\Gamma\deq C$.
Otherwise we define $\Gamma=\gamma H-M\cdot\E$
as a symmetrization of $C$ as in  \cite{DKMS}*{Lemma 2.2}, that is,  we sum the curves $C_k=dH-\sum m_{\sigma^k(i)}E_i$
over a length $n$ cycle $\sigma\in\Sigma_n$.

Let $\mu=\frac{M}{\gamma}$. Note that $\mu=\mu(L;\E)$.
Indeed, it is in any case $\mu\leq\mu(L;\E)$ by the construction
of $\Gamma$ and a strict inequality would contradict the fact that
$\Gamma$ computes the Seshadri constant.

Now, we claim that
$$H-t\E=\frac{\mu-t}{\mu-\eps}(H-\eps\E)+\frac{t-\eps}{\gamma(\mu-\eps)}\Gamma$$
is the Zariski decomposition of $H-t\E$ for $\eps\leq\, t\leq\, \mu$. Indeed,
$H-\eps\E$ is nef by definition and it is orthogonal to all components of $\Gamma$,
which together with the Index Theorem implies that the intersection matrix of $\Gamma$
is negative definite.

Thus the ray $R$ after leaving the nef chamber remains in a single Zariski chamber.
\end{proof}

At the other extreme, one can quickly come up with examples with rays  meeting a maximal number of chambers.

\begin{example}[\cite{DKMS}*{Example 2.6}]
	Keeping the notation from Proposition \ref{prop:only 2 chambers}
	let $L=(\frac{n(n+1)}{2}+1)H-E_1-2E_2-\ldots-nE_n$. Then $L$ is an ample divisor
	on $X$ and the ray $R=L+\lambda\E$ crosses $s+1=\rho(X)$ Zariski chambers.
	Indeed, with $\lambda$ growing, exceptional divisors $E_1, E_2,\ldots, E_n$
	join the support of the Zariski decomposition of $L-\lambda\E$ one by one.
 \end{example}

Although the rays $H-t\E$ in the Nagata example are not of the form $L-tC$ for an irreducible curve $C$, together with 
Theorem~\ref{thm:LM surfaces}  they inspire the following definition.

\begin{definition}[Numerical Newton--Okounkov bodies on surfaces]
	Let $X$ be a smooth projective surface, $L$ a big $\RR$-divisor, $\alpha\in \NR$ an arbitrary class. We set 
	\[
	\Delta^{\text{num}}_\alpha(L) \deq \st{(t,y)\in\RR^2\mid 0\leq t\leq \mu_\alpha(L)\, ,\, 0\leq y\leq  (P_{L-t\alpha}\cdot\alpha)}\ ,
	\]
	where $\mu_\alpha(L) \deq \sup\st{t>0\mid L-t\alpha\text{ is big}}$. 
\end{definition}

\begin{exer} With notation as above,
	\begin{enumerate}
		\item if $-\alpha$ is not pseudo-effective, then $\Delta^{\text{num}}_\alpha(L)$ is a convex polygon.
		\item The construction is homogeneous respects numerical equivalence of divisors.
		\item Explore the continuity properties of $\Delta^{\text{num}}_\alpha(L)$ in $\alpha$ and $L$.
		\item For which $\alpha$ is it true that $\vol{\RR^2}{\Delta^{\text{num}}_\alpha(L) }=2\cdot \vol{X}{L}$? 
	\end{enumerate}
\end{exer}

\subsection{Rationality questions regarding Newton--Okounkov polygons}

Next, we look at the question what kind of a number can   $\mu_C(L)$ can be. Let us see first that it is not always rational. 

\begin{eg}(Taken from \cite{LM})
Let $X=Y\times Y$ be an abelian surface where $Y$ is an elliptic curve without complex multiplicaton, then the classes $F_1,F_2$ of the fibres of the two projections and the diagonal form a basis of  $N^1(X)_\RR$. The nef and pseudo-effective cones coincide, and are described by 
\[
\oEff(X) \equ \st{\alpha\in N^1(X)\mid (\alpha^2) \geq 0\, ,\, (\alpha\cdot H)\geq 0 }
\]
for some (or, equivalently, all) ample classes $H$. The source of the irrationality is the fact that equalities of the form $((L-tC)^2)=0$ tend to have irrational roots. 

Let $L$ be an ample Cartier divisor, $C\subseteq X$ and ample curve, and $x\in X$ an arbitrary point. Then $\alpha(t)\equiv 0$, for $X$ has no negative curves at all. 

Then 
\[
\mu_C(L) \equ \text{ the smaller root of the polynomial $((L-tC)^2)$.}
\]
For most choices of $L$ and $C$, $\mu_C(L)\not\in\QQ$. Based on this information, one can quickly compute that 
\[
\beta(t) \equ (L\cdot C) - t(C^2)\ .  
\]
\end{eg}

\begin{thm}[\cite{KLM1}]
With notation as above, 
\begin{enumerate}
	\item $\mu_C(L)\in \QQ$ or it satisfies a quadratic equation over $\QQ$.
	\item If $\alpha\in\RR_{>0}$ satisfies a quadratic equation and it is the smaller of the two roots, then $\alpha=\mu_C(L)$ for suitable $X,L,C$. 
\end{enumerate}	
\end{thm}

\begin{proof}
We only prove the first statement and refer the reader for a verification of the second one to \cite{KLM1}. 
\end{proof}

\begin{exer}
Find a smooth projective surface $X$, an ample Cartier divisor $L$, and a curve $C$, such that $\mu_C(L)\not\in\QQ$, and 
$\dyl$ has an inner breakpoint. 	
\end{exer}

Although the following rationality result can be deduced from Corollary~\ref{cor:rational polygon}, it is educational to see a stand-alone proof.

\begin{prop}
All Okounkov bodies of big divisors on the nine-point blow-up of the
projective plane are rational polygones.
\end{prop}

\begin{proof}
Let $L$ be a big divisor on the nine-point blow-up $X$, $C$ an
irreducible curve on $X$, we take the point in the flag to be a very general
point on $C$.

With this setup we want to show that $\Delta\deq\Delta_C(L)$ is a rational
polygone. By the results of \cite{KLM1} $\Delta$ is certainly a
polygone, which is rational if and only if $\mu\deq\mu_C(L)$ is rational. This is
what we will prove now.

Let
\[
L-\mu C = P + N = P + \sum_{i=1}^{r}a_iE_i
\]
be the Zariski decomposition of the pseudo-effective (but not big!) divisor
$L-\mu C$.

\noindent
\textsc{Case 1:}  Suppose $N=0$. Then $L-\mu C$ is a nef but not big divisor, hence a
non-negative real multiple of $-K_X$: $P=b(-K_X)$. However, since $L$,$C$, and
$-K_X$ are all rational (in fact, integral) points in the N\'eron--Severi
space, $b$ must be rational as well, hence so i $\mu$.

Note: here we make very heavy use of the special nature of the situation: more
precisely , that we know that the only nef but not big divisor is integral
(namely $-K_X$). \\

\noindent\textsc{Case 2:} Suppose $P=0$, hence $L-\mu C=\sum_{i=1}a_iE_i$. In this case
$\mu$ is the 'time' when $L-tC$ hits the rational plane spanned by the
integral vectors $E_1,\dots,E_r$, hence it must be rational again. \\

\noindent\textsc{Case 3:}  Suppose that none of $P$ and $N$ are zero. Then again, since $L-\mu C$
is not big, neither is $P$, so again, $P=b(-K_X)$ for some \emph{positive} real
number $b$. However, $-K_X$ is \emph{not  perpendicular to any negative curve}
on $X$ (it has positive intersection with all of them), hence $N$ must be
zero, a contradiction.
\end{proof}  

\begin{prop}\label{prop:surface}
Let $X$ be a smooth projective surface and $L$ be a big and nef line bundle on $X$.  Then there exists an irreducible curve $Y_1\subseteq X$ such that the Okounkov body $\Delta_{Y_\bullet}(L)$ is a rational simplex, where $Y_2$ is a general point on the curve $Y_1$.
\end{prop}

\begin{proof}
If $L$ is semi-ample, this follows from Theorem~\ref{thm:rtl poly}.  We can therefore assume that $L$ is not semi-ample, i.e. the stable base locus $\textbf{B}(L)$ is nonempty.  
Let $C_1,\ldots ,C_r$ be the irreducible curves contained in $\textbf{B}(L)$.  Then there exist positive integers $m,a_1,\ldots ,a_r\in\ZZ$ such that the base locus of the divisor $mL-\sum_{i=1}^{i=r}a_iC_i$ 
is a finite set.  By the Zariski-Fujita theorem \cite{PAGI}*{Remark 2.1.32} this divisor is semi-ample.  Consequently, taking large enough multiples, we can assume that $mL-\sum_{i=1}^{i=r}a_iC_i$ is a base point-free divisor and $\textbf{B}(L)=C_1\cup\ldots \cup C_r$.  Taking this into account we choose $Y_1\in |mL-\sum_{i=1}^{i=r}a_iC_i|$ to be an irreducible curve.

The matrix $||(C_i.C_j)||_{i,j}$ is negative-definite.  To see this, first observe that $(L.C_i)=0$ for all $i=1,\ldots ,r$, because $\mathbf{B}(L)\subseteq\mathbf{B}_+(L)=\textup{Null}(L)$.  (The inclusion follows from \cite{PAGII}*{Definition 10.3.2}, and the equality from 
\cite{PAGII}*{Theorem 10.3.5}.)  Therefore any linear combination $D=\sum b_i C_i$ has $(L.D)=0$.  Since $(L^2)>0$ (as $L$ is big and nef), the Hodge Index Theorem implies $(D^2)<0$, which shows the matrix is negative-definite.  It follows that for any $b_i\in\mathbb{R}_+$, the divisor $\sum b_iC_i$ is pseudo-effective but not big.

Now consider the family of divisors $D_t = L-tY_1$ for all $t\geq 0$.  We claim that $D_t$ is effective iff $0\leq t\leq 1/m$. For this, notice that 
\[
 D_t \equ L-tY_1 \sim L-t(mL-\sum_{i=1}^{r}a_iC_i) \equ (1-mt)L + t\cdot\sum_{i=1}^{r}a_iC_i\ .
\]
By this description $D_t=L-tY_1$ is effective for $0\leq t\leq 1/m$. If $t>1/m$, set $\epsilon=t-1/m$. Since $(L\cdot C_i)=0$ for all $i=1,\ldots , r$, and $(L^2)>0$, then the intersection number
\[
(D_t\cdot L) \equ ( (-\epsilon L+ (1+\epsilon)\sum_{i=1}^{r}a_iC_i)\cdot L) \equ -\epsilon(L^2) < 0\ .
\]
Since $L$ is nef, this implies that $D_t$ is not effective when $t>1/m$.

 Above, we have seen that $D_t\sim (1-mt)L + t\cdot\sum_{i=1}^{r}a_iC_i$. Thus the divisors $P_t \equ (1-mt)L$ and $N_t \equ t\sum_{i=1}^{r}a_iC_i$ form the Zariski decomposition of $D_t$ for all $0\leq t\leq 1/m$. This follows since $D_t = P_t+N_t$, $(P_t\cdot N_t)=0$, $P_t$ is nef (since $L$ is), $N_t$ is a negative definite cycle, and the uniqueness of Zariski decompositions. This and the claim above say that the divisor $D_t=L-tY_1$ stays in the same Zariski chamber for all $0\leq t \leq 1/m$. 

Finally, we know by \cite{LM}*{Theorem 6.4} that if we choose a smooth point $x\in Y_1$, we have the following description of the Okounkov body:
\[
\Delta_{(Y_1 ,x)}(L) \ = \ \{ (t,y)\in\RR^2 \ | \ a\leq t\leq \mu_{Y_1}(L), \textup{ and } \alpha (t)\leq y\leq \beta (t) \ \} .
\]
The choice of $Y_1$ forces $a=0$ and the discussion in the previous paragraph imply that $\mu_{Y_1}(L)=1/m$. The proof of \cite{LM}{Theorem 6.4} says that $\alpha(t)=\textup{ord}_x(N_t|_{Y_1})$ and $\beta(t)=\textup{ord}_x(N_t|_{Y_1})+(Y_1 .P_t)$.  By choosing $x\notin \textbf{B}(L)$, which is always possible since the curve $Y_1$ moves in base point free linear series, we obtain $\alpha(t)=0$ and $\beta(t)=(1-mt)(Y_1 .L)$. Thus $\Delta_{Y_\bullet}(L)$ is a simplex. 
\end{proof}

\begin{cor}\label{cor:surface}
Let $X$ be a smooth projective surface, $D$ a big divisor on $X$. Then there exists an admissible flag $Y_\bullet$ with respect to which
$\Delta_{Y_\bullet}(L)$ is a rational simplex. 
\end{cor}
\begin{proof}
Let $D=P+N$ be the Zariski decomposition of $D$. By \cite{LS11}{Corollary 2.2}, $\Delta_{Y_\bullet}(D)$ is a rational translate of $\Delta_{Y_\bullet}(P)$. By Proposition \ref{prop:surface} there exists a flag $Y_\bullet$ for which $\Delta_{Y_\bullet}(P)$ is a rational simplex, but then so is $\Delta_{Y_\bullet}(D)$ for the same
flag. 
\end{proof}

\begin{rmk}
If $L$ is big and nef, but not semi-ample, then $R(X,L)$ is not finitely generated. Proposition~\ref{prop:surface} yields a significant class of examples of line bundles that are not 
finitely generated, but possess Newton--Okounkov bodies that are rational polygones. Observe that the above proof is very surface-specific,  it is an interesting  question to construct 
big and nef, but non-finitely generated divisors with  rational polytopes as Okounkov bodies if the dimension of the underlying variety is at least three.
\end{rmk}

\begin{exer} $\star$ 
In the situation of Proposition~\ref{prop:surface}, set $Y\deq X\times \PP^1$, and $D\deq L\boxtimes \sO_{\PP^1}(1)$. Check if $R(Y,D)$ is finitely generated, and find flags $\ybul$ 
for which $\nob{\ybul}{D}$ is a rational polytope. 
\end{exer}

\subsection{Infinitesimal Newton--Okounkov polygons}
 
Infinitesimal Newton--Okounkov bodies have been treated in Subsections 2.5 and 2.6 (the original sources are \cites{KL14,KL15b}), here we will take a closer look at the 
infinitesimal theory in dimension two. First we set up notation. Let $X$ be a smooth projective surface, $x\in X$ an arbitrary point, $D$ a divisor on $X$, possibly with 
$\QQ$- or $\RR$-coefficients. Write $\pi\colon \tilde{X}\to X$ for the blowing up of $X$ at $x$, and write $E$ for the exceptional divisor of $\pi$. 

An \emph{infinitesimal Newton--Okounkov polygon} of   $D$ at $x\in X$ is a polygon of the form $\Delta_{(E,z)}(\pi^*D$), where $z\in E$ is an arbitrary point. 
As it turns out, one has more control over infinitesimal Newton--Okounkov polygons than in general. 

\begin{proposition}[\cite{KL14}*{Proposition 3.1}]\label{prop:propinf}
With notation as above, one has 
\begin{enumerate}
\item $\Delta_{(E,z)}(\pi^*D)\subseteq \Delta_{\mu'(D,x)}^{-1}$ for any $z\in E$;
\item there exist finitely many points $z_1,\ldots ,z_k\in E$ such that the polygon $\Delta_{(E,z)}(\pi^*D)$ is independent of 
$z\in E\setminus\{z_1,\ldots,z_k\}$, with  base  the whole line segment $[0,\mu']\times\{0\}$, 
\end{enumerate}
where $\mu'\deq \sup\st{t>0\mid \pi^*D-tE\text{ is big }}$. 
\end{proposition}

\begin{rmk}
 The invariant $\mu'(D,x)$ is a slight variation of the usual Seshadri constant with respect to the pseudo-effective cone: 
 \[
  \mu'(D,x) \deq \mu(\pi^*D,z)\ ,
 \]
where we take into account that the latter turns out to be independent of the choice of $z\in E$. 
\end{rmk}

\begin{proof}
$(1)$ This statement we have already proved in all dimensions.  
$(2)$ Let
\[
D'_t \deq  \pi^*(D)-tE =  P_t'+N_t'
\]
be the appropriate Zariski decomposition. Neither of the coefficient $\nu'$ of $E$ in the negative part $N_0'$ and $\mu'(D,x)$  
depend on the choice of  $z\in E$, hence $\Delta_{(E,z)}(\pi^*(D))\subseteq [\nu',\mu']\times \RR_+$ for any $z\in E$. 

For any $\xi\in [\nu',\mu']$ the length of the vertical slice $\Delta_{(E,z)}(\pi^*(D))_{t=\xi}$ 
is independent of  $z\in E$. Thus, to finish up, it suffices  to show that there exist finitely many points $z_1,\ldots ,z_k\in E$ 
such that $[\nu',\mu']\times\{0\}\subseteq \Delta_{(E,z)}(\pi^*(D))$ for all $z\in E\setminus\{z_1,\ldots ,z_k\}$. 

This, however,  is a consequence of Proposition~\ref{prop:key result},    which states that the function $t\in [\nu',\mu'] \longrightarrow N_t'$ is increasing: in particular, 
the divisor $N_{\mu'}'-N_{t}'$ is effective  for any $t\in [\nu',\mu']$. Consequently,  the function $\alpha$ is identically zero  on the whole interval $[\nu',\mu']$ 
whenever $x\notin\textup{Supp}(N_{\mu'}')\cap E$, as stated.
\end{proof}

The second part of the Proposition implies that it makes sense to talk about the \emph{generic infinitesimal Newton--Okounkov polygon} of $D$ at the point $x\in X$, 
which we denote by $\Delta_x(D)$. 

It is a useful observation that  local positivity at very general points is somewhat less difficult to control.  
Our starting point is a result  of Nakamaye, based on the ideas from \cite{EKL}.

\begin{lemma}[Nakamaye]\label{lem:nakamaye}
Let $x\in X$ be a very general point and $D$ be an effective integral divisor on $X$. Suppose $W\subseteq X$ be an irreducible curve passing through $x$. Let $\overline{W}$ be the proper transform of $W$ through the blow-up $\pi :X'\rightarrow X$ of the point $x$. Also, define 
\[
\alpha(W) \ = \ \textup{inf}_{\beta\in \QQ}\{\overline{W}\subseteq\Null(\pi^*(D)-\beta E)\}\ .
\]
Then $\textup{mult}_{\overline{W}}(||\pi^*(D)-\beta E||)\geq \beta -\alpha(W)$ for all $\beta\geq \alpha(W)$.
\end{lemma} 

Lemma~\ref{lem:nakamaye} forces the generic infinitesimal Newton--Okounkov polygon of very generic points to land in certain area of the plane, depending on the Seshadri constant. 
The resulting  convex-geometric estimate has manifold applications, here  we will focus on the connection with lower bounds on Seshadri constants. 

\begin{proposition}[\cite{KL14}*{Proposition 4.2}]\label{prop:genericinf}
Let $L$ be an ample Cartier divisor on $X$ and  $x\in X$ be a very general point. Then the following mutually exclusive cases can occur. 
\begin{enumerate}
\item $\mu'(L;x)=\epsilon (L;x)$, then $\Delta_x(L)=\Delta^{-1}_{\epsilon(L;x)}$.
\item $\mu'(L,x)>\epsilon (L,x)$, then  there exists an irreducible curve $F\subseteq X$ with $(L\cdot F)=p$ and $\mult_x(F)=q$ such that $\epsilon(L;x)=p/q$. 
Under these circumstances, 
\begin{enumerate}
\item If $q\geq 2$, then $\Delta_x(L)\subseteq \triangle_{ODR}$, where $O=(0,0), D=(p/q,p/q)$ and $R=(p/(q-1),0)$.
\item If $q=1$, then the polygon $\Delta_x(L)$ is contained in the area  below the line $y=t$, and between the horizontal lines 
$y=0$ and $y=\epsilon(L;x)$.
\end{enumerate}
\end{enumerate}
\end{proposition}

\begin{proof}
If  $\epsilon(L,x)=\mu'(L,x)$, then one has automatically $\Delta=\Delta_{\mu'(L,x)}^{-1}$. Therefore we can assume without loss of generality that $\mu'(L,x)>\epsilon(L,x)$. 
In particular, there exists a curve $C\subseteq X$ with $(L\cdot C)=p$ and $\mult_x(C)=q$ such that $\epsilon(L,x)=p/q$. 

Let $\overline{C}$ be the proper transform of  $C$ on $\tilde{X}$. The idea is to calculate  the length of the  vertical segment in the polygon $\Delta(L,x)$ at $t=t_0$ for any 
$t_0\geq \epsilon(A,L)$:
\[
\length\big(\Delta(L)_{t=t_0}\big) \equ  (P_{t_0}\cdot E)\ = \ t_0-(N_{t_0}\cdot E)\ ,
\]
where $\pi^*(L)-t_0E=P_{t_0}+N_{t_0}$ is the corresponding Zariski decomposition. By Lemma~\ref{lem:nakamaye}, one can write $N_{t_0} \equ (t_0-\epsilon(L,x))\overline{C}+N_{t_0}'$, 
where $N_{t_0}'$ remains  effective. This implies the  inequality
\begin{equation}\label{eq:length}
\length\big(\Delta(L)_{t=t_0}\big) \equ t_0-\big((t_0-\epsilon(A,x))\overline{C}+N_{t_0}' \cdot E\big) \ \leq \ 
t_0-(t_0-\epsilon(A,x))q \ .
\end{equation}
The vertical line segment $\Delta(L;x)_{t=t_0}$ starts on the $t$-axis at the point $(t_0,0)$ for any $t_0\geq 0$. Therefore,  by (\ref{eq:length}), the polygon sits below the line 
\[
y \equ t_0-(t_0-\epsilon(L,x))q \equ  (1-q)t_0+\epsilon(L,x)q \ .
\]
When $q=1$, then this line is the horizontal line $y=\epsilon(L,x)$ and when $q\geq 2$ then it is the line passing through the points $D=(p/q,p/q)$ and $B=(p/(q-1),0)$. This finishes the proof.
\end{proof}

As a first corollary, we recover a result of Ein and Lazarsfeld \cite{EL} about lower bounds on Seshadri constants. The idea of the proof is to use areas comparisons between the relevant 
convex objects in the place. 

\begin{corollary}\label{cor:genericseshadri}
 Let $L$ be an ample line bundle on a smooth projective surface. If $x\in X$ is very general, then  $\epsilon(L,x)\geq 1$.
\end{corollary}
\begin{proof}
 By the definition of Seshadri constants and Proposition~\ref{prop:genericinf}, it suffices  to consider the case $(2a)$. 
 Thus,  we know that $\Delta(L;x)\subseteq \triangle_{ODB}$, and as a consequence 
\[
\textup{Area}(\Delta(L;x)) \equ  \frac{L^2}{2} \ \leq \ \textup{Area}(ODB) \equ \frac{p^2}{2q(q-1)} \ . 
\]
In particular,  $\epsilon(L,x)\geq \sqrt{(L^2)(1-\frac{1}{q})}$. Hence,  if we assume  $\epsilon(L,x)<1$, then  by  the rationality of $\epsilon(L;x)$, we also have  
$\epsilon(L,x)\leq \frac{q-1}{q}$. Using the inequality between the areas, we arrive at  $(L^2)<1$, which stands in contradiction with the assumption
that $L$ is an ample Cartier divisor. 
\end{proof}

A small twist of the above argument delivers a new angle on an interesting conjecture of Szemberg \cite{Szemb}. 

\begin{conj}[Szemberg \cite{Szemb}]\label{conj:Szemberg}
Let $X$ be a smooth projective surface with $\rho(X)=1$, $L$ the ample generator of $N^1(X)$ with $(L^2)=N$. Assume that $N$ is not a square. Let $(p_0,q_0)$ be the primitive solution
of the Pell equation $y^2-dx^2=1$, then 
\[
 \epsilon(L;x) \dgeq  \frac{p_0N}{q_0} 
\]
for a very general point $x\in X$. 
\end{conj}

The argument of Corollary~\ref{cor:genericseshadri} yields the following result. 

\begin{thm}[\cite{CracowGroup}]
Let $X$ be a smooth projective surface, $x\in X$ a very general point , $L$ an ample line bundle on $X$ such that $(L^2)=N$ is ample, and $N$   is not a square. Let $(p,q)$ be an arbitrary solution 
of the Pell equation $y^2-dx^2=1$. Then either 
\[
 \epsilon(L;x) \dgeq 1\ ,
\]
or there exists an effectively computable finite set $\Exc(N,p,q)$ such that $\epsilon(L;x)=a/b$ implies $(a,b)\in \Exc(N,p,q)$.  
\end{thm}


The following inequality is useful for giving lower bounds on Seshadri constants at very general points (cf. \cite{Xu95}).

\begin{prop}[\cites{Bas,KSS}]
Let $X$ be a smooth projective surface, $(C_t,x_t)_{t\in T}$ a non-trivial family of pointed curves on $X$ such that there 
exists an integer $m\geq 2$ with $\mult_{x_t}C_t \geq 2$ for some $t\in T$. Then
\[
(C_t^2) \dgeq m(m-1) + \gon(C_t)\ .
\]	
\end{prop}

\begin{exer}[cf. \cite{CracowGroup}]  $\star$ Let $X$ be a smooth projective surface with $\rho(X)=1$, write $L$ for the ample generator of $N^1(X)$, 
and set $(L^2)=N$. Verify that Conjecture~\ref{conj:Szemberg} holds whenever $N=n^2-1$ or $N=n^2+n$ for some positive integer $n$. More concretely, prove that under the given conditions 
\[
\epsilon(L;x) \dgeq \frac{n^2-1}{n}\ \ \ \text{ and }\ \ \ \epsilon(L;x) \dgeq \frac{2(n^2+n)}{2n+1}\ ,
\]
respectively, for  $x\in X$  a very general point. 
\end{exer}

 \newpage

\section{Construction of singular divisors and higher syzygies on abelian surfaces}

\subsection{Generalities on syzygies and property $(N_p)$}

A very effective way to study  varieties is via  their possible embeddings into projective spaces of various dimensions. As we saw earlier,  a very ample line bundle $L$ on a projective variety 
$X$ gives rise  to an embedding
\[
\phi_L \colon X \ \hookrightarrow \ \PP^N \equ \PP(H^0(X,L)) \ . 
\]
The homogeneous coordinate ring of the image is  the main algebraic invariant associated to the pair $(X,L)$, this can be identified with the section ring  
\[
R(X,L) \equ \bigoplus_{m\in\NN} H^0(X,L^{\otimes m})
\]
of $L$. There is an exciting interplay between  the algebraic properties of the section ring $R(X,L)$  and the geometry of the embedding $\phi_L$, which 
 has long been an important area of research. 

The algebraic behavior of $R(X,L)$ is best studied in the category of graded modules over the polynomial ring 
\[
S \deq \CC[x_0,\ldots ,x_N] \equ \Sym^\bullet(H^0(X,L))\ .
\]
As all finitely generated graded $S$-modules, $R(X,L)$ admits a minimal graded free resolution $E_{\bullet}$ of the shape 
\[
\ldots \rightarrow E_i \rightarrow E_{i-1}\rightarrow \ldots \rightarrow E_2\rightarrow E_1\rightarrow E_0\rightarrow R_l \rightarrow 0 \ ,
\]
where each $E_i=\oplus_{j}S(-a_{i,j})$ is a free graded $S$-module. Beside its obvious geometric relevance, studying the set of numbers $(a_{i,j})$, or equivalently, 
the Betti diagram that arises,  is of great  interest. 

Building on ideas of Castelnuovo  and later  Mumford \cite{Mumford}, Green and Lazarsfeld \cite{GL}  introduced a sequence of properties requiring   
 the first $p$ terms in $E_{\bullet}$  to be as simple as possible. 
\begin{defn}[Property $(N_p)$]
With notation as above, we say that the pair $(X,L)$ (or simply that $L$) satisfies property $(N_p)$, if  
$E_0=S$ and 
\[
a_{i,j} \equ i+1 \ \ \ \ \text{for all $j$ and all $1\leq i\leq p$.}
\] 
\end{defn}
 
\begin{rmk}[Koszul cohomology]
Still keeping the notation from above, observe that we can rewrite
\[
  E_i = \bigoplus_{q\in \mathbb{Z}}S(-i-q)\otimes_{\mathbb{C}} K_{i,q}(X,L) \ ,
\]
where the vector spaces $K_{i,q}(X,L)$ are called that \emph{Koszul cohomology groups of $(X,L)$} in the appropriate bidegree. 
\end{rmk}
 
\begin{rmk}  
In geometric terms, property $(N_0)$ holds for $L$ if and only if $\phi_L$ defines a projectively normal embedding, while  property $(N_1)$ 
is equivalent to asking for property $(N_0)$  and that the homogenous ideal $\sI_{X|\PP^N}$ of $X$ in $\PP^N$ be  generated by quadrics. 
\end{rmk}

Due to its classical roots and its obvious relevance for projective geometry, the area surrounding  property  $(N_p)$ has generated a significant amount of  
work in the last thirty years or so with some of the highlights being \cites{BEL, EL1, EL2,P,Voisin}. As it can be seen from the state of affairs, controlling higher syzygies 
has always been  a notoriously difficult question. For more details about this circle of ideas the reader is kindly referred  to \cite{PAGI}*{Section 1.8.D}  and \cite{E} 
for a more algebraic treatment.

One of the early results on property $(N_p)$ that shaped the field to a large extent is Green's theorem \cite{G} on property $(N_p)$ for curves. 

\begin{thm}[Green]\label{thm:Green}
Let $p$ be a natural number, $X$ a smooth curve of genus $g$, $L$ an ample line bundle on $X$. If $\deg_X L \geq 2g+1+p $, then $L$ has property $(N_p)$.  
\end{thm}

In Subsection~\ref{subsection:Green} we show how a simple stability argument leads to a slightly weaker version of Green's theorem. 

Inspired by Fujita's influential conjectures on global generation and very ampleness,  and Green's theorem on the other hand, Mukai formulated a series of conjectures for 
the properties $(N_p)$. 

\begin{conj}[Mukai]\label{conj:Mukai}
Let $p$ be a natural number, $X$  a smooth projective variety, $L$ an ample line bundle on $X$. Then $K_X+mL$ has property $(N_p)$, whenever $m\geq \dim(X) + p+2$. 
\end{conj}

Although non-trivial work has been done in dimension two (see \cite{GP} for a nice overview, and \cite{GPK3} for further interesting work), the conjecture in its entirety 
appears  to be way out of reach even for surfaces currently. 
One of the few class of varieties where the  Fujita-Mukai conjecture has been successfully treated in all dimensions  is that of abelian varieties. 
Based on earlier work of Kempf \cite{Kempf}, 
Pareschi \cite{P} proved a conjecture of Lazarsfeld that gives a Fujita-type answer to syzygy questions on abelian varieties. 

\begin{thm}[Pareschi]\label{thm:Pareschi}
Let $(X,L)$ be a polarized abelian variety, $p$ a natural number. If $m\geq p+3$, then $mL$ satisfies property $(N_p)$. 
\end{thm}

The idea behind the statement is a general principle conjecture by Lazarsfeld: syzygies on abelian varieties should behave roughly the same way as on elliptic curves. 
Pareschi's Theorem is a particularly precise instance of this phenomenon. 

\begin{rmk}
 An interesting feature of abelian varieties is that the bound for property $(N_p)$ is independent of $\dim(X)$. Note at the same time that as far as taking multiples 
 of given ample divisors goes, Pareschi's theorem is sharp already in the case of elliptic curves. If $X$ is an elliptic curve, then a divisor of degree $1$ is effective, 
 but has a base point. A divisor of degree two is base-point free, in fact it defines a double cover of $\PP^1$, but not very ample. This pattern continues, see the exercise 
 below (cf. \cite{GP}).
\end{rmk}

\begin{exer} Let $p$ be a natural number, $X$  an elliptic curve, $D$ a divisor of degree $p+3$ on $X$. Show that $D$ satisfies property $(N_p)$, but not $(N_{p+1})$. 
\end{exer}

\begin{rmk}
Observe that Fujita--Mukai type conjectures are not quite satisfactory from the point of view that they do not scale well under taking multiples. More concretely, 
let $L$ be an ample line bundle on an abelian variety $X$, $p$ a natural number. By Pareschi's theorem $(p+3)L$ satisfies property $(N_p)$. Let us now write 
$L'\deq (p+3)L$ and ask which multiples of $L'$ satisfy property $(N_p)$. If we turn to Pareschi's theorem again, the answer is $\geq p+3$, which is now very 
far from the truth, which we know to be $\geq 1$. The conclusion we can draw is that Mukai-type results tend to be unprecise for non-primitive line bundles. 

In general  Koll\'ar's suggestion  \cite{SingsPairs} (see also \cite{PAGII}) to try and control positivity in terms of intersection numbers appears to be more precise.  
\end{rmk}

\subsection{Green's theorem and stability}\label{subsection:Green}

As an interlude we show  that at the cost of weakening the bound in Green's theorem for curves 
one can buy a  simplified treatment by studying syzygy bundles. This proof  was motivated by the paper \cite{AB} of Arcara and Bertram.

\begin{thm}[\cite{KKM_unpub}]\label{thm:suboptimal Green}
Let $C$ be a smooth projective curve of genus  $g$, $L$ a line bundle on $C$, $p$ a natural number. If 
\[
 \deg L \dgeq g + \frac{g+p-2+\sqrt{(g+p-2)^2+4(g(p-2)+2)}}{2}\ ,
\]
then $L$ satisfies property $N_p$. 
\end{thm}

\begin{rmk}
 If we fix $g$, and let $p$ tend towards infinity, in the limit we obtain $\deg L \geq p + O(p^{1/2+\epsilon})$ for $\epsilon>0$ arbitrarily small. 
 In this sense, our bound is asymptotically optimal. 
\end{rmk}

Our primary tools are  syzygy bundles. If $E$ is a globally generated vector bundles, then the \emph{syzygy bundle} $M_E$ associated to $E$
 is defined by the short exact sequence
\[
 0 \lra M_E \lra \HH{0}{X}{E}\otimes_\CC \sO_X \lra E \lra 0 
\]
obtained from the evaluation map on global sections. 

The main technical tools connected to the bundles $M_E$  we use are the following

\begin{thm}[\cite{EL1}]\label{thm:EL vanishing condition}
Let $X$ be a  projective variety, $L$ a globally generated  ample line bundle on $X$. If 
\[
 \HH{1}{X}{\wedge^{i+1}M_L\otimes L^{\otimes s}} \equ 0
\]
for all $0\leq i\leq p$, and all $s\geq 1$, then $L$ has the property $N_p$. 
\end{thm}

We can control the stability of syzygy bundles using the following result. 

\begin{thm}[Butler, \cite{But94}]\label{thm:Butler}
 Let $C$ be a smooth projective curve of genus $g$, $E$ a semi-stable globally generated vector bundle on $C$. If $\mu(E)\geq 2g$, then
 $M_E$ is semi-stable as well. 
\end{thm}

Theorem~\ref{thm:suboptimal Green} is a consequence of  a vanishing statement related to slope stability. The basic idea is that a semistable vector bundle of negative 
slope has no sections. In what follows $C$ is a smooth projective curve of genus $g$, and 
$L$ is a line bundle on $C$. We denote the rank of $M_L$ by $r$. 

\begin{prop}\label{prop:stability}
Assume that  $\deg L\geq 2g$, and let $i\geq 0$, $s\geq 1$ be arbitrary integers. Then 
\[
 \binom{r}{i+1}\cdot (2g-2-s \cdot \deg L) + \binom{r-1}{i}\cdot \deg L \,<\, 0
\]
implies 
\[
 \HH{1}{C}{\wedge^{i+1}M_L\otimes L^{\otimes s}} \equ 0\ .
\]
\end{prop}

\begin{cor}\label{cor:light Green}
 With notation as above, if $\deg L\geq 2g$, and 
 \[
  \binom{r}{i+1}\cdot (2g-2-s \cdot \deg L) + \binom{r-1}{i}\cdot \deg L \,<\, 0
 \]
for all $0\leq i\leq p$ and $s\geq 1$, then $L$ has property $N_p$. 
\end{cor}

\begin{proof}[Proof of Proposition~\ref{prop:stability}]
First of all, $\deg L\geq 2g$ implies by \cite{HS}{Corollary IV.3.2 (a)} that $L$ is globally generated, hence the syzygy bundle $M_L$ is defined. By 
Serre duality,
\[
 \HH{1}{C}{\wedge^{i+1}M_L\otimes L^{\otimes s}} \equ \HH{0}{C}{\omega_C\otimes \wedge^{i+1}M_L^*\otimes L^{\otimes -s}}\ .
\]
Note that all line bundles are semistable, and so is $M_L$ by Theorem~\ref{thm:Butler}, hence $M_L^*$ and $\wedge^{i+1}M_L^*$ for all $i\geq 0$. In particular,
the vector bundles $\omega_C\otimes \wedge^{i+1}M_L^*\otimes L^{\otimes -s}$
are all semistable. As 
\[
 \rk \wedge^{i+1}M_L^* \equ \binom{r}{i+1}\ \ \ \text{and}\ \ \ \deg \wedge^{i+1}M_L^* \equ \binom{r-1}{i}\ ,
\]
this latter by \cite{Ful}{Remark 3.2.3 (c)S},
we have that 
\[
 \mu(\omega_C\otimes \wedge^{i+1}M_L^*\otimes L^{\otimes -s}) \equ 
 \frac{1}{\binom{r}{i+1}}\cdot (\binom{r}{i+1}\cdot (2g-2-s \cdot \deg L) + \binom{r-1}{i}\cdot \deg L )\ .
\]
Since a semistable vector bundle of negative slope cannot have a non-zero global section, the condition in the Proposition guarantees that 
\[
 \HH{0}{C}{\omega_C\otimes \wedge^{i+1}M_L^*\otimes L^{\otimes -s}} \equ 0\ 
\]
as required.
\end{proof}

\begin{proof}[Proof of Theorem~\ref{thm:suboptimal Green}]
We will check that the assumption
\[
 \deg L \dgeq g + \frac{g+p-2+\sqrt{(g+p-2)^2+4(g(p-2)+2)}}{2}\ ,
\]
guarantees that the conditions of Corollary~\ref{cor:light Green} are satisfied in the required range. 

First of all, observe that to verify
\[
  \binom{r}{i+1}\cdot (2g-2-s \cdot \deg L) + \binom{r-1}{i}\cdot \deg L \,<\, 0
 \]
for all $0\leq i\leq p$ and $s\geq 1$, it suffices to treat the case $s=1$. The conditon is then equivalent  to 
\[
 \binom{r}{i+1}\cdot (2g-2) \,<\, \binom{r-1}{i+1}\cdot \deg L\ ,
\]
and hence to 
\[
 \frac{r}{r-i-1}\cdot (2g-2) \,<\, \deg L\ .
\]
As 
\[
 \frac{r}{r-i-1} \, > \, \frac{r+1}{(r+1)-i-1}\ ,
\]
it suffices to treat the case when we have equality in the Riemann--Roch formula 
\[
 r \equ \hh{0}{C}{L} - 1 \dgeq \deg L + 1-g\ ,
\]
hence we will assume $r=\deg L+1-g$. This leads to the equivalent formulation
\[
 (\deg L-g+1)(2g-2) \,<\, \deg L \cdot (\deg L-g-p)\ .
\]
We set $\lambda = \deg L -g$, and arrive at the quadratic inequality 
\[
 \lambda^2 + (-g-p+2)\lambda + (-gp-2g-2) \,>\, 0\ .
\]
A routine check will show that the inequality indeed holds under the assumption of the Theorem. 
\end{proof}

\subsection{Higher syzygies on abelian surfaces --- main result and overview}

Motivated by  Koll\'ar's line of thought regarding Fujita-type conjectures, it is  natural to ask whether it is feasible to  study property $(N_p)$ for a given line bundle with certain numerics instead. 
Such  a new line of attack in the case of abelian varieties has been recently initiated by Hwang--To \cite{HT} where complex analytic techniques 
(more precisely upper bounds on volumes of tubular neighbourhoods of subtori of abelian varieties) were  used  to control projective normality of line bundles on 
abelian varieties in terms of Seshadri constants. 

The next step was taken by  Lazarsfeld--Pareschi--Popa \cite{LPP}, who used multiplier ideal methods to extend the results of \cite{HT}  to  higher syzygies on abelian varieties. 
The essential contribution of \cite{LPP} is that property $(N_p)$ can be guaranteed via constructing effective divisors with  prescribed multiplier 
ideals and numerics. Our approach builds in part  on the method of proof developed in  \cite{LPP}. For the record we state the final  result of \cite{LPP}. 

\begin{thm}[Lazarsfeld--Pareschi--Popa]\label{thm:LPPmain}
 Let $(X,L)$ be a polarized abelian variety of dimension $n$.  Assume that $\epsilon(L;o) > (p+2)n$. Then $L$ satisfies property $(N_p)$. 
\end{thm}

The main goal of this section  is to study property $(N_p)$ for divisors  on abelian surfaces, and sketch the ideas behind the following fairly precise result, which 
is the central result of this section.  The proof will be given over the last four subsections of Section 3. 

\begin{theorem}\label{thm:np}\cite{KL15c}
Let $p\geq 0$ be a natural number, $X$  a complex abelian surface,  $L$ an ample line bundle on $X$ with $(L^2)\geq 5(p+2)^2$. 
Then the following are equivalent. 
\begin{enumerate}
\item $X$ does not contain an elliptic curve $C$ with $(C^2)=0$ and $1\leq (L\cdot C) \leq p+2$.
\item The line bundle $L$ satisfies property $N_p$.
\end{enumerate} 
\end{theorem}

\begin{rmk}
Part of the added value of Theorem~\ref{thm:np} comes from  treating  the cases where $(L^2)$ is large, but 
$\epsilon(L;o)$ is small. As seen in \cite{BS}  such line bundles abound. Here is a concrete family. In order to be in the situation of \cite{BS}*{Theorem 1}, 
let $C$ be an elliptic curve without complex multiplication, and let $L=a_1F_1+a_2F_2+a_3\Delta$ be an ample line
bundle on $C\times C$, where  the $F_i$'s  are  fibres of the two natural projections and $\Delta\subseteq C\times C$ stands for the class the diagonal. 
The self-intersection is then computed by 
\[
 (L^2) \equ 2a_1a_2 + 2a_1a_3+ 2a_2a_3 \ .
\]
We will  take $a_1,a_2,a_3>0$, hence  \cite{BS}*{Example 2.1} applies. Set $a_2=3$ and $a_3=2$, our plan is to take $a_1\gg a_2,a_3$; in any case if $a_1\geq 4$ then 
\[
 \epsilon(L;o) \equ \min\st{a_1+a_2,a_1+a_3,a_2+a_3} \equ 5\ ,
\]
and there is no elliptic curve of $L$-degree less than $5$ on $C\times C$. Our choice of $a_2$  and $a_3$ forces the line bundle  $L$ to be primitive as well. 
\end{rmk}

As explained above, the sequence of properties $(N_p)$ is best considered as increasingly stronger geometric versions of positivity for line bundles along the lines of global generation
and very ampleness. From this point of view, Theorem~\ref{thm:np} is a natural generalization of Reider's celebrated result \cite{Reider}.

\begin{theorem}[Reider]\label{thm:reider}\cites{Reider,LazLect}
 Let $L$ be a nef line bundle on a smooth projective surface $X$. 
 \begin{enumerate}
  \item Assume that $(L^2)\geq 5$, and let $x\in X$ be a base-point of the linear series $|K_X+L$. Then there exists an effective divisor $D$ passing through $x$ such that
  \begin{eqnarray*}
   \text{either} && (D\cdot L) \equ 0 \ \ \text{ and }\ \ (D^2) \equ -1\  ,\\ 
   \text{or} && (D\cdot L) \equ 1\ \ \text{ and }\ \ (D^2) \equ 0\ .
  \end{eqnarray*}
  \item Assume that $(L^2) \geq 10$, and $x,y\in X$ are two points (possibly infinitely near) that are not separated by the linear series $|K_X+L|$. Then there must exist
  an effective divisor $D$ passing through $x$ and $y$ such that 
   \begin{eqnarray*}
   \text{either} && (D\cdot L) \equ 0 \ \ \text{ and }\ \ (D^2) \in \st{-1,-2}\ , \\ 
   \text{or} && (D\cdot L) \equ 1\ \ \text{ and }\ \ (D^2) \in \st{0,-1}\ , \\
   \text{or} && (D\cdot L) \equ 2\ \ \text{ and }\ \ (D^2) \equ 0\ .
  \end{eqnarray*}
 \end{enumerate}
\end{theorem}

\begin{rmk} Among many others (see \cite{LazLect} for more examples) one can draw the following conclusions from Reider's theorem that are particularly relevant for us: Let $(X,L)$ 
be a polarized abelian surface. 
\begin{enumerate}
 \item If $(L^2)\geq 5$ then $L$ is globally generated 
if and only if there is no elliptic curve $C\subseteq X$ with $(L\cdot C)=1$ and $(C^2)=0$.
 \item  In a similar manner,  if $(L^2)\geq 10$, then $L$ is  very ample 
exactly if $X$ does not contain an elliptic curve $C$ with $(C^2)=0$ and $1\leq (L\cdot C)\leq 2$.
\end{enumerate}

It is worth pointing it out  here that not every very ample line bundle defines a projectively normal embedding, hence  the discrepancy between between Reider's theorem 
for very ample line  bundles and Theorem~\ref{thm:np} in the case of $p=0$. For a concrete example, one can take a general abelian surface $X$ with  a principal 
polarization $L$ of type $(1,6)$ and hence $(L^2)=12$. By Reider's theorem,  $L$ is very ample. In order to be projectively normal however, the multiplication map  
\[
\Sym^2(H^0(X, L))\rightarrow H^0(X, L^{\otimes 2})
\]
needs to be  surjective. It follows from the Riemann--Roch theorem that  the dimension of the first vector space is $21$, while that of the second one  
 is $24$. Thus $L$  is very ample but not projectively normal. 
\end{rmk}

\begin{exer}
 Let $(X,L)$ be a polarized abelian surface of polarization type $(1,5)$. Show that $L$ is very ample, but not projectively normal. 
\end{exer}

\begin{rmk}
Returning  to Theorem~\ref{thm:np}: while  the proofs of the theorems of  Reider (Theorem~\ref{thm:reider})   and Pareschi (Theorem~\ref{thm:Pareschi})
both rely mostly on vector bundle techniques, our approach in confirming property
$(N_p)$ relies in addition  on  multiplier ideals and  the associated vanishing theorems together with the theory of infinitesimal Newton--Okounkov polygons.  

The essential novelty of our proof is the use of  infinitesimal Newton--Okounkov polygons to construct effective $\QQ$-divisors  
whose multiplier ideal coincides with  the maximal ideal of the origin; this  replaces the straightforward genericity argument of \cite{LPP}. 

The implication $(2)\Longrightarrow (1)$ on the other hand is achieved  by a method introduced in  \cite{GLP} and developed further in  \cite{EGHP}. 
The strategy   is that  property $(N_p)$ for the line bundle $L$ implies vanishing of certain higher cohomology group of  the ideal sheaf of the scheme-theoretical 
intersection $X\cap \Lambda$, where $\Lambda$ is a plane of small dimension inside the projective space  $|L|^*$. 
This is explained in detail in \cite{KL15c}, however, since Newton--Okounkov bodies do not enter in any way, we will not discuss it here.  
\end{rmk}

\begin{exercise}
Derive Theorem~\ref{thm:LPPmain} from Theorem~\ref{thm:np} for abelian surfaces.  
\end{exercise}

Let us see how we can use  Theorem~\ref{thm:np} to recover classical results on linear series on abelian surfaces. 

\begin{corollary}\label{cor:one}
Let $X$ be an abelian surface and $L$ an ample line bundle with $L^2\geq 5$. Then
\begin{enumerate}
\item (Pareschi) The line bundle $L^{\otimes (p+3)}$ satisfies condition property $(N_p)$.
\item The line bundle $L^{\otimes (p+2)}$ satisfies property $(N_p)$ if and only if  $(X, L)\ncong (C_1\times C_2, L_1\boxtimes L_2)$, where $L_1$ is a principal polarization of the elliptic curve $C_1$ and $L_2$ is of type $(d)$ on the elliptic curve $C_2$.
\end{enumerate}
\end{corollary}
\begin{proof}
Part $(i)$ is immediate. For part $(ii)$, by  a result of Nakamaye \cite{Nak}*{Lemma 2.6} we know that  there exists an elliptic curve $C_1\subseteq X$ with $(L.C_1)=1$ if and only if $X$ is the product 
of $C_1$ and another elliptic curve $C_2$ and $L\simeq\sO_{C_1\times C_2}(P, D)$ where $P\in C_1$ a point and $D$ a divisor on $E_2$. Then $(ii)$ is immediate from Theorem~\ref{thm:np}. 
\end{proof}

\begin{rmk}
In connection with  their theory of $M$-regularity on abelian varieties, Pareschi and Popa \cite{PP} obtain that for an ample line bundle $L$ with no fixed components 
on an abelian variety $X$,  $L^{p+2}$  already has property $(N_p)$. Thus Corollary~\ref{cor:one}.(ii) is a numerical counterpart of \cite{PP}*{Theorem 6.2} in dimension two.
\end{rmk}

One of the applications of  \cite{EGHP} is that property $(N_p)$ implies  $(p+1)$-very ampleness. A line bundle $L$ is called $k$-very ample if the restriction map $H^0(X,L)\rightarrow H^0(L|_Z)$ is surjective for any $0$-dimensional subscheme $Z\subseteq X$ of length at most $k+1$. As a consequence of \cite{EGHP}*{Remark 3.9}, we can draw the following consequence.

\begin{corollary}\label{cor:two}
Under the assumptions of Theorem~\ref{thm:np}, the following  are equivalent:
\begin{enumerate}
\item $X$ does not contain an elliptic curve $C$ with $(C^2)=0$ and $1\leq (L\cdot C) \leq p+2$.
\item The line bundle $L$ is $(p+1)$-very ample.
\end{enumerate} 
\end{corollary}

\begin{rmk}
The implication $(2)\Longrightarrow (1)$ was first proven by Terakawa. Note that the condition $(L^2)\geq 5(p+2)^2$ and the Hodge index theorem 
imply that the result of Terakawa is equivalent to Corollary~\ref{cor:two}. Terakawa  obtained his result as a consequence of 
the work of Beltrametti-Francia-Sommese from Duke, which in turn uses  Reider's theorem. 

The main results along these lines are due to Bauer and Szemberg. In \cite{BS1} they prove for powers of an ample line bundle on an arbitrary abelian variety. In \cite{BS2}, they 
tackle the case when $X$ is an abelian surface with Picard number one and the line bundle is primitive. Thus the benefit of Collary~\ref{cor:two} is that it does not restrict itself 
to powers of line bundles, or to primitve ones.
\end{rmk}

From a technical point of view,  the essential contribution of the work \cite{LPP} can be summarized in the following statement. 

\begin{theorem}(\cite{LPP})\label{thm:LPP}
	Let $X$ be an abelian surface, $L$ an ample line bundle on $X$, and $p$ a positive integer such that there exists an effective $\QQ$-divisor $F_0$ on $X$ satisfying
	\begin{enumerate}
		\item $F_0 \equiv \frac{1-c}{p+2} L$ for some $0<c\ll 1$, and 
		\item $\sJ(X;F_0)\equ \sI_0$, the maximal ideal at the origin. 
	\end{enumerate}
	Then $L$ satisfies the property $N_p$. 
\end{theorem}

For the sake of clarity we give a quick outline of the argument in \cite{LPP}; to this end,  we  quickly recall some terminology.  As above, $p$ will denote a natural number, 
one  studies sheaves on the $(p+2)$-fold self-product $X^{\times (p+2)}$. Following Green \cite{Green2}*{\S 3}, the $N_p$ property for the line bundle  $L$ holds provided 
\begin{equation}\label{eqn:Green}
H^i(X^{\times (p+2)},\boxtimes^{p+2}L\otimes N\otimes \sI_{\Sigma}) \equ 0\ \ \text{for all $i>0$}, 
\end{equation}
and for any nef line bundle $N$, where $\sI_{\Sigma}$ is the ideal sheaf of the reduced algebraic subset
\[
\Sigma \deq \st{(x_0,\dots,x_{p+1}\mid x_0=x_i\ \text{for some $1\leq i\leq p+1$}}\ .
\]
This is the vanishing condition one  intends to verify. 

Observe that $\Sigma\subseteq X^{\times (p+2)}$ can be realized in the following manner. Upon forming the self-product $Y\deq X^{\times (p+1)}$ with projection maps 
$\pr_i\colon Y\to X$, one considers  the subvariety
\[
\Lambda \deq \bigcup_{i=1}^{p+1}\pr_i^{-1}(0) \equ \st{(x_1,\dots,x_{p+1})\mid x_i=0\ \ \text{for some $1\leq i\leq p+1$}}\ .
\]
Next, look at the morphism 
\[
\delta\colon X^{\times (p+2)} \to  X^{\times (p+1)}\ ,\ (x_0,\dots,x_{p+1}) \mapsto (x_0-x_1,\dots,x_0-x_{p+1})\ ,
\]
then 
\[
\Sigma \equ \delta^{-1}(\Lambda)
\]
scheme-theoretically. Consider the divisors 
\[
E_0 \deq  \sum_{i=1}^{p+1} \pr_i^*F_0\ \ \text{ and }\ \ E\deq \delta^*E_0\ ,
\]
as forming multiplier ideals commutes with taking pullbacks by smooth morphisms (cf. \cite{PAGII}*{9.5.45}), one observes that 
\[
\sJ(X^{\times p+2};E) \equ \sJ(X^{\times p+2};\delta^*E_0) \equ \delta^*\sJ(X^{\times p+1};E_0) \equ \delta^*\sI_{\Lambda} \equ \sI_{\Sigma}\ .
\]
The divisor $(\boxtimes_{i=1}^{p+2}L)-E$ is ample by \cite{LPP}*{Proposition 1.3}, therefore 
\[
H^i(X^{\times (p+2)},\boxtimes^{p+2}L\otimes N\otimes \sI_{\Sigma}) \equ H^i(X^{\times (p+2)},\boxtimes^{p+2}L\otimes \sJ(X^{\times p+2};E)) \equ 0 
\ \text{ for all $i>0$}
\]
by Nadel vanishing, where one sets $N=0$.

The rest of the subsection is devoted to an outline of  the ideas behind the implication $(1)\Longrightarrow (2)$ in Theorem~\ref{thm:np} in the simplest case $p=0$, that is, for   
projective normality. 

Assume henceforth  that $X$ is an abelian surface, and $L$  an ample line bundle on $X$ with $(L^2)\geq 20$. 
Suppose in addition that $X$ does not contain any elliptic curve $C$ with $(C^2)=0$ and $1\leq (L\cdot C)\leq 2$. To ease the presentation we assume that there exists a 
Seshadri-exceptional  curve $F\subseteq X$  passing through the origin $o\in X$ with the property that 
$r\deq (L\cdot F)\geq q=\mult_o(F)\geq 2$ and $\epsilon\deq \epsilon(L;o)=r/q$. 

Our starting point is the method of  \cite{LPP}, which builds on the following observation of Green \cite{Green2} (see also \cite{I}): consider the diagonal $\Delta\subseteq X\times X$ 
with ideal sheaf $\sI_{\Delta}$. Projective normality of $L$ is equivalent to the vanishing condition
\begin{equation}\label{eq:intro Green}
H^1(X\times X, L\boxtimes L\otimes\sI_{\Delta}) \equ 0\ .
\end{equation}
The authors of \cite{LPP} then go on to show that in order to guarantee the vanishing in (\ref{eq:intro Green}), it suffices to verify 
the existence of an effective  $\QQ$-divisor 
\[
D\ \equiv \ \frac{1-c}{2}L\ ,\ \text{ for some } 0<c\ll 1 \ ,
\]
such that $\sJ(X,D)=\sI_{X,o}$. 

Assuming one can do so, using  the difference morphism $\delta\colon X\times X\rightarrow X$ given  by 
$\delta (x,y)=x-y$, one deduces that 
\[
\sI_{\Delta} \equ f^*(\sJ(X;D)) \equ \sJ(X\times X,f^*(D)) \ ,
\] 
which in turn leads to (\ref{eq:intro Green}) via Nadel vanishing  for multiplier ideals. 

While directly constructing divisors with a given multiplier ideal is quite difficult in general, a simple observation 
from homological algebra  ensures that at least in the case of projective normality  
 it is enough to exhibit  such a divisor $D$ with $\sJ(X,D)=\sI_{X,o}$ locally around $o\in X$. 
 
This is where the main new ingredient of the paper comes into play: it turns out that one can use  infinitesimal Newton--Okounkov polygons to show the existence of  
suitable $\QQ$-divisors $D$ with $\sJ(X,D)=\sI_{X,o}$ over an open subset containing $o$. 

Write  $\pi:X'\rightarrow X$ for the blowing-up of $X$ at the origin $o$ with exceptional divisor $E$, and let $B\deq\frac{1}{2}L$. The first step is to find a criterion
in terms of infinitesimal Newton--Okounkov polygons that guarantee the existence of divisors  $D$ as above. In Theorem~\ref{thm:NO to singular} we show that if
\[
\intt \Delta_{(E,z)}(\pi^*(B)) \ \cap \ \{(t,y)\ | \ t\geq 2, 0\leq y\leq 2t\}
\] 
is non-empty for any $z\in E$, then one always find an effective $\QQ$-divisor $D=(1-c)B$ so that $\sJ(X;D)=\sI_{X,o}$ in a neighborhood of the origin $o$.

Suppose the opposite, i.e.  that for some $z_0\in E$ the Newton--Okounkov polygon $\Delta_{(E,z_0)}(\pi^*(B))$ does not intersect the interior of the region 
\[
\Lambda \deq \{ (t,y)\in \RR^2\ | \ t\geq 2, 0\leq y\leq 2t\}\ 
\] 
(for an illustration  see Figure~\ref{fig:1} (a)).

\begin{figure}  
	\caption{An infinitesimal Newton--Okounkov body and the region $\Lambda$ \label{fig:1}}
\begin{tikzpicture}


\draw [->]  (0,0) -- (4.5,0);   
\node [below right] at (4.5,0) {$t$};
\draw [->]  (0,0) -- (0,4.5); 
\node [left] at (0,4.5) {$y$}; 
\draw (0,0) -- (4.5,4.5); 
\node [left] at (4,4.5) {$y=t$};
\draw [fill=gray, ultra thick] (0,0) -- (2,2) -- (3.5,3) -- (2,0) -- (0,0);  
\node [left] at (2.5,3) {$\Delta_{(E,z_0)}(\pi^*B)$};
\draw (2,0) -- (2,2); 
\draw [ultra thick] (3,0) -- (3,1.5);  
\draw (0,0) -- (5,2.5); 
\node [above] at (5,2.5) {$y=2t$};  
\shadedraw [dotted,shading=axis] (3,0) -- (3,1.5) -- (5,2.5) -- (4.25,0);
\node [above] at (4,1) {$\Lambda$}; 
\draw (4.24,0) -- (3,0) -- (3,1.5) -- (5,2.5);


\node [below left] at (0,0) {$O$};
\node [below] at (2,0) {$(\epsilon,0)$};
\filldraw [black]  (3,1.5) circle (3pt); 
\node [above] at (3.67,1.85) {$(2,1)$};
\end{tikzpicture}
\end{figure}
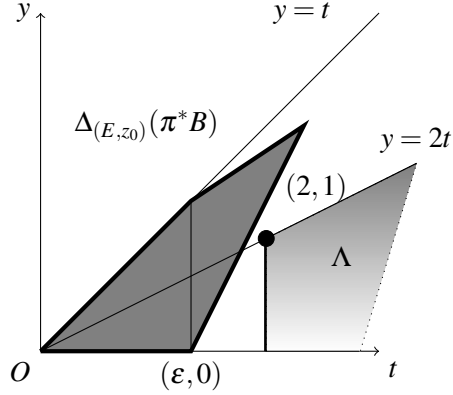 
 
This implies that the polygon $\Delta_{(E,z_0)}(\pi^*(D))$ sits above a certain line that passes through the point $(2,1)$. But since the area of 
$\Delta_{(E,z_0)}(\pi^*(B))$ is quite big, namely equal to $(B^2)/2\geq 5/2$,  the Seshadri constant $\epsilon(B;o)$ is then forced to be  small by convexity, for it is 
equal to the size of the largest inverted simplex inside $\Delta_{(E,z_0)}(\pi^*(B))$ by \cite{KL14}*{Theorem 3.11}. 
A more precise computation gives  the upper bound  $\epsilon(B;o)\leq \frac{5-\sqrt{5}}{2}$. 

In order to obtain  a contradiction, notice that $X$ carries a transitive group action, thus the origin $o\in X$ behaves like a very general point. 
By the work of Ein, K\"uchle, and Lazarsfeld (see \cite{EL} and \cite{EKL}), one knows that the Seshadri constant at a very generic point for a line bundle 
has to be quite big, especially when it is a non-integral rational number, as  assumed. 

Relying on  \cite{EKL} and \cite{NakVeryGen},  the authors show in \cite{KL14}*{Proposition 4.2}   the  inclusion  
\[
\Delta_{(E,z)}(\pi^*(B)) \ \subseteq \ \triangle OAC, \text{ for generic point }z\in E \ ,
\]
where $O=(0,0),A=(r/2q,r/2q)$ and $C=r/2(q-1)$. The area of the polygon on the right-hand side  is  $(B^2)/2\geq 5/2$, hence a simple area comparison gives  
$\epsilon(B;o)=\frac{1}{2}\epsilon(L;o)\geq \sqrt{\frac{5}{2}}$ (see Figure~\ref{fig:1b}), which immediately leads to a contradiction since    $\frac{5-\sqrt{5}}{2} < \sqrt{\frac{5}{2}}$.

\begin{figure} 
\begin{tikzpicture}
%
%
\draw [->]  (0,0) -- (4.5,0);   
\node [below right] at (4.5,0) {$t$};
\draw [->]  (0,0) -- (0,4.5); 
\node [left] at (0,4.5) {$y$}; 
\draw (0,0) -- (4.5,4.5); 
\node [left] at (4,4.5) {$y=t$};
\draw [fill=gray, ultra thick] (0,0) -- (3,3) -- (3.5,1.5) -- (3.75,0) -- (0,0);  
\node [left] at (2,3) {$\Delta(\pi^*B)$};
\draw (3,0) -- (3,3); 
\draw (3,3) -- (4,0);
%
%
\node [below left] at (0,0) {$O$};
\node [below] at (3,0) {$(\epsilon,0)$};
\node [above] at (3,3) {$A$};
\node [below] at (4,0) {$C$};
\end{tikzpicture} 
\caption{The containment $\Delta_{(E,z)}(\pi^*(B))\subseteq \triangle OAC$  \label{fig:1b}}
\end{figure}
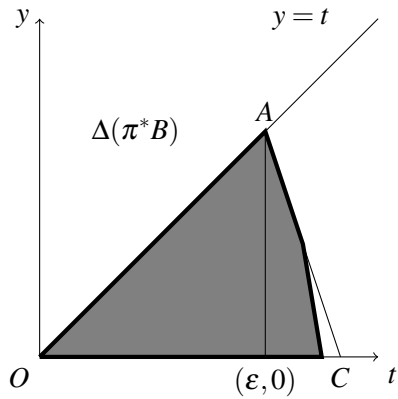

\subsection{On the existence  of singular divisors}

The construction of effective $\QQ$-divisors with prescribed singularities and numerical behaviour is a very strong tool in projective geometry. To indicate its power, it 
can be used to prove many of the classical 'big theorems' of the minimal model program, at the same time it also delivers  some of the most powerful general positivity 
theorems in existence, like the results of Angehrn--Siu and Koll\'ar--Matsusaka. A large part of its strength comes from the fact that it provides a way to extend global 
sections of adjoint line bundles from subvarieties.

Illustrated in a simple setting, the basic idea goes as follows: let $X$ be a smooth projective variety, $Z\subset X$ a subvariety, $L$ an ample line bundle on $X$. It is 
natural to consider the restriction map 
\[
\res^X_Z \,\colon\,  \HH{0}{X}{\sO_X(K_X+L)} {\lra} \HH{0}{Z}{\sO_Z(K_X+L)}\ . 
\]
One would like $\res^X_Z$ to be surjective, since that would imply that we can lift global sections of $K_X+L$ from $Z$ to $X$. 

Since $L$ is ample,  Kodaira vanishing yields that $\res^Z_X$ is surjective precisely if 
\[
 \HH{1}{X}{\sO_X(K_X+L)\otimes \sI_{Z/X}} \equ 0\ . 
\]
The upshot is that \emph{assuming} that $Z$ can be realized as the zero scheme of the multiplier ideal $\sJ(X;D)$ of an effective $\QQ$-divisor $D$, Nadel 
vanishing will often do the trick. 

\begin{thm}[Nadel vanishing]
Let $X$ be a smooth projective variety, $L$ an integral Cartier divisor, $D$ an effective $\QQ$-divisor such that $L-D$ is big and nef. Then 
\[
 \HH{i}{X}{\sO_X(K_X+L)\otimes \sJ(X;D)} \equ 0 \ \ \ \text{for all $i\geq 1$}.
\]
\end{thm}

\begin{cor}
Let $X$ be a smooth projective variety, $Z\subset X$ a subvariety, $L$ an ample  line bundle on $X$. If there exists an effective $\QQ$-divisor $D$ for which
\begin{enumerate}
 \item $L-D$ is big and nef
 \item $\sJ(X;D)=\sI_{Z/X}$
\end{enumerate}
then the restriction map $\res^X_Z \colon \HH{0}{X}{\sO_X(K_X+L)} \rightarrow \HH{0}{Z}{\sO_Z(K_X+L)}$ is surjective. 
\end{cor}

\begin{rmk}
The difficulty lies in satisfying the conditions $L-D$ ample and  $\sJ(X;D)=\sI_{Z/X}$ at the same time: the more elbow room we need in the construction of $D$,  the more positive it becomes, hence the less likely it is that $L-D$ is big and nef, 
and conversely.  
\end{rmk}

Another type of question where construction of singular divisors appears as a very useful tool is to confirm vanishing results for the higher cohomology of vector bundles. 
This type of problem often occurs in connection with syzygy bundles \'a la Lazarsfeld--Mukai; here the issue is to reduce  vanishing for  twisted syzygy bundles 
to the a similar question for line bundles. As one would expect a correction term comes in, which happens to be the ideal sheaf of certain partial diagonals on 
Cartesian self-products of $X$.

For now let $X$ be a smooth projective variety, $L$ a big and nef line bundle, $Z\subseteq X$ a subvariety. We are interested in constructing 
effective $\QQ$-divisors $D$ with the properties that 
\begin{enumerate}
 \item $L-D$ is big and nef, 
 \item $\sJ(X;D) \equ \sI_{Z/X}$. 
\end{enumerate}

\begin{eg}(Points in the projective plane)
As a first instance of the construction problem, let us look at the case when $X=\PP^2$, $Z=\st{P}$, and $L=dH$ an integral divisor with $d$ a positive integer. Let $x,y,z$ be homogeneous coordinates 
on $\PP^2$, and assume without loss of generality  that $P=\st{x=y=0}$. Consider the curve 
\[
 C \deq \Zeroes(x^3-y^2z)\ ,
\]
then \cite{PAGII}*{Example 9.2.15} shows that 
\[
 \sJ(X;c\cdot C) \equ \begin{cases} \m_P \equ (x,y) \subseteq \sO_X & \text{if } c=\frac{5}{6} \\
 \sO_X & \text{if } 0<c<\frac{5}{6}\ , \\
                       \end{cases}
\]
hence $\sJ(X;D)= \m_P$ for $D=\tfrac{5}{6}C$. As $\deg \tfrac{5}{6}C \equ \tfrac{5}{2}$, the construction problem in our case can be definitely solved for $d\geq 3$. 
\end{eg}

\begin{question}
In the situation of the above example, what is the infimum of  positive rational numbers $\alpha$ such that there exists an effective $\QQ$-divisor $D$ with $\sJ(X;D)=\m_P$ and $D\equiv \alpha H$?  
What comes from the curves $x^n-y^mz^{n-m}$ with $n\geq m$? 
\end{question}

\begin{rmk}
Note that it often suffices to construct a $\QQ$-divisor $D$ such that 
\[
 \sJ(X;D)|_U \equ \sI_{Z/X}|_U
\]
for some neighbourhood $Z\subseteq U$. 
\end{rmk}

\begin{rmk}
The excellent survey article \cite{CKL} provides proofs using the construction of singular divisors to the non-vanishing, base-point free, and rationality theorems. The main extra ingredient is Kawamata's subadjunction theorem. 
\end{rmk}

\subsection{Singular divisors via Newton--Okounkov bodies on arbitrary surfaces}

Here we show how to use infinitesimal Newton--Okounkov bodies to construct singular divisors. and as such, contains the technical core of our proof of Theorem~\ref{thm:np}. 
More precisely, we will present a method that proves the existence of an effective $\QQ$-divisor $D$ from a given numerical equivalence class such that the multiplier ideal of $D$ 
coincides with the maximal ideal sheaf of a point, at least in an open neighborhood of our point of interest. Naturally, the existence of such an object is conditional;
in our case the condition will be in terms of the  relative position of convex   subsets in the plane.

For the duration of this subsection let $X$ be an arbitrary smooth projective surface, $x\in X$ an arbitrary point, $B$ an ample $\QQ$-divisor on $X$. 
Let $\pi\colon X'\to X$ stand for the blowing up of $X$ at $x$ with exceptional divisor $E$. Points on $E$ will be denoted by variants of the letter $z$. Write 
\[
\Lambda \deq \st{ (t,y)\in \RR^2\mid t\geq 2, y\geq 0, \text{ and }t\geq 2y}\ . 
\]
The promised result goes as follows. 
\begin{theorem}\label{thm:NO to singular}
Let $B$ be an ample $\QQ$-divisor on $X$. Then,  if 
\begin{equation}\label{eq:nopolygon1}
\text{interior of }\big(\Delta_{(E,z)}(\pi^*B) \ \bigcap \ \Lambda\big) \ \neq \ \varnothing, \forall z\in E\ ,
\end{equation}
then there exists an effective $\QQ$-divisor $D\equiv (1-c)B$ for any $0<c\ll 1$ such that $\sJ(X;(1-c)D)=\sI_{x}$ in a neighborhood of the point $x$.
\end{theorem}

\begin{proof} 
Let us fix a point $z\in E$, hence an infinitesimal flag $(E,z)$. We  will  first work on the blowing up  $X'$ of $X$, and show there  for any $0<c\ll 1$  condition $(\ref{eq:nopolygon1})$ 
implies  the existence of a $\QQ$-effective divisor $D'\equiv (1-c)\pi^*B$ with $\sJ(X',D')|_U=\sO_U(-2E)$ for some open neighborhood $U$ of the exceptional divisor. 

Note that  $(\ref{eq:nopolygon1})$ is an open condition, hence it is also satisfied  for the divisor class $\pi^*((1-c)B)$ whenever $0<c\ll 1$. In what follows,  fix a rational number $c>0$
such that the above property holds, and set  $B' \deq \pi^*((1-c)B)-2E$. 

Condition $(\ref{eq:nopolygon1})$ yields that $\Delta_{(E,z)}(\pi^*B)$ contains an interior point $(t,y)\in\Lambda$ with $2\leq t<\mu_E(\pi^*B)$, therefore $\pi^*B-2E$ is 
a big $\QQ$-divisor on $X'$. By \cite{KL14}*{Remark 1.7}  we know that 
\[
\Delta_{(E,z)}(B') \equ \Delta_{(E,z)}(\pi^*((1-c)B)) \ \cap \ \{(t,y) \ | \ t\geq 2\} \ +\ 2\eone\ .
\]
Write $B'=P+N$ for  the   Zariski decomposition of $B'$; we will look for  the divisor $D'$ in the form
\[
D' \equ  P' + N +2E \ (\,\equiv \pi^*((1-c)B)\, )\ ,
\]
with $P'\equiv P$  an effective divisor\footnote{Note that the above expression is in general not the Zariski decomposition
of $D'$}. 

Since $P$ is big and nef, \cite{PAGI}*{Theorem 2.3.9} shows that 
one can find an effective divisor $N'$ and a sequence of ample $\QQ$-divisors $A_k$ with the property that  
\[
P   \equ A_k+\tfrac{1}{k}N' \ \ \text{for  $k\gg 0$.}
\]
Choose $P'\equiv P$ to be an effective $\QQ$-divisor such that $A_k=P'-\tfrac{1}{k}N$ is general and effective. Indeed, a 
very general element in its linear series will do. This yields 
\[
\sJ(X',D') \equ  \sJ(X',A_k+\frac{1}{k}N'+N+2E) \equ  \sJ(X', \frac{1}{k}N'+N+2E) \equ  \sJ(X',N)\otimes \sO_{X'}(-2E)\ ,
\]
for $k\gg 0$. The second equality is an application of the Koll\'ar--Bertini theorem   \cite{PAGII}*{Example~9.2.29}, the third one 
comes from invariance under small perturbations (see \cite{PAGII}*{Example~9.2.30}) and  \cite{PAGII}*{Proposition 9.2.31}. 

Thus, it remains to check that  $\sJ(X',N)$ is trivial  at any point $z\in E$.  Since 
\[
\Delta_{(E,z)}(B') \ \bigcap \ \{0\}\times [0,1)\ \neq \ \varnothing,\ \forall z\in E\ ,
\]
\cite{LM}*{Theorem 6.4} implies 
\[
1 \ > \ \alpha(0) \equ  \ord_z(N|_E), \forall z\in E\ .
\]
On the other hand,  $\ord_z(N|E)\geq \ord_z(N)$  yields 
\[
\ord_z(N) \ < \ 1, \textup{ for all } z\in E\ .
\]
In the light of  \cite{PAGII}*{Proposition 9.5.13} this implies that $\sJ(X',N)$ is trivial at any $z\in E$, as  needed.

Back on $X$, by  Lemma~\ref{lem:going down} there exists an effective $\QQ$-divisor $D\equiv (1-c)B$ on $X$ with $D'=\pi^*D$.  
Hence the birational transformation rule for muliplier ideals (see \cite{PAGII}*{Theorem~9.2.33}) we have the sequence of 
equalities 
\[
\sJ(X,D) \equ \pi_*\big(\sO_{X'}(K_{X'/X})\otimes \sJ(X',D')\big) \equ \pi_*(\sJ(X',N))\otimes \sI_x \ ,
\]
where $K_{X'/X}=E$. Since $\sJ(X',N)$ is trivial at any $z\in E$, this means that  $\sJ(X,D)=\sI_x$ in a neighborhood of the point $x\in X$. 
\end{proof}

\begin{thm}\label{thm:NO to small Seshadri}
With notation as above,  assume that  $B^2\geq 5$. If 
\begin{equation}\label{eq:nopolygon2}
\textup{interior of }\big(\Delta_{(E,z_0)}(\pi^*B)\ \bigcap \ \Lambda\big) \ =\ \varnothing\ ,
\end{equation}
for some point $z_0\in E$, then the Seshadri constant $\epsilon(B;x)\leq \frac{5-\sqrt{5}}{2}$. 
\end{thm}
\begin{proof}
Set $\epsilon\deq \epsilon(B;x)$ and  observe that  \cite{KL14}*{Theorem D} yields the containement
\[
\triangle OAA'  \dsubseteq \Delta_{(E,z_0)}(\pi^*B), \text{ where } O\equ (0,0), A\equ (\epsilon,0), \text{ and } A'\equ (\epsilon,\epsilon)\ .
\]
If $\epsilon>2$, then 
\[
\intt \big(\Delta_{(E,z_0)}(\pi^*B)\ \bigcap \ \Lambda\big) \dsupseteq \intt \big(\triangle OAA'\cap \Lambda\big) \neq \varnothing
\]
contradicting condition (\ref{eq:nopolygon2}). The case $\epsilon=2$ is equally impossible since that would imply via  (\ref{eq:nopolygon2}) that 
$\Delta_{(E,z_0)}(\pi^*B)$ lies to the left of the  line $t=2$;   as it lies underneath the diagonal anyway, it would have volume less than $2$, 
contradicting $(B^2)\geq 5$. Therefore we can safely assume $\epsilon<2$. 

Now, condition~$(\ref{eq:nopolygon2})$ implies that the segment $\{2\}\times [0,1)$ does not intersect  $\Delta_{(E,z_0)}(\pi^*B)$. Since the latter is convex, it must lie above some  line $\ell$ that passes through the point $(2,1)$ and below the diagonal. 
Let $(\delta,0)$ be the point of intersection of  $\ell$ and  the $t$-axis. It follows  that $\delta\geq \epsilon$, since we know that the inverted simplex $\triangle OAA'$ is contained in $\Delta_{(E,z_0)}(\pi^*B)$. So, our  goal is now to find 
an upper bound on $\delta$. 

We can assume that both $\epsilon,\delta>1$, since $\epsilon\leq 1$ already implies our statement. In this case we have $\tfrac{1}{2-\delta}>1$ for the slope of the line $\ell$, therefore the diagonal and $\ell$ intersect at the  point  
$(\tfrac{\delta}{\delta-1},\tfrac{\delta}{\delta-1})$ in the first quadrant. The triangle formed by $\ell$, the diagonal,  and the $t$-axis includes our 
Newton--Okounkov polygon, therefore its area is no smaller than the area of $\Delta_{(E,z_0)}(\pi^*B)\geq 5/2$. 
Hence, we obtain $\epsilon(B;x)\leq\delta\leq \frac{5-\sqrt{5}}{2}$. 
\end{proof}

\begin{exer}
In the situation of Theorem~\ref{thm:NO to small Seshadri}, give a lower bound for the Seshadri constant $\epsilon(B;x)$ if we know that $(B^2)\geq d (\geq 5)$. 
\end{exer}

\begin{remark}
In his seminal work \cite{D} Demailly introduced  Seshadri constants with the aim of controlling the asymptotic growth of separation of jets by an ample line bundle 
at the given point. Our  Theorem~\ref{thm:NO to singular} can be viewed  as a more effective version of Demailly's  idea as explained in 
\cite{PAGI}*{Theorem~5.1.17} and\cite{PAGI}*{Proposition~5.1.19}.

To see this  note that  $\epsilon(B;x)$ equals  the largest inverted simplex $\Delta_{\lambda}^{-1}$ embedded in $\Delta_{(E,z)}(\pi^*B)$  for any $z\in E$. 
\end{remark}

Coupled with Nadel vanishing we obtain the following Reider-tpye effective  global generation result, in the spirit of  Demailly's original train of thought.

\begin{corollary}\label{cor:eff glob gen}
Let $X$ be a smooth projective surface, $x\in X$, $L$ an ample line bundle on $X$ with $(L^2)\geq 5$. 
If  $\epsilon(L;x)\geq \tfrac{5-\sqrt{5}}{2}$, then 
$x$ is not a base point of the adjoint linear series $|K_X+L|$.
\end{corollary}

\begin{remark}
 Note that  conditions  (\ref{eq:nopolygon1}) and (\ref{eq:nopolygon2}) are complementary whenever  $(L^2)\geq 5$. 
\end{remark}

\begin{lemma}\label{lem:going down}
 Let $Y$ be a smooth projective variety, $Z\subseteq X$ a smooth subvariety, $\pi\colon Y'\to Y$ the blowing-up of $Y$ along $Z$. Furthermore, let 
 $B$ a Cartier divisor on $Y$, $D'$ a Cartier divisor on $Y'$. 
 If $D'\equiv \pi^*B$, then there exists a divisor $D\equiv B$ on $Y$ such that $D'\equ \pi^*D$. 
\end{lemma}
  
\begin{exer}
 Prove Lemma~\ref{lem:going down} with the help of the Negativity Lemma. 
\end{exer}

\subsection{Singular divisors at very general points.}
It is an important guiding principle that local positivity of a line bundle is considerably easier to control
at a general or a very general point. This observation is manifest in the work of Ein--K\"uchle--Lazarsfeld \cite{EKL} (see also \cite{EL}), where the authors  give a lower bound on Seshadri constants at very general points depending only on the dimension of the ambient space. 

Later, Nakamaye \cite{NakVeryGen} elaborated some of the ideas of \cite{EKL}, while translating them to the language of multiplicities. This thread was in turn
picked up in \cite{KL14}, and further developed in the framework of infinitesimal Newton--Okounkov bodies of surfaces, as seen in Proposition~\ref{prop:genericinf}. It is hence not surprising that 
one can expect stronger-than-usual results on singular divisors at very general points.

\begin{theorem}\label{thm:verygeneric}
Let $p$ be a positive integer, $X$ a smooth projective surface, $L$ an ample line bundle on $X$ with $(L^2)\geq 5(p+2)^2$. 
Let  $x\in X$ be  a very general point, and assume  that there is no irreducible curve $C\subseteq X$ smooth at $x$ with $1\leq (L\cdot C)\leq p+2$. 

Then, for every  some (or,  equivalently, every) point $z\in E$
\begin{equation}\label{eq:verygeneric}
\length\Big(\Delta_{(E,z)}(\pi^*B)\cap \{2\}\times \RR\Big) \ > \ 1\ , 
\end{equation}
where as usual we write  $B\deq\frac{1}{p+2}L$.
\end{theorem}

\begin{corollary}\label{cor:very general}
Under the assumptions of  Theorem~\ref{thm:verygeneric}, there always exists an effective $\QQ$-divisor $D\equiv (1-c)B$ for some $0<c\ll 1$
such  that $\sJ(X;D)=\sI_{X,x}$ in a neighborhood of the point $x$.
\end{corollary}
\begin{proof}
By Proposition~\ref{prop:genericinf} any infinitesimal Newton--Okounkov polygon sits under the diagonal in $\RR^2$. Therefore $(\ref{eq:verygeneric})$ implies that condition $(\ref{eq:nopolygon1})$ from Theorem~\ref{thm:NO to singular} is satisfied, hence the claim. 
\end{proof}

\begin{remark}
As abelian surfaces are homogeneous, Theorem~\ref{thm:verygeneric} holds for all points on them. 
\end{remark}

Before proceeding with the proof, we make some preparations. Let $\epsilon=\epsilon(L;x)$, and write  $\epsilon_1 \deq \e(B;x)=\frac{1}{p+2}\epsilon$. 
As  Zariski decompositions of the divisors $\pi^*B-tE$ along the line segment $t\in [\epsilon_1,2)$ will play a decisive role, we will fix notation
for them as well. 

By \cite{KLM1}*{Proposition 2.1}  there exist only finitely many curves $\og_1,\dots,\og_r$ that occur in the negative part of $\pi^*B-tE$ 
for any $t\in [\epsilon_1,2)$. Write  $\e_i$ for  the value of $t$ where $\og_i$ first appears. We can obviously assume
that $\e_1\leq\e_2\leq \ldots \leq \e_r$, in addition we put $\e_{r+1}=2$. Denote also by $m_i\deq (\og_i\cdot E) \equ\mult_x(\Gamma_i)$ for $1\leq i\leq r$, where $\Gamma_i=\pi(\og_i)$. 
For any $t\in[\epsilon_1 ,2)$ let $\pi^*B-tE=P_t+N_{\pi^*(B)-tE}$ be the Zariski decomposition of the divisor.

\begin{lemma}\label{lem:Nakamaye and Zariski}
	With notation as above, we have 
	\[
	N_{\pi^*B-tE} \equ \sum_{i=1}^{r}\mult_{\og_i}(\| \pi^*B-tE\|) \og_i
	\]
	for all $t\in [\e_1,2)$. 
\end{lemma}
\begin{proof}
	It follows from the definition of asymptotic multiplicity and the existence and uniqueness of Zariski decomposition that for an arbitrary big and 
	nef $\RR$-divisor $D$ and an irreducible curve $C$, the expression $\mult_C \|D\|$ picks up the coefficient of $C$ in the negative part of $D$.    
\end{proof}

\begin{lemma}\label{lem:estimate for many curves}
	With notation as above, if $x\in X$ is a very general point, then
	\[
	\length\Big(\Delta_{(E,z)}(\pi^*B)\cap\{t\}\times\RR\Big) \equ  (P_t\cdot E) \dleq l_i(t) \,\deq\, (1-\sum_{j=1}^{i}m_j)t + \sum_{j=1}^{i}\e_jm_j \ ,
	\]
	for all  $t\in [\epsilon_i,\epsilon_{i+1}]$ and all  $z\in E$.
\end{lemma}
\begin{remark}
	Lemma~\ref{lem:estimate for many curves} is essentially a restatement of \cite{NakVeryGen}*{Lemma 1.3} in the context of Newton--Okounkov bodies. Note that Nakamaye's 
	claim is strongly based on \cite{EKL}*{\S 2} (see also  \cite{PAGI}*{Proposition~5.2.13}), providing a way 
	of  ``smoothing divisors in affine families'' using differential operators. This observation will play  an important role in the proof of Theorem~\ref{thm:verygeneric} 
	and consequently in that of  Theorem~\ref{thm:np}.
\end{remark}

Returning to the verification of  Lemma~\ref{lem:estimate for many curves}, let $A_i$ stand for  the point of intersection of the lines $t=\e_i$ and 
$l_i$ (here we identify the function $l_i$ with its graph). Note that the function $\ell\colon [\e_1,2]\rightarrow \RR_+$ given by $\ell(t) \deq l_i(t)$ on each interval $[\e_i,\e_{i+1}]$ is continuous, 
and satisfies $(\e_i,\ell(\e_i))=A_i$. Let $T_i$ be the polygon spanned by the origin, the points $A_1,\dots,A_i$, and the intersection point $F_i$  of the 
horizontal axis with the line $l_i$. 

\begin{remark}\label{rmk:estimate}
	Since the upper boundary of the polygon $\Delta_{x}(B)$ is concave, we have 
	\[
	\Delta_x(B) \ \dsubseteq \ T_i, \text{ for all } 1\leq i\leq r \ . 
	\]
	An explicit computation using the definition will convince that the subsequent slopes of the functions $l_i$  are decreasing, therefore 
	one  has $T_i \dsupseteq T_{i+1}$ for all $1\leq i\leq r$ as well.
\end{remark}

\begin{proof}[Proof of Lemma~\ref{lem:estimate for many curves}]
	It follows from \cite{KL14}*{Remark 1.9} that 
	\[
	\length\Big(\Delta_{(E,z)}(\pi^*B)\cap\{t\}\times\RR\Big) \equ  (P_t\cdot E), \textup{ for all points } z\in E \ .
	\]
	The point  $x\in X$ was chosen to be very general, therefore one has  $\mult_{\overline{\Gamma}_i}(\|\pi^*(B)-t\og_i\|) \ \geq \ t-\epsilon_i$ 
	for all $1\leq i\leq r$ and all $t\geq \e_i$ by Nakamaye's Lemma \cite{KL14}*{Lemma 4.1} (see also \cite{NakVeryGen}*{Lemma 1.3}). As a consequence, 
	we obtain via Lemma~\ref{lem:Nakamaye and Zariski} that 
	\begin{eqnarray*}
		(P_t\cdot E) & = & \big( \big( (\pi^*B-tE) - \sum_{i=1}^{r}\mult_{\og_i}(\| \pi^*B-tE\|) \og_i \big) \cdot E\big) \\ 
		& \leq &  t - \sum_{j=1}^{i}(t-\e_j)m_j \equ (1-\sum_{j=1}^{i}m_j)t + \sum_{j=1}^{i}\e_jm_j
	\end{eqnarray*}
	for all $\e_i\leq t< 2$, as required.
\end{proof}

\begin{proof}[Proof of Theorem~\ref{thm:verygeneric}]
	We will subdivide the proof  into  several cases   depending on the size and nature  of the Seshadri constant $\epsilon\deq \epsilon(L;x)$.
	
	\smallskip\noindent
	\textit{Case 1:} $\epsilon(L;x)\geq 2(p+2)$. Observe that \cite{KL14}*{Theorem D} yields
	\[
	\length \big(\Delta_{(E,z)}(\pi^*B)\cap \{2\}\times\RR\big) \ > \ 1 \ ,
	\]
	since $\epsilon_1=\epsilon(B;x)=\frac{1}{p+2}\epsilon\geq 2$,  and  $\Delta_{(E,z)}(\pi^*B) \supseteq \Delta^{-1}_{\epsilon(B;x)} \supseteq \Delta^{-1}_2$.
	
	\smallskip\noindent
	\textit{Case 2:} $\epsilon(L;x)=\frac{(L.F)}{\textup{mult}_x(F)}$ for some irreducible curve $F\subseteq X$ with 
	$r\deq (L.F)\geq q\deq\mult_x(F)\geq 2$ (note that $r\geq q$ follows from the main result of \cite{EL}).
	
	The point  $x\in X$ was chosen to be  very general, therefore  Proposition~\ref{prop:genericinf} implies  $\Delta_x(B)\subseteq \triangle OMN$, where 
	$O=(0,0),M=(\frac{r}{(p+2)q},\frac{r}{(p+2)q)})$ and $N=(\frac{r}{(p+2)(q-1)},0)$. Since the area of the first polygon is $B^2/2\geq 5/2$, we obtain the inequality 
	\[
	\frac{1}{(p+2)^2}\cdot \frac{r^2}{q(q-1)} \dgeq 5\ .
	\]
	Remembering that $q\geq 2$, we arrive at  the following lower bound on the Seshadri constant
	\begin{equation}\label{eqn:6}
	\epsilon(B;x) \equ \frac{r}{(p+2)q} \dgeq \sqrt{5\zj{1-\frac{1}{q}}} \dgeq \sqrt{\frac{5}{2}}\ ,
	\end{equation}
	Next, fix a point $z_0\in E$, and  let $S=(\epsilon_1,\epsilon_1)$, $T=(\epsilon_1,0)$, $S'=(2,2)$ and $T'=(2,0)$. The triangle $\triangle OST$ is the largest inverted standard simplex  inside  $\Delta_{(E,z_0)}(\pi^*B)$. Let 
	\[
	[AD] \deq \Delta_{(E,z_0)}(\pi^*B)\cap \{2\}\times\RR  \ ,
	\]
	and aiming at a contradiction, suppose  $||AD||\leq 1$ (for a visual guide see Figure~\ref{fig:2}).
	
	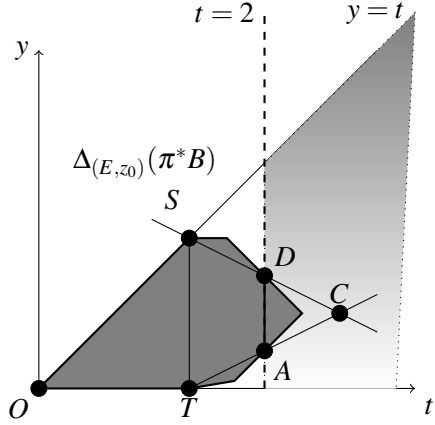
\begin{figure}
		\begin{tikzpicture}
		%
		%
		\draw [->]  (0,0) -- (5,0);   
		\node [below right] at (5,0) {$t$};
		\draw [->]  (0,0) -- (0,4.5); 
		\node [left] at (0,4.5) {$y$}; 
		\draw (0,0) -- (4.5,4.5); 
		\node [left] at (5,5) {$y=t$};
		\shadedraw [dotted,shading=axis] (3,0) -- (3,3) -- (5,5) -- (4.75,0);
		\draw [fill=gray, thick] (0,0) -- (2,2) -- (2.5,2) -- (3,1.5) -- (3.5,1) -- (3,0.5) -- (2.6,0.1) -- (2,0) -- (0,0);  
		\node [left] at (2.5,3) {$\Delta_{(E,z_0)}(\pi^*B)$}; 
		\draw [thick, dashed] (3,0) -- (3,5);  
		\node [left] at (3,5) {$t=2$}; 
		\draw (2,0) -- (2,2); 
		\draw [thick] (3,0.5) -- (3,1.5);  
		\draw (1.5,2.25) -- (4.5,0.75); 
		\draw (2,0) -- (4.5, 1.25); 
		%
		%
		\node [below left] at (0,0) {$O$}; 
		\filldraw[black]  (0,0) circle (3pt);
		\node [below] at (2,0) {$T$}; 
		\filldraw[black] (2,0) circle (3pt);
		\node [above left] at (2,2.25)  {$S$}; 
		\filldraw[black] (2,2) circle (3pt); 
		\node [above right] at (3,1.5) {$D$}; 
		\filldraw [black]  (3,1.5) circle (3pt); 
		\node [below  right] at (3,0.5) {$A$}; 
		\filldraw[black]  (3,0.5) circle (3pt);
		\node [above] at (4,1) {$C$}; 
		\filldraw[black] (4,1) circle (3pt); 
		\end{tikzpicture}
		\caption{$\Delta_{(E,z_0)}(\pi^*B)$ and the triangles $\triangle STC$ and $\triangle ADC$ \label{fig:2}} 
	\end{figure}  
	
	Note first that by $(\ref{eqn:6})$, one has $\epsilon_1>1$ and thus $||ST||>1$. Together with the assumption  $||AD||\leq 1$,  this implies that the lines $TA$ and $SD$ 
	intersect to the right of  the vertical line $t=2$, let us call the point of intersection  $C$. Denote by $x=\dist(C, t=2)$. Since both line segments $[TA]$ and $[SD]$ 
	are contained in $\Delta_{(E,z_0)}(\pi^*B)$, convexity yields the inclusion
	\[
	\triangle ADC \ \supseteq \ \Delta_{(E,z_0)}(\pi^*B)_{t\geq 2} \deq  \Delta_{(E,z_0)}(\pi^*B) \ \cap \ \{t\geq 2\}\times\RR \ .
	\]
	An area comparison yields the  string of inequalities
	\begin{align*}
	\Area(\triangle ADC) \ &\geq \ \Area(\Delta_{(E,z_0)}(\pi^*(B))_{t\geq 2} \equ   \Area(\Delta_{(E,z_0)}(\pi^*(B)) \ - \ \Area(\Delta_{(E,z_0)}(\pi^*(B))_{t\leq  2} \\
	& \geq \frac{\vol_X(B)}{2} -\Area(\triangle OS'T') \dgeq  \frac{5}{2}-\frac{4}{2} \equ  \frac{1}{2} \ .
	\end{align*}
	By  the similarity between $\triangle ADC$ and $\triangle TSC$, we see that  
	\[
	||AD|| \equ  \frac{\epsilon_1 x}{x+2-\epsilon_1}\ .
	\]
	Along  with the condition  $||AD||\leq 1$ this  implies
	\begin{equation}\label{eqn:7}
	\frac{x+2}{x+1} \ \geq \ \epsilon_1 \ .
	\end{equation}
	On the other hand, by the above we also have 
	\[
	\Area(\triangle ADC) \equ  \frac{\epsilon_1 x^2}{2(x+2-\epsilon_1)} \ \geq \ \frac{1}{2} \ ,
	\]
	which gives 
	\[
	\epsilon_1  \dgeq  \frac{x+2}{x^2+1} \ .
	\]
	Combining this inequality with $(\ref{eqn:7})$, we arrive at 
	\[
	\frac{x+2}{x+1} \dgeq  \frac{x+2}{x^2+1} \ ,
	\]
	forcing $x\geq 1$. Since the function $f(x)=\frac{x+2}{x+1}$ is decreasing for $x\geq 1$, the inequality $(\ref{eqn:7})$ implies that $\epsilon_1 \leq f(1)=3/2$, 
	contradicting   $\epsilon_1\geq \sqrt{5/2}$ from  $(\ref{eqn:6})$. Thus $||AD||>1$, as required. 
	
	\smallskip\noindent
	\textit{Case 3:} Here we assume that $\epsilon=\epsilon(L;x)\in \NN$ with $1\leq \epsilon\leq 2p+3$, and that there exists an irreducible curve $\Gamma_1\subseteq X$ with $\mult_x(\Gamma_1)=1$ and $\epsilon(L;x)=(L\cdot \Gamma_1)=\epsilon$. Denote by $\og_1$ the proper transform via $\pi$ of $\Gamma_1$.
	
	We point out that when  $1\leq \epsilon\leq p+2$, our assumptions  reduces to the condition in the statement of our theorem with  $\Gamma_1$ playing the role of  $C$.
	Whence, in what follows  we can assume that in addition $p+3\leq \epsilon\leq 2p+3$. Secondly, since $\mult_x(\Gamma_1)=1$, $(L\cdot \Gamma_1)=\epsilon\leq 2p+3$,
	and $L^2\geq 5(p+2)^2$, the Hodge Index Theorem yields  $(\Gamma_1^2)\leq 0$.
	
	Since the point $x\in X$ is  very general, Proposition~\ref{prop:genericinf} (2.b) shows  that the slope of the linear function $t\longmapsto (P_t.E)$ is non-positive. 
	By $(\ref{eq:positivepart})$,  the only way  this can happen is  when $m_1=1$ and $(\Gamma_1^2)=0$ (since  $(\Gamma_1^2)\leq 0$).
	
	We consider two sub-cases.
	
	\smallskip\noindent
	\textit{Case 3(a):} $\og_1$  is the only curve appearing in the negative part of the divisor $\pi^*B-tE$ for any $t\in [\epsilon_1,2)$. 
	
	Note first that by above we have that $\mult_x(\Gamma_1)=1$ and thus $\og_1=\pi^*\Gamma_1-E$. Since the curve $\og_1$ is 
	the only curve appearing in the negative part of $\pi^*B-tE$ for any $t\in [1,2)$,  we can apply the algorithm for finding the negative part of the Zariski decomposition 
	for each divisor $\pi^*B-tE$ and deduce that
	\[
	\pi^*B-tE \equ  P_t \ + \ (t-\epsilon_1)\overline{\Gamma}_1
	\]
	is indeed the appropriate Zariski decomposition for  any $t\in [\epsilon_1,2)$. In particular, we obtain that  
	\begin{equation}\label{eq:positivepart}
	(P_t.E) \equ  \epsilon_1\ \ \ \text{for all $t\in [\epsilon_1,2)$.}
	\end{equation}
	
	Having positive self-intersection  $\pi^*B-tE$ is big for all $t\in [\epsilon_1,2]$, therefore the Zariski decomposition along this  line segment 
	is continuous by \cite{BKS04}*{Proposition 1.16}.  Accordingly,  $(\ref{eq:positivepart})$ yields 
	\[
	\length\big(\Delta_{(E,z)}(\pi^*B)\cap \{2\}\times\RR\big) \ \equ \ (P_2.E) \ \equ \ \epsilon_1 \ \equ \ \frac{\epsilon}{p+2} \ > \ 1 
	\]
	for  we assumed that $\epsilon\in\st{p+3,\ldots ,2p+3}$. 
	
	\smallskip\noindent
	\textit{Case: 3(b)} Assume  that the negative part of $\pi^*B-tE$ contains other  curve(s)  beside  $\overline{\Gamma}_1$ for some  $t\in [\epsilon_1,2)$. 
	\\
	
	Denote these other curves that appear in the negative part of the segment line $\pi^*B-tE$ for $t\in [\epsilon, 2)$ by $\og_2,\ldots ,\og_r$ for some $r\geq 2$. Our main tool is going to be  the generic infinitesimal Newton--Okounkov polygon $\Delta_{x}(B)$. An important property of $\Delta_x(B)$  is that the vertical line segment $\Delta_{x}(B)\cap \{t\}\times\RR$ starts on the $t$-axis for any $t\geq 0$ (see \cite{KL14}*{Theorem 3.1}).
	
	Aiming at a contradiction, suppose  that $\length(\Delta_{x}(B)\cap\{2\}\times\RR)\leq 1$. Note first that  this is equivalent to having the point $(2,1)$ outside the interior of the polygon $\Delta_{x}(B)$, as its lower boundary  sits on the $t$-axis.
	
	As  $x\in X$ was chosen to be a very general point, by Remark~\ref{rmk:estimate} we have $\Delta_{x}(B) \subseteq T_2$, where the latter polygon denotes the convex hull  
	of  the vertices  $O,A_1,A_2,F_2$, where   $O=(0,0)$ and $A_1=(\epsilon_1,\epsilon_1)$. We will focus on  the slope of the line $A_2F_2$. 
	If  $A_2=A_1$  then by Lemma~\ref{lem:estimate for many curves}, we know that 
	\[
	\text{slope of}\ A_2F_2 \dleq \text{slope of }\ \ell (t) \equ (1-m_1-m_2)t+\epsilon_1m_1+\epsilon_2m_2\ .
	\]
	In particular, the  slope of $A_2F_2$ is at most  $-1$, since $m_1=1$ and $m_2\geq 1$.
	
	In the non-degenerate case $A_1\neq A_2$, since we have $m_1=1$, then 
	\[
	\text{slope of }\ A_2F_2 \dleq \text{slope of }\ \ell (t) \equ -m_2t+\epsilon_1+\epsilon_2m_2 \dleq -1
	\]
	again. Note also that in this case, based on the proof of $Case~3(a)$, we know for sure that the segment $[A_1A_2]$ is actually an edge of the convex polygon $\Delta_{x}(B)$.
	
	Since  the upper boundary of the polygon $\Delta_x(B)$ is concave by \cite{KLM1}*{Theorem B}, then  all the supporting lines of the edges on this boundary, besides $[OA_1]$ and 
	$[A_1A_2]$, have slope at most  $-1$. However, we initially assumed  that $(2,1)\notin\intt\Delta_x(B)$, thus convexity yields  that the first edge of the polygon 
	$\Delta_{x}(B)$ intersecting  the region $[2,\infty)\times\RR$  will  do so at a point on the line segment
	$\{2\}\times[0,1]$. Furthermore, this edge and all the other edges of the upper boundary in the region $[2,\infty)\times \RR$ will have slope at most  $-1$. 
	
	Now, set  $A=(2,1), C=(3,0)$ and $D=(2,0)$. Since the line $AC$ has slope $-1$, then by what we said just above, we have the following inclusions due to convexity reasons
	\[
	\Delta_x(B)\cap [2,\infty)\times\RR  \dsubseteq  \triangle ADC\ .
	\]
	Write  $R=(2,2)$; since  $\Area(\Delta_x(B))\geq 5/2$, the above  inclusion implies 
	\[
	\Area(\triangle ADC) \equ  \frac{1}{2} \ > \ \Area(\Delta_x(B)) - \Area(\triangle ODR)  \geq \frac{5}{2}-2 \equ  \frac{1}{2}
	\]
	contradicting $(2,1)\notin\intt \Delta_{x}(B)$.  In particular, $\length (\Delta_x(B)\cap\{2\}\times\RR)>1$,  and we are done. 
\end{proof}

\subsection{Proof of Theorem~\ref{thm:np}}

In this subsection we  prove the direct implication of Theorem~\ref{thm:np}. As opposed to the case of projective normality  it is no longer clear whether 
finding an effective divisor $D$ with $\sJ(X;D)=\sI_o$ will suffice to verify $(N_p)$. We show however, that with a bit more work one can in fact control the multiplier 
ideal of the divisor found in Theorem~\ref{thm:NO to singular} over the whole abelian  surface $X$. Again, $\pi\colon X'\to X$ denotes the blow-up of $o$ with exceptional  divisor $E$. 

The main goal of this subsection is to prove the following theorem.

\begin{theorem}\label{thm:nopolygonabelian}
Let $X$ be an abelian surface and $B$ an ample $\QQ$-divisor on $X$. Suppose that
\begin{equation}\label{eq:nopolygonabelian}
\textup{length}\big(\Delta_{(E,z_0)}(\pi^*(B)) \cap \{2\}\times\RR \big) \ > \ 1 \ ,
\end{equation}
for some point $z_0\in E$. Then there exists an effective $\QQ$-divisor $D\equiv (1-c)B$ for some $0<c<1$ such that $\sJ(X,D)=\sI_{X,o}$ over the whole of $X$.
\end{theorem}

\begin{proof}[Proof of Theorem~\ref{thm:np}, $(1)\Rightarrow (2)$]
By Theorem~\ref{thm:LPP} it suffices  to find a divisor as produced by Theorem~\ref{thm:nopolygonabelian}. Since $X$ is abelian, it is enough to treat the case 
when the origin $o$ behaves like a very general point. By Theorem~\ref{thm:verygeneric} the  condition $(\ref{eq:nopolygonabelian})$ is automatically satisfied whenever 
$X$ does not contain an elliptic curve $C$ with $(C^2)=0$ and $1\leq (L\cdot C) \leq p+2$.

It remains to show that the exceptions in Theorem~\ref{thm:verygeneric} correspond to  the exceptions 
in the statement in Theorem~\ref{thm:np}. For a curve $C\subseteq X$ which is smooth at the point $x$ and satisfies $1\leq (L\cdot C)\leq p+2$ one has  $(C^2)\leq 0$ by the Hodge index theorem  since $(L^2)\geq 5(p+2)^2$. Since we are on an abelian surface,  then automatically we have that $(C^2)=0$ and by adjunction this indeed forces $C$ to be an elliptic curve. 
\end{proof}

\begin{proof}[Proof of Theorem~\ref{thm:nopolygonabelian}]
To start with, \cite{KL14}*{Proposition 3.1} gives the inclusion
\[
\Delta_{(E,z)}(\pi^*(B)) \cap \{2\}\times\RR  \dsubseteq  \{2\}\times [0,2], \textup{ for any } z\in E \ .
\]
Hence, according to  \cite{KL14}*{Remark~1.9}, condition $(\ref{eq:nopolygonabelian})$ implies   $(\ref{eq:nopolygon1})$ in  Theorem~\ref{thm:NO to singular} for any $z\in E$. 

By Theorem~\ref{thm:NO to singular}  we know how to find a divisor $D\equiv B$ so that $\sJ(X,D)=\sI_{X,0}$ locally around a point. It remains to show that this equality 
in fact holds over the whole of $X$. 

Recall that $D$ is the image of a divisor $D'$ on $X'$, where (revisiting the proof of Theorem~\ref{thm:NO to singular}) 
\[
D' \ \equiv \ P+ \sum_{i=1}^{i=r}a_iE_i'+ 2E \ , 
\]
and  $\pi^*(B)-2E = P+\sum a_iE_i'$ is the appropriate Zariski  decomposition. 

Writing  $E_i$ for the image of $E_i'$ under  $\pi$,  the first step in the proof is to show the following claim:\\

\noindent
\textit{Claim:} Assume that  $o\neq y\in X$ such that  $\sJ(X,D)_y\neq \sO_{X,y}$. Then 
\begin{equation}\label{eq:claim}
\sum_{i=1}^{i=r}a_i\mult_o(E_i) \ < \ \sum_{i=1}^{i=r} a_i\mult_y(E_i)\ .
\end{equation}
\textit{Proof of Claim.} 
First observe that by \cite{PAGII}*{Proposition~9.5.13}, the condition $\sJ(X,D)_y\neq \sO_{X,y}$ implies  $\mult_y(D)\geq 1$. Since the morphism $\pi$ is an isomorphism 
around the point $y$,  considering $y$ as a point on  $X'$, we actually have  $\mult_y(D')\geq 1$. In the proof of Theorem~\ref{thm:NO to singular} we were able to write 
$P=A_k+\frac{1}{k}N$, where $A_k$ is ample and $N'$ is an effective $\QQ$-divisor  for any $k\gg 0$. 

Thus, we chose $D'=P'+\sum a_iE_i'+2E$, where $P'=A+\frac{1}{k}N'$, $A\equiv A_k$ is a generic choice, and $k\gg 0$. Since $A_k$ is ample, 
a generic choice of $A$ does not pass through the point $y$, in particular $\mult_y(P')\rightarrow 0$ as  $k\rightarrow \infty$. Since  $\mult_y(D')\geq 1$, 
this implies that
\[
\mult_y(\sum_ia_iE_i) \equ  \sum_ia_i\mult_y(E_i) \dgeq  1 \ . 
\]
As a consequence, it suffices to check that 
\[
\sum_i a_i\mult_o(E_i) \ < \ 1 \ .
\]
To this  end recall  that $E_i'=\pi^*E_i-\mult_o(E_i)E$, and therefore
\[
2 \equ  ((\pi^*B-2E)\cdot E) \equ  (P\cdot E) + \sum_ia_i(E_i'\cdot E) \equ  (P\cdot E)+\sum_ia_i\mult_o(E_i) \ .
\]
Because $(P\cdot E)$ is equal to the length of the vertical segment $\Delta_{(E,z)}(\pi^*B)\cap \{2\}\times\RR$ for any $z\in E$, 
$(\ref{eq:nopolygonabelian})$  implies that $\sum_ia_i\mult_o(E_i)<1$ as we wanted.\\
\textit{End of Proof of the Claim} \\

\noindent
We return to the proof of the main statement. We  will argue by contradiction and suppose that  there exists a point $y\in Y$ for which $\sJ(X;D)_x\neq \sO_{X,y}$. 
In other words, by the claim above, we have the inequality
\[
 \sum_{i}a_i\mult_o (E_i) < \sum_{i}a_i\mult_y (E_i) \ .
\]
This yields the existence of a curve $E_i$ with  $\mult_o(E_i) < \mult_y(E_i)$. 

By our assumptions $E_i'$ is a negative curve on $X'$; let $E_i^y$ denote  the proper transform  of $E_i$ with respect to the blow-up $\pi_y\colon X_y\to X$. 
Since $\mult_o(E_i)<\mult_y(E_i)$, we obtain 
\[
(E_i^y)^2 \equ  E_i^2-(\mult_y(E_y))^2 \ < \ E_i^2-(\mult_o(E_i))^2 \equ  E_i'^2\ <\ 0\ ,
\]
in particular, we deduce that $E_i^y$ remains a negative curve  on $X_y$, just  as $E_i'$ on $X'$. However, Lemma~\ref{lem:two points} shows  that $E_i$ must then be a 
smooth elliptic curve passing through $o$  and $y$. Therefore, $\mult_oE_i=\mult_yE_i=1$, which contradicts the  inequality above. 
\end{proof}

\begin{lemma}\label{lem:two points}
Let $X$ be an abelian surface,  $C\subseteq X$ a curve passing through two distinct points $x_1, x_2\in X$. If the proper transforms of $C$ for the respective  blow-ups 
of $X$ at the $x_i$ both become  negative curves,  then $C$ must be the smooth elliptic curve that is invariant under the translation maps $T_{x_1-x_2}$ or  $T_{x_2-x_1}$. 
\end{lemma}

\begin{proof}
Denote by $C_1$ and $C_2$  the proper transforms of $C$ with respect to  the blow-up of $X$ at $x_1$ and $x_2$, respectively. Aiming at a contradiction suppose that 
$T_{x_2-x_1}(C)\neq C$. Since  both proper transforms $C_1$ and $C_2$ are negative curves, one has 
\[
 0 > (C_1^2) \equ (C^2) - (\mult_{x_1}(C))^2\ \ \text{ and }\ \ 0 > (C_2^2) \equ(C^2)- (\mult_{x_2}(C))^2\ ,
\]
or,  equivalently $(C^2) <\min\{ \mult_{x_1}(C)^2 , \mult_{x_2}(C)^2\}$. On the other hand observe that 
\[
 (C^2) \equ  (C\cdot T_{x_2-x_1}(C)) \dgeq \mult_{x_2}(C)\cdot \mult_{x_2}(T_{x_2-x_1}(C))  \equ  \mult_{x_2}(C)\cdot \mult_{x_1}(C)\ ,
\]
for  $C$ is algebraically equivalent to its  translate $T_{x_2-x_1}(C)$. This is a contradiction, so  we conclude that $C$ is invariant under both of the translation maps $T_{x_1-x_2}$ or $T_{x_2-x_1}$, and is indeed an elliptic smooth curve. 
\end{proof}

  \newpage


\end{document}